\documentclass[12pt, reqno]{amsart}
\usepackage{amsmath, amsthm, amscd, amsfonts, amssymb, graphicx, color}
\usepackage[bookmarksnumbered, colorlinks, plainpages]{hyperref}

\newtheorem{theorem}{Theorem}[section]
\newtheorem{lemma}[theorem]{Lemma}

\theoremstyle{definition}
\newtheorem{definition}[theorem]{Definition}

\newtheorem{problem}[theorem]{Problem}
\theoremstyle{remark}

\numberwithin{equation}{section}

\begin{document}
\setcounter{page}{1}

\title[The product of a weak Asplund space  ]{The product of a weak Asplund space and a one-dimensional space is a weak Asplund space }

\author[S. Shang]{Shaoqiang Shang$^1$$^{*}$  }

\address{$^{1}$  College of mathematics and science,
Harbin Engineering University, Harbin 150001, P. R. China}
\email{\textcolor[rgb]{0.00,0.00,0.84}{sqshang@163.com    }}
\email{\textcolor[rgb]{0.00,0.00,0.84}{    functionequation@163.com        }}


\subjclass[2010]{Primary 46G05, 46B20, 26E15, 58C20.}

\keywords{ weak Asplund space, Banach-Mazur game, G$\mathrm{\hat{a}}$teaux
differentiable point, $G_{\delta}$-set, two-dimensional space}

\date{Received: xxxxxx; Revised: yyyyyy; Accepted: zzzzzz.
\newline \indent $^{*}$Corresponding author.}

\begin{abstract}

In this paper, we prove that if $X$ is a weak Asplund space, then $X\times\mathbb R$ is also a weak Asplund space. This resolves an open problem posed by D. G. Larman and R. R. Phelps in 1979 (J. London Math. Soc. (2) 20 (1979), 115-127).
The proof combines the Banach-Mazur game, properties of maximal monotone operators, and the G$\mathrm{\hat{a}}$teaux differentiability of Minkowski functionals. Specifically, we introduce an iterative perturbed Minkowski convex cone approximation game method, and establish a theoretical framework to verify the existence of dense differentiability sets for convex functions on product spaces. By relating the properties of convex functions on the base space and the product space via projection mappings, and constructing dense $G_\delta$ subsets using sequences of dense open cones, we demonstrate that the product space satisfies the definition of a weak Asplund space.

\end{abstract} \maketitle

\section{Introduction and preliminaries}

\par Let $(X,\|\cdot\|)$ be a real Banach space. $S(X)$ and $B(X)$ denote the unit sphere and the closed unit ball of $X$, respectively. We denote by $X^{*}$ the dual space of $X$. $B(x,r)$
stands for the closed ball centered at $x$ with radius $r$.

\par Let $D$ be a nonempty open convex subset of $X$, and let $f$ be a real-valued continuous convex function on $D$. We say that $f$ is G$\mathrm{\hat{a}}$teaux (Fr\'echet) differentiable
at a point $x\in D$ if there exists a functional $df(x)\in X^{*}$ such that
\[
\mathop {\lim }\limits_{t \to 0} \left| \frac{{f(x + ty) - f(x)}}{t}-\left\langle df(x), y\right\rangle \right| =0~~\quad~~\mathrm{for}~~~~~~~\mathrm{every}~~~~~~\quad~~y\in X.
\]
\[
\left( \mathop {\lim }\limits_{t \to 0} \sup_{y\in B(X)}\left| \frac{{f(x + ty) - f(x)}}{t}-\left\langle df(x), y\right\rangle \right| =0 \right)
\]

\par In 1968, E. Asplund extended Mazur's theorem in two directions. First, he proved that Mazur's theorem remains valid for a class of Banach spaces larger than the separable ones.
Second, he studied another class of Banach spaces for which a stronger conclusion holds: the G$\mathrm{\hat{a}}$teaux differentiability in Mazur's theorem can be replaced by  Fr\'echet differentiability.
These spaces are referred to as weak Asplund spaces and
Asplund spaces, respectively.

\begin{definition}(see [1])
A Banach space $X$ is said to be a weak Asplund (Asplund) space if for every continuous convex function $f$ on an open convex subset $O$ of $X$, there exists a dense $G_\delta$-subset $G$ of $O$ such that $f$ is G$\mathrm{\hat{a}}$teaux (Fr\'echet) differentiable at every point of $G$.
\end{definition}

\par It is well known that there exist weak Asplund spaces that are not Asplund. For example, $\ell^1$
 is weak Asplund but not Asplund.
Moreover, $X$ is Asplund if and only if $X^*$ has the Radon-Nikodym property (see [5]). In 1933, Mazur proved that every separable Banach space is weak Asplund (see [5]). In 1990, D. Preiss, R. Phelps and I. Namioka showed that every smooth Banach space is weak Asplund (see [12]). In 1997, M. J. Fabian proved that any quotient of a weak Asplund space is weak Asplund (see [6]).

\par Furthermore, Asplund spaces and weak Asplund spaces are fundamental to convex differentiability analysis (see [7]-[12]). Since the 1970s, Banach space theory has witnessed substantial progress, successively establishing that:
the product of two Asplund spaces is Asplund;
and every closed subspace of an Asplund space is Asplund.
It was conjectured that analogous
results might hold for weak Asplund spaces. However, despite
their earlier introduction, much less is known about weak Asplund spaces.
One major obstacle is that sets of G$\mathrm{\hat{a}}$teaux differentiability points need not be $G_\delta$, and the $G_\delta$ property is not preserved under continuous linear mappings.
To circumvent
this difficulty, D. G. Larman and R. R. Phelps introduced and studied G$\mathrm{\hat{a}}$teaux differentiability spaces, and posed several open problems in [1], which
have been widely discussed since 1979.

\begin{definition}(see [1])
A Banach space $X$ is said to be a G$\mathrm{\hat{a}}$teaux differentiability space if every continuous convex function on $X$ is G$\mathrm{\hat{a}}$teaux differentiable on a dense subset of $X$.
\end{definition}

It is well known that $X$ is a G$\mathrm{\hat{a}}$teaux differentiability space if and only if for every bounded weak$^*$ closed convex subset $C^*$ of $X^*$, the support functional $\sigma_{C^*}$ is G$\mathrm{\hat{a}}$teaux differentiable on a dense subset of $X$.

\begin{problem}
Must every G$\mathrm{\hat{a}}$teaux differentiability space $X$ be a weak Asplund space?
\end{problem}

\begin{problem}
Let $X$ be a weak Asplund space. Must $X\times\mathbb R$ be a weak Asplund space?
\end{problem}

\begin{problem}
Let $X$ be a weak Asplund space and let $M$ be a closed subspace of $X$. Must $M$ be a weak Asplund space?
\end{problem}

\par These three problems  are fundamental to the theory of weak Asplund spaces and closely related to its applications. Since Asplund spaces
are well known to have important applications in differential equations, variational analysis, and optimization theory, resolving these three problems will advance
 the theory of weak Asplund spaces and facilitate its applications to these fields.

\par We now state some known results concerning weak Asplund spaces and G$\mathrm{\hat{a}}$teaux  differentiability spaces. In the mid-1980s, M. Fabian proved via penalty function methods that if $X$ is a G$\mathrm{\hat{a}}$teaux differentiability space, then $X\times\mathbb R$ is also a G$\mathrm{\hat{a}}$teaux differentiability space. R. R. Phelps, a pioneer in convex analysis, described this as the only positive progress on G$\mathrm{\hat{a}}$teaux differentiability spaces (see [5]).

\par In 2001, Lixin Cheng and M. Fabian showed that the product of a G$\mathrm{\hat{a}}$teaux  differentiability space and a separable Banach space remains a G$\mathrm{\hat{a}}$teaux differentiability space (see [3]). In 2006, Warren B. Moors and Sivajah Somasundaram constructed a G$\mathrm{\hat{a}}$teaux differentiability space that is not weak Asplund (see [2]), thereby settling Problem 1.3.
The main objective of  this paper is to resolve Problem 1.4, which is of substantial theoretical and practical interest.

\par In this paper, we prove that if $X$ is a weak Asplund space, then $X\times\mathbb R$ is a weak Asplund space. Our proof relies on the iterative perturbed
Minkowski convex cone approximation game, the intuition of which is as follows.

\par
(1) Building the bridge.

\par By taking upper and lower slices of convex sets in the product space along
the real direction, we reduce the higher-dimensional differentiability problem to
the base space.
Subsequently, through radial enlargement, we construct open convex cones that establish a robust topological structure, enabling differentiability information to be
transferred from the base space to the product space.

\par (2) Layer-by-layer screening.

\par Instead of directly locating differentiability points, we construct a sequence of successively refined geometric Minkowski functionals,  introducing a small perturbation at each step.
Using convex-hull geometric analysis on two-dimensional planes, we iteratively eliminate non-optimal directions on the unit sphere, retaining only those near the optimum.
This procedure resembles filtering with progressively finer sieves, yielding a unique limiting direction.

\par (3) Game-theoretic closure.

\par We embed the foregoing geometric screening procedure into a topological game framework, namely the Banach-Mazur game. In each round,
further screening is performed within a
smaller open ball.
After countably many rounds, the nested balls shrink to a single point.
We prove that this limit point is a differentiability point of the convex function,
and that such differentiability points form a $G_{\delta}$-dense set,
which completes  the proof.

\begin{definition}(see [5])
Let $f$ be a continuous convex function on an open convex subset $O$ of $X$. The set $\partial f(x)$ is called the subdifferential of $f$ at $x\in O$, where
\[
\partial f(x)=\bigl\{x^{*}\in X^{*}:\langle x^{*},y-x\rangle\leq f(y)-f(x)\text{ for all }y\in O\bigr\}.
\]
\end{definition}

\begin{definition}(see [12])
A set-valued mapping $T\colon X\to 2^{X^{*}}$ is called a maximal monotone operator provided
$
\langle x^{*}-y^{*},x-y\rangle\ge 0
$
for all $x,y\in X$, $x^{*}\in T(x)$ and $y^{*}\in T(y)$.
\end{definition}
Let $(eT)(x)= \{\langle  x^{*}  ,e   \rangle :  x^{*} \in T\left(x\right) \}$ and
$\sigma_{T}\left(x,e\right)=\sup\left\{    \left\langle x^{*},e \right\rangle   :   x^{*} \in T\left(x\right)       \right \}$.
In particular, if $T=\partial f$, then
\[
e(\partial f)(x)= \{\langle  x^{*}  ,e   \rangle :  x^{*} \in \partial f\left(x\right) \}~~\quad~~~~~~~~\mathrm{and}~~~~~~~~~~~\quad~~~~\sigma_{\partial f}\left(x,e\right)=\sup\left\{    \left\langle x^{*},e \right\rangle   :   x^{*} \in \partial f\left(x\right)       \right \}.
\]

\begin{lemma}(see [12])
Let $T\colon X\to 2^{X^{*}}$ be a maximal monotone operator, and suppose that $D=\operatorname{int} D(T)=\operatorname{int}\left\{x\in X: T(x)\neq\emptyset\right\}$ is nonempty. Then

\par (1)  for every $x\in D$, $f_{x,T}(e)=\sigma_T(x,e)$ is subadditive and positively homogeneous, and for all $\lambda>0$, $\sigma_{\lambda T}\left(x,e\right)=\lambda\sigma_T\left(x,e\right)$;
\par (2)  for every $x\in D(T)$,
\[
\sup\bigl\{\sigma_T(x,e):\|e\|=1\bigr\}=\sup\bigl\{\sigma_T(x,e):\|e\|\le 1\bigr\}=\sup\bigl\{\|x^{*}\|:x^{*}\in T(x)\bigr\};
\]
\par (3)  the set $(eT)(x)$ is a singleton if and only if $\sigma_T\left(x,-e\right)=-\sigma_T\left(x,e\right)$;
\par (4)  if $x_0\in D$, $e\in X$ and the set $(eT)(x_0)$ is a singleton, then $f_{e,T}(x)=\sigma_T(x,e)$ is continuous at $x_0$;
\par (5)  let $x\in D$ and $e\neq 0$. Let $I=\{t\in\mathbb R: x+te\in D\}$ and define
\[
f(t)=\sigma_T\big(x+te,e\big),\quad t\in I.
\]
Then $f$ is nondecreasing on $I$. Furthermore, if $f$ is continuous at $t_0\in I$, then  $(eT)(x+t_0 e)$ is a singleton.

\end{lemma}

\begin{definition} (see [12])
 Let $X$ be a Hausdorff space and $S$ be a subset of $X$. Let $A$ and $B$ denote the players of the game. A play is a decreasing sequence of nonempty open sets $U_{1}\supset{V_{1}}\supset{U_{2}}\supset{V_{2}}\supset{\ldots}$ which have been chose alternately; the $U_{k}'s$ by $A$, the $V_{k}'s$ by $B$. Player $B$ is said to have won a play if $\cap_{n=1}^{\infty}{V_{n}\subset {S}}$; Otherwise, we say that  Player $A$ have won.  (It is not required that the intersection be nonempty.) We say that Player $B$ has a winning strategy if, using it, Player $B$ wins every play, independently of player $A's$ choices. (A more detailed description of the Banach-Mazur game  may be found in [13].)
\end{definition}

\begin{lemma}(see [12])
If player $B$ has a winning strategy, then $S$ is a residual set (that is, $X\setminus S$ is of first category). In particular, if $X$ is a completely metrizable space, then $S$ contains a dense $G_\delta$ subset.
\end{lemma}

\section{ Main Theorems }

\begin{theorem}
Suppose that $X$ is a weak Asplund space. Then $X\times\mathbb R$ is a weak Asplund space.
\end{theorem}

To prove this theorem, we first establish several lemmas.

\begin{lemma}
Let $C$ be a bounded open convex subset of $X\times\mathbb R$ with $(0,0)\in C$. Let
\[
f(x)=\inf\bigl\{r\in\mathbb R:(x,r)\in\overline{C}\bigr\},
\]
and let $T\colon X\times\mathbb R\to X$ be the projection defined by $T(x,r)=x$. Then

\par (1) the function $f$ is continuous and convex on $TC$;
\par (2)  if $x_0\in TC$, then $f$ is G$\hat{a}$teaux  differentiable at $x_0$ if and only if $\sigma_{C^*}$ is G$\hat{a}$teaux  differentiable at $(x_0,f(x_0))$;
\par (3)  if $f$ is G$\hat{a}$teaux  differentiable on $A_T$, where $A_T$ is a $G_\delta$ subset of $TC$, then there exists a sequence of open sets $\{O_n\}_{n=1}^{\infty}$ with $O_n\subset G$ and $A_T\subset TO_n$ such that each $O_n$ is a cone and $\sigma_{C^*}$ is G$\hat{a}$teaux  differentiable on $\cap_{n=1}^{\infty}O_n$, where
\[
C^*=\big\{ (x^{*},s)\in X^{*}\times\mathbb R: \left\langle  (x^{*},s), (x,r)              \right\rangle \leq1, ~  (x,r)\in C                  \big\}
\]
and
\[
G=\bigl\{(\lambda x,\lambda f(x))\in X\times\mathbb R:x\in TC,\ \lambda\in(0,+\infty)\bigr\}.
\]

\end{lemma}

\begin{proof}

(1) Since $C$ is a bounded open convex subset of $X\times\mathbb R$ and $(0,0)\in\operatorname{int}C$, it is easy to verify that
\[
f(x)=\inf\big\{r\in\mathbb R:(x,r)\in\overline{C}\big \}
\]
is convex on $TC$. We define the epigraph of the convex function $f$, where
\[
\operatorname{epi} f=\bigl\{(x,r)\in X\times\mathbb R:x\in T\left(\overline{C}\right),\ f(x)\le r\bigr\}.
\]
This implies that $\operatorname{epi} f$ is a convex subset of $X\times\mathbb R$.
We next prove that $\operatorname{epi} f$ is closed in $X\times\mathbb R$. Indeed, take any $(x_0,r_0)\in\overline{\operatorname{epi} f}$.
Then there exists a sequence $\{(x_n,r_n)\}_{n=1}^{\infty}$ such that
\[
\left( {x_n,{r_n }} \right) \in \mathrm{epi}f~\quad~\mathrm{and}~\quad~\mathop {\lim } \limits_{n \to \infty}\left\|\left( {x_n,{r_n }} \right)- \left( {x_0 ,{r_0}} \right)\right\|=0.  \eqno  (2.0)
\]
Then we have $r_n - f(x_n)\ge 0$ for every $n\in\mathbb N$. Hence, if $\liminf\limits_{n\to\infty}\bigl[r_n - f(x_n)\bigr]=0$, then, by (2.0), we may assume without loss of generality that
\[
\mathop {\lim } \limits_{n \to \infty}\left\|\left( {x_n,f(x_n)} \right) - \left( {x_0,{r_0}} \right)\right\| =    \mathop {\lim } \limits_{n \to \infty}\left\|\left( {x_n,{r_n }} \right)- \left( {x_0 ,{r_0}} \right)\right\|             =0.
\]
Therefore, since $(x_n,f(x_n))\in\overline{C}\subset\operatorname{epi} f$, we conclude that $(x_0,r_0)\in\overline{C}\subset\operatorname{epi} f$. Moreover, if
\[
\liminf_{n\to\infty}\bigl[r_n - f(x_n)\bigr]=m>0,
\]
then we may assume without loss of generality that
\[
\lim_{n\to\infty}\bigl[r_n - f(x_n)\bigr]=m>0.
\]
Hence there exists a sequence $\{\varepsilon_n\}_{n=1}^{\infty}\subset\mathbb R$ such that
$f(x_n)=r_n - m+\varepsilon_n$ for all $n\in\mathbb N$, with $\varepsilon_n\to 0$ as $n\to\infty$. Since $(x_n,f(x_n))\in\overline{C}$, we conclude that
\[
\bigl(x_n,r_n - m+\varepsilon_n\bigr)=(x_n,f(x_n))\in\overline{C}.
\]
Therefore, by (2.0) and the closedness of $\overline{C}$, we conclude that
\[
\mathop{\lim } \limits_{n \to \infty}\left\|\left( x_n,{r_n } - m +\varepsilon_{n}\right)-\left( x_0,r_0 - m \right)\right\|=0 ~\quad~\mathrm{and}~\quad~\left( x_0,r_0 - m \right)\in \overline{C} \subset \mathrm{epi}f.
\]
Hence we obtain that $(x_0,r_0)\in\operatorname{epi} f$. This shows that $\operatorname{epi} f$ is closed in $X\times\mathbb R$. It follows that $f$ is a continuous convex function on $TC$.

\par (2) Necessity. Pick a point $x_0\in TC$ and a functional $(y_0^{*},r_0)\in\partial\sigma_{C^{*}}(x_0,f(x_0))$. Then, for every $(x,f(x))\in\overline{C}$, we conclude that
\[
1=\left\langle ({y_{0}^{*}},r_{0}), (x_{0},f(x_{0})) \right\rangle \geq \left\langle ({y_{0}^{*}},r_{0}), (x,f(x)) \right\rangle.  ~~~~~~~~~~~~~~~~~~~~~~~~~\eqno ~~~~~~~~~~~~~~~~~ (2.1)
\]
We claim that $r_0\le 0$. Indeed, there exists a real number $\eta\in(0,+\infty)$ such that $(x_0,f(x_0)+\eta)\in C$.
Since $(y_0^{*},r_0)\in\partial\sigma_{C^{*}}(x_0,f(x_0))$, we obtain that
\[
0\leq \left\langle ({y_{0}^{*}},r_{0}), (x_{0},f(x_{0})) \right\rangle - \left\langle ({y_{0}^{*}},r_{0}), (x_{0},f(x_{0})+\eta)\right\rangle=- r_{0}\eta.
\]
Hence we obtain $r_0\le 0$. We further claim that $r_0<0$. Indeed, suppose $r_0=0$. Then, by (2.1), it is easy to see that $y_0^{*}(x_0)=1>0$. Since $C$ is a bounded open convex subset of $X\times\mathbb R$ and $(0,0)\in C$, $TC$ is open. Hence there exists a point $(z_0,r)\in X\times\mathbb R$ such that $x_0+z_0\in TC$ and $y_0^{*}(z_0)>0$. Therefore, by  (2.1) together with $r_0=0$, we obtain that
\[
0\leq \left\langle y_{0}^{*},x_{0}\right\rangle - \left\langle y_{0}^{*}, x_{0}+z_{0} \right\rangle = \left\langle y_{0}^{*}, -z_{0}\right\rangle<0,
\]
a contradiction. Hence we obtain that $r_0<0$.
Therefore, by  (2.1) together with $r_0<0$, we have the following inequality
\[
\left\langle -\frac{1}{r_0} y_0^{*},~x-x_0  \right\rangle \leq f(x)-f(x_0)~~~~~~~\quad~~~~~~\mathrm{for}~~~~\mathrm{every}~~~~~~~\quad~~~~ x\in TC.
\]
Pick a point $x_{0}^{*}\in \partial f(x_{0})$. Since $f$ is G$\mathrm{\hat{a}}$teaux  differentiable at $x_0$, we obtain that $y_0^{*}= -r_0x_{0}^{*}$.
Moreover, by (2.1), we conclude that $r_0=-1/[{x_{0}^{*}}(x_{0}) - f(x_{0})]$. Since $y_0^{*}= -r_0x_{0}^{*}$, we conclude that
$y_{0}^{*}= x_{0}^{*}/ \left[{x_{0}^{*}}(x_{0}) - f(x_{0})\right]$. Hence $\sigma_{C^{*}}$ is G$\mathrm{\hat{a}}$teaux differentiable at $(x_0,f(x_0))\in X \times \mathbb R$.

\par Sufficiency. We pick a point $x_{0}\in T C$ and  a functional $x_{0}^{*}\in \partial f\left(x_{0}\right)$. Then, for every $(z,r)\in C$, we obtain that $r\geq f(z)$. This implies that
\[
x_{0}^{*}\left(z\right)-r\leq x_{0}^{*}\left(z\right)-f\left(z\right)\leq x_{0}^{*}\left(x_{0}\right)-f\left(x_{0}\right).
\]
Hence $\langle({y_{0}^{*}},r_{0}),(x_{0},f(x_{0}))\rangle =1$ and $({y_{0}^{*}},r_0)\in C^{*}$, where $y_{0}^{*}=x_{0}^{*}/({x_{0}^{*}}(x_{0}) - f(x_{0}))$ and $r_{0}=-1/({x_{0}^{*}}(x_{0}) - f(x_{0}))$.
This implies that $({y_{0}^{*}},r_{0})\in \partial \sigma_{C^{*}}(x_{0},f(x_{0}))$. Hence $x_0^{*}=-r_0y_{0}^{*}$.
Since $\sigma_{C^{*}}$ is G$\mathrm{\hat{a}}$teaux differentiable at $(x_0,f(x_0))$, we obtain that
 $f$ is G$\mathrm{\hat{a}}$teaux  differentiable at  $x_0\in X$.

\par (3) Since $A_T$ is a $G_{\delta}$-subset of set $TC$, there exists an open set sequence $\{G_n\}_{n=1}^{\infty}$ with $G_n\subset TC$ such that
$A_T=\cap_{n=1}^\infty G_n$. Hence we
define the set
\[
W_n= \big\{ (x,f(x))\in X \times\mathbb R: x\in G_n          \big\}~\quad~~\mathrm{for}~~~~\mathrm{every}~\quad~~n\in N.
\]
Define the set $O_n=\cup_{\lambda>0}\lambda W_n$ for every $n\in\mathbb N$. We claim that the set $O_n$ is open for every $n\in\mathbb N$.
In fact, we fix a natural number $n$ and pick a point $(z,r)\in O_n$.
Then there exist a real number $\lambda_0>0$ and a point $(z_0,f(z_0))\in W_n$ such that
\[
(z,r)=\lambda_0\cdot\big(z_0,f\left(z_0\right)\big)=\big(\lambda_0z_0,\lambda_0 f\left(z_0\right)\big)\in O_n.~~~~~~~~~~~~~~~~~~~~~\eqno ~~~~~~~~~~~~~~~~~ (2.2)
\]
Since the set $G_n$ is an open subset of $TC$, there exists a real number $d\in (0,+\infty)$
such that $B\left(z_0,d\right)\subset G_n$ and $z_0\in TC$.
Since $C$ is a bounded open convex subset of $X\times\mathbb R$, by the open mapping theorem, we obtain that $TC$ is open.
Moreover, since $z_0\in TC$, there exists a sufficiently small open neighborhood $V$ of $\left(z_0,f\left(z_0\right)\right)$ with $(0,0) \notin V$ such that
\[
TV\subset B\left(z_0,d\right)\subset G_n,~\quad~\quad~~\frac{1}{\sigma_{C^{*}}\left(u,v\right)}u\in TC~\quad~~\mathrm{whenever}~\quad~~(u,v)\in V
\]
and
\[
\frac{1}{\sigma_{C^{*}}\left(u,v\right)}v= f\left(\frac{1}{\sigma_{C^{*}}\left(u,v\right)}u\right) ~\quad~~\mathrm{whenever}~\quad~~(u,v)\in V.
\]
Since $V$ is a sufficiently small open neighborhood in $X\times\mathbb R$
and $TC$ is open, from the preceding proof and the identity
$W_n=\bigl\{(x,f(x))\in X\times\mathbb R:x\in G_n\bigr\}$, we may assume without loss of generality that
\[
\left( \frac{1}{\sigma_{C^{*}}(u,v)}u,~\frac{1}{\sigma_{C^{*}}(u,v)}v \right) =\left( \frac{1}{\sigma_{C^{*}}(u,v)}u,~f\left(\frac{1}{\sigma_{C^{*}}(u,v)}u\right) \right)  \in W_n
\]
whenever $(u,v)\in V$.
Therefore, since $ O_n=  {\cup }_{\lambda>0 }\lambda W_n $, we obtain that $(u,v)\in O_n $. Hence we obtain that $(z_0,f(z_0))\in \mathrm{int}O_n $.
Therefore, since $ O_n=  {\cup }_{\lambda>0 }\lambda W_n $ and $(z,r) = \lambda_0\left(z_0,f(z_0)\right)$, we conclude that  $(z,r)\in \mathrm{int}O_n $.
This shows that  $ O_n$ is an open set. Moreover, since $ O_n=  {\cup }_{\lambda>0 }\lambda W_n $,  we conclude that $ O_n$ is a cone.

\par Since $f$ is G$\mathrm{\hat{a}}$teaux differentiable on $A_T$ and $W_n= \{ (x,f(x))\in X \times \mathbb R: x\in G_n \}$, by condition (2) and the fact that $O_n=\cup_{\lambda>0 }\lambda W_n $, we obtain that
$\sigma_{C^{*}}$ is G$\mathrm{\hat{a}}$teaux differentiable on
$\cap_{n=1}^{\infty}O_n$.
Hence condition (3) holds, which completes the proof.
\end{proof}

\begin{lemma}
Let $C$ be a bounded open convex subset of $X\times\mathbb R$ with $(0,0)\in C$. Let
\[
g\left(x\right)=\sup\bigl\{r\in\mathbb R:(x,r)\in\overline{C}\bigr\},
\]
and let $T\colon X\times\mathbb R\to X$ be the projection satisfying $T(x,r)=x$. Then

\par (1) the function $-g$ is continuous and convex on $TC$;
\par (2) if $x_0\in TC$, then $-g$ is G\^ateaux differentiable at $x_0$ if and only if $\sigma_{C^*}$ is G\^ateaux differentiable at $(x_0,g(x_0))$;
\par (3) if $-g$ is G\^ateaux differentiable on $A_T$, where $A_T$ is a $G_\delta$-subset of $TC$, then there exists a sequence of open sets $\{O_n\}_{n=1}^{\infty}$ with $O_n\subset G$ and $A_T\subset TO_n$ such that each $O_n$ is a cone and $\sigma_{C^*}$ is G\^ateaux differentiable on $\bigcap_{n=1}^{\infty}O_n$, where
\[
G=\bigl\{(\lambda x,\lambda g(x))\in X\times\mathbb R:x\in TC,\ \lambda\in(0,+\infty)\bigr\}.
\]
\end{lemma}

\begin{proof}
Similar to the proof of Lemma 2.2, we obtain that Lemma 2.3 holds true, which completes the proof.
\end{proof}

\begin{lemma}
Let $C$ be a closed convex subset of $X$ and $x_0\not\in C$.  Then
\[
\overline{\mathrm{co}}\left(\{x_0\}\cup C\right)= \{\lambda x_0+ (1-\lambda)x\in X: \lambda\in [0,1],~ x\in C\}.
\]
Moreover, if the space $H$ is a two-dimensional subspace of $X$ with $x_0\in H\backslash C$,
$D=\overline{\mathrm{co}}\left(\{x_0\}\cup C\right)$ and $0\in \mathrm{int} C$. Then there exists a point $y_0\in D\cap H$ such that
\[
\big[x_0,y_0\big]= \{x\in X: \mu_{D}\left(x\right)=1\}\cap \{ (1-\lambda)x_0+\lambda y_0\in X:  \lambda\in\mathbb R       \}.
\]
Further, if $C_1$ is a closed convex subset of C, $D_1=\overline{\mathrm{co}}\left(\{x_0\}\cup C_1\right)$, $x_0\in H\setminus C$ and $0\in \mathrm{int}\,C_1$, then there exists a point $z_0\in D_1\cap H$ such that
\[
\mu_{D_1}\big[\left(1-\lambda\right)x_0+\lambda z_0\big]\equiv 1,\quad \lambda\in [0,1],
\]
and there exists a point $u_0\in [0,y_0]$ satisfying $u_0\in \{(1-\lambda)x_0+\lambda z_0:\lambda\ge 0\}$.

\end{lemma}

\begin{proof}
Define the set $A= \{\lambda x_0+ (1-\lambda)x: \lambda\in [0,1],~ x\in C\}$. Then we obtain that $A\subset \overline{\mathrm{co}}(\{x_0\}\cup C)$.
Take two points $y_1\in A$ and $y_2\in A$. Then there exist $x_1\in C$ and $x_2\in C$ such that
$y_1=\lambda_1 x_0+ \left(1-\lambda_1\right)x_1$ and $y_2=\lambda_2 x_0+ \left(1-\lambda_2\right)x_2$, with $\lambda_1,\lambda_2\in [0,1]$. It
follows that for every $t\in [0,1]$, we have
\[
\quad \quad \quad t y_1~+~ (1-t)y_2\quad\quad\quad\quad \quad \quad \quad  \quad\quad \quad\quad \quad\quad \quad\quad \quad\quad \quad\quad\quad\quad \quad\quad \quad\quad \quad\quad \quad\quad \quad\quad \quad\quad
\]
\[
\quad =~t \left(\lambda_1 x_0+ (1-\lambda_1)x_1\right)~+~(1-t)\left(\lambda_2 x_0+ (1-\lambda_2)x_2\right) \quad \quad\quad \quad\quad \quad\quad \quad\quad \quad\quad \quad\quad \quad\quad \quad\quad \quad\quad \quad\quad \quad\quad \quad
\]
\[
\quad =~ \left(t \lambda_1 +(1-t)\lambda_2\right)x_0 ~+~ t \left(1-\lambda_1\right)x_1    ~+~   (1-t)\left(1-\lambda_2\right)x_2\quad \quad\quad \quad\quad \quad\quad \quad\quad \quad\quad \quad\quad \quad\quad \quad\quad \quad\quad \quad\quad \quad\quad \quad
\]
\[
\quad =~ \left(t \lambda_1 +(1-t)\lambda_2\right)x_0 \quad\quad \quad\quad \quad\quad \quad\quad \quad\quad \quad\quad \quad\quad \quad\quad \quad\quad\quad\quad \quad\quad \quad\quad \quad\quad \quad\quad \quad\quad \quad\quad \quad\quad \quad\quad
\]
\[
\quad \quad+  [ 1-  (t \lambda_1 +(1-t)\lambda_2)  ]\left(\frac{ t (1-\lambda_1)}{1-  (t \lambda_1 +(1-t)\lambda_2)}          x_1    +  \frac{(1-t)(1-\lambda_2)}{1-  (t \lambda_1 +(1-t)\lambda_2)} x_2\right)\quad\quad \quad\quad \quad\quad \quad\quad \quad\quad \quad\quad \quad\quad \quad\quad \quad\quad
\]
\[
\quad \in~ (t \lambda_1 +(1-t)\lambda_2)x_0 ~+~ [ 1-  (t \lambda_1 +(1-t)\lambda_2)   ] C\subset A.\quad\quad \quad\quad \quad\quad \quad\quad \quad\quad \quad\quad \quad\quad \quad\quad \quad\quad\quad\quad \quad\quad \quad\quad \quad\quad \quad\quad \quad\quad \quad\quad \quad\quad \quad\quad
\]
Therefore, since $t \lambda_1 +(1-t)\lambda_2\in [0,1]$, we conclude that $A$ is convex.
We claim that $A$ is a closed convex subset of $X$. In fact, take a point  $y_0\in \overline{A}$. Then, since $y_0\in \overline{A}$,  there exists a sequence $\{\lambda_n x_0+ (1-\lambda_n)x_n\}_{n=1}^{\infty}\subset A$ such that
\[
\mathop {\lim }\limits_{n \to \infty } \left\|  \lambda_n x_0+ (1-\lambda_n)x_n-y_0                  \right\| =0.
\]
Hence we may assume without loss of generality that $\lambda_n \to \lambda_0\in [0,1]$. Then
\[
\mathop {\lim }\limits_{n \to \infty } \left\|  \lambda_0 x_0+ (1-\lambda_0)x_n-y_0                  \right\| =\mathop {\lim }\limits_{n \to \infty } \left\|  \lambda_n x_0+ (1-\lambda_n)x_n-y_0                  \right\| =0.
\]
This means that $\{x_n\}_{n=1}^{\infty}$ is a Cauchy sequence.  Since the set $C$ is a closed convex subset of $X$, we have
$x_n \to x_0\in C$. Hence we have $y_0 = \lambda_0 x_0+ (1-\lambda_0)x_0\in A$. It follows that  $A$ is a closed convex set.
Therefore, since $A\subset \overline{\mathrm{co}}\left(\{x_0\}\cup C\right)$ and $A\supset \left(\{x_0\}\cup C\right)$, we conclude that $A= \overline{\mathrm{co}}\left(\{x_0\}\cup C\right)$.

\par (b) Since $x_0\not\in C$, we obtain that $\mu_{D}(x_0)=1 $ and $\mu_{C}(x_0)>1 $.
Moreover, since $0\in \mathrm{int} C$, we obtain that the Minkowski functionals
$\mu_{C}$ and $\mu_{D}$ are continuous on $X$. Therefore, since
$\mu_{D}(x_0)=1 $ and $\mu_{C}(x_0)>1 $,
there exists a point $y_0\in H$ with $\mu_{D}\left(y_0\right)=1 $ such that
$\mu_{C}(y_0)>1 $ and $y_0 \not\in$ $  \{\lambda x_0: \lambda\in \mathbb R\}$. Hence we obtain that $y_0\not\in C$. Then, since
$y_0\not\in C$ and $\mu_{D}\left(y_0\right)=1 $,
there exists a point $z_0$ $\in C$ such that
$y_0\in (x_0, z_0)= \{\lambda x_0+(1-\lambda)z_0: \lambda\in (0,1)\}$.
Since $\mu_{D}$ is a continuous Minkowski functional, $1=\mu_{D}\left(y_0\right)= \mu_{D}\left(x_0\right)\geq \mu_{D}\left(z_0\right)$
and $y_0\in (x_0, z_0)$, we obtain that
\[
\mu_{D}\left(x\right)\equiv 1~~\quad~~~~~\mathrm{whenever}~~~~~~\quad~~~x\in \left[x_0, z_0\right]=\{\lambda x_0+(1-\lambda)z_0: \lambda\in [0,1]\}.~~~~\eqno ~~~~~~~~~~~~~~~~~ (2.3)
\]
Since $H$ is a two-dimensional subspace of $X$,  $x_0\in H\backslash C$ and $y_0 \not\in \left\{\lambda x_0: \lambda\in \mathbb R\right\}$,
there exists a functional $x^{*}|_{H}\in H^{*} \backslash\{0\}$
such that
\[
L=\{x\in H: x^{*}|_{H}(x)=x^{*}|_{H}(x_0) =x^{*}|_{H}(y_0)    \}=\{ (1-\lambda)x_0+\lambda y_0:  \lambda\in\mathbb R       \}.
\]
Therefore, since $y_0 \not \in \{\lambda x_0: \lambda\in\mathbb R\}$, we have $0 \not \in \{ (1-\lambda)x_0+\lambda y_0:  \lambda\in\mathbb R       \}$.
Hence we may assume without loss of generality  that $x^{*}|_{H}(x_0)>0$.
Pick a point $x\in H$ such that $x^{*}|_{H}(x)>x^{*}|_{H}(x_0)$. We claim that $x\notin D$. In fact, suppose that $x\in D$.
Then there exists a real number $\lambda_0\in (0,1)$ such that $x^{*}|_{H}\left(\lambda_0 x\right)=x^{*}|_{H}\left(x_0\right)$ and
$\lambda_0 x \in D$. Since $\mu_{D}\left(x_0\right)  =\mu_{D}\left(y_0\right)=1 $  and $\lambda_0 x \in D$, by (2.3) and the fact that
\[
\lambda_0 x \in \{\lambda x_0+\left(1-\lambda\right)z_0\in X: \lambda\in \mathbb R\},
\]
we have $\mu_{D}\left(x\right)>\mu_{D}\left(\lambda_0 x\right)= 1$, a contradiction. Hence we get that $x\notin D$. Then
\[
x^{*}|_{H}\left(x_0\right)= x^{*}|_{H}\left(y_0\right)=\sup \left\{ x^{*}|_{H}(x):   x\in H \cap D                 \right\}.
\]
Since $H$ is a two-dimensional subspace of $X$, by the above formula, there exists a point $y_{0,1}\in D\cap H$ such that (For convenience, $y_{0,1}$ is still referred to as $y_{0}$)
\[
\big[x_0,y_{0}\big]= \{x\in X: \mu_{D}(x)=1\}\cap \{ (1-\lambda)x_0+\lambda y_{0}:  \lambda\in \mathbb  R       \}.
\]

\par (c) We pick a point $e_0\in[x_0,y_0]$ such that $e_0\in \left\{ \left(1-\lambda\right)x_0+\lambda y_0:  \lambda\in \left(0,1\right)     \right  \}$ and $e_0$ is sufficiently close to $x_0$.
Therefore, from the proof of (b), we obtain that $ \mu_{D}\left(e_0\right)=1 $. Since $C_1$ is a  closed convex subset of $C$ and $D_1=\overline{\mathrm{co}}\left(\{x_0\}\cup C_1\right)$,
we obtain that $D_1$ is a  closed convex subset of $D$. Moreover, since $ \mu_{D}\left(e_0\right)=1 $, we conclude that $ \mu_{D_1}(e_0)\geq 1 $.
Hence there exists a point $z_0  \in \{\lambda e_0: \lambda \in [0,1] \}$ such that
$\mu_{D_1}(z_0)=1$. Therefore, from the proof of part (b) and the equalities $\mu_{D_1}(x_0)=$ $\mu_{D_1}(z_0)=1$, we conclude that
\[
\mu_{D_1} \left[\left(1-\lambda\right)x_0+\lambda z_0\right] \equiv 1~~\quad~~~\mathrm{ whenever}~~~~~\quad~~~~\lambda \in [0,1]
\]
whenever $e_0$ is sufficiently close to $x_0$.
Moreover, since $e_0\in[x_0,y_0]$ and $z_0\in \{\lambda e_0:$ $\lambda\in[0,1]\}$, we obtain that $z_0\in \mathrm{co} \left\{x_0,y_0,0\right\}$.
Then, since $z_0\in \mathrm{co} \left\{x_0,y_0,0\right\}$, there exist $\lambda_1,\lambda_2,\lambda_3\in[0,1]$
with $\lambda_1+\lambda_2+\lambda_3=1$
such that $z_0=\lambda_1 x_0+\lambda_2 0+\lambda_3 y_0$.
Pick a point
\[
u_0= \left( 1- \frac{1}{1-\lambda_{1}}     \right)x_0 + \frac{1}{1-\lambda_{1}}  z_0 \in \left\{ (1-\lambda)x_0+\lambda z_0: \lambda \geq 0  \right  \}.
\]
Since $\{\lambda_{1}, \lambda_{2}, \lambda_{3}\}\subset [0,1]$, $\lambda_{1}+ \lambda_{2}+\lambda_{3} =1$ and $z_0= \lambda_{1} x_0 + \lambda_{2} 0 + \lambda_{3} y_0$, we have
\[
u_0 ~=~  \left( 1- \frac{1}{1-\lambda_{1}}     \right)x_0 + \frac{1}{1-\lambda_{1}}  z_0\quad \quad\quad \quad \quad \quad \quad
\]
\[
\quad \quad \quad \quad \quad \quad \quad \quad  =~ \left( 1- \frac{1}{1-\lambda_{1}}     \right)x_0  ~+~ \frac{1}{1-\lambda_{1}}  \left( \lambda_{1} x_0 + \lambda_{2} 0 + \lambda_{3} y_0 \right)\quad \quad \quad \quad \quad \quad \quad \quad \quad \quad\quad \quad \quad \quad \quad \quad \quad \quad \quad \quad \quad \quad \quad \quad \quad \quad \quad \quad \quad \quad
\]
\[
\quad \quad \quad \quad \quad \quad \quad \quad =~ \frac{1}{1-\lambda_{1}}  \left(  \lambda_{3} y_0 \right)= \frac{\lambda_{3} }{\lambda_{2} +\lambda_{3} }  y_0\in \left [0,y_0 \right]. \quad \quad \quad \quad \quad \quad \quad \quad \quad \quad \quad \quad \quad \quad \quad \quad \quad \quad \quad \quad\quad \quad \quad \quad \quad \quad \quad \quad \quad \quad
\]
This implies that $u_0\in \left\{ \left(1-\lambda\right)x_0+\lambda z_0: \lambda \geq 0  \right  \} $ and $u_0\in [0,y_0]$. Hence Lemma 2.4 holds,
which completes the proof.
\end{proof}

\begin{lemma}
Suppose that

\par (1) the space $X$ is a two-dimensional Banach space;
\par (2) $L= \left\{ (1-\xi) x_0  + \xi y_0 \in X:  \xi \geq 0  \right\}$ and $H= \left\{ \xi x_0\in X : \xi\leq 1\right\}$, where $x_0\neq 0$ and $y_0\not\in \left\{ \xi x_0\in X : \xi\in \mathbb R\right\}$;
\par (3) $z_0\not\in \overline{\mathrm{co}}\left(L\cup H\right)$, $x^{*}(z_0)>0$, $x^{*}(y_0)>0$ and
$x^*(x_0)=0$.
\\
Then there exists a point $\xi_0 \in [ 0 ,+ \infty)$ such that $(1-\xi_{0}) x_0  + \xi_{0} y_0 \in [0,z_0]$.
\end{lemma}

\begin{proof}
Since $X$ is a two-dimensional Banach space and $L= \{ (1-\xi) x_0  + \xi y_0 \in X:  \xi \geq 0  \}$, there exists a functional $y^{*}\in X^{*}\backslash \{0\}$ such that
\[
y^{*}(z)\equiv 1~~\quad~~~~\mathrm{whenever}~~~~~~\quad~~z\in L_0= \left\{\left (1-\xi\right) x_0  + \xi y_0 \in X:  \xi \in \mathbb R \right\}.
\]
Then $y^{*}\left(z\right)\leq 1$ whenever $z\in H$. Since $x^{*}\left(y_0\right)>0$ and $x^*(x_0)=0$, we conclude that $x^{*}(z)\geq 0$ whenever $z\in L$.
Moreover, we have $x^{*}(z)= 0$ whenever $z\in H$. Since $X$ is a two-dimensional Banach space,
we conclude that
\[
\overline{\mathrm{co}}\left(L\cup H\right)= \{x\in X:x^{*}(x)\geq 0 \} \cap \{x\in X:y^{*}(x)\leq 1 \}.
\]
Therefore, since $z_0\not\in \overline{\mathrm{co}}\left(L\cup H\right)$ and $x^{*}(z_0)>0$, we obtain that $ y^{*}(z_0)>1 $.
Since $y^{*}(z)=1$ for all $z\in L_0$, from $ y^{*}(0)=0 $ and $ y^{*}(z_0)>1 $, there exists a point
\[
u_0\in L_0=\left\{ \left(1-\xi\right) x_0  + \xi y_0 \in X:  \xi \in \mathbb R \right\}
\]
such that $u_0\in[0,z_0]$. Moreover, since $x^{*} (z_0 ) >0$ and $u_0\in[0,z_0]$, we obtain that $ x^{*}( u_0)\geq 0$.
Let $u_0=\left(1-\xi_{0}\right) x_0  + \xi_{0} y_0 $. Then, since $ x^{*}\left(u_0\right) \geq 0$ and $ x^*\left( x_0\right)=0$, we have the following inequalities
\[
x^{*}\left(u_0\right)=x^{*}\left(\big(1-\xi_{0}\big) x_0  + \xi_{0} y_0 \right)=\xi_{0}x^{*}\left(y_0\right)\geq 0.
\]
Moreover, since $x^{*}(y_0)>0$, we conclude that $\xi_{0} \geq 0$. Hence
there exists  $\xi_0 \in [ 0,$ $+\infty) $ such that $u_0=\left(1-\xi_{0}\right) x_0  + \xi_{0} y_0 \in [0,z_0]$, which completes the proof.
\end{proof}

\begin{lemma}
Suppose that

\par (1) $T: X\rightarrow2^{X^{*}}$ is a maximal monotone operator and the set $D=\mathrm{int} D\left(T\right)= \mathrm{int} \left\{x\in X: T\left(x\right)\neq \emptyset \right\}$ is nonempty;
\par (2) the set $C$ is a bounded closed convex set of $X$ and $0\in \mathrm{int}C$;
\par (3) there exist two points $x_0\in X$ and $e_0\in X$ with $\mu_{C}(e_0)=1$ such that $e_0 T$ is single-valued at $x_0\in X$;
\par (4) $e_0 T(x_0) =\alpha$ and $\sup \left\{ \sigma_{T}\left(x_0, e\right) : \mu_{C}(e)=1   \right\}\leq \alpha$.
\\
Then $T\left(x_0\right)\subset \alpha \partial \mu_{C}\left(e_0\right)$. Moreover, if the Minkowski functional $\mu_{C}$ is G$\hat{a}$teaux differentiable at $e_0\in X$, then
$T(x_0)$ is a singleton.
\end{lemma}

\begin{proof}
Pick a point $x^{*}\in T\left(x_0\right)$. Then, since $\sup \left\{ \sigma_{T}\left(x_0, e\right) : \mu_{C}(e)=1   \right\}\leq \alpha$, we have
$\langle x^{*},e \rangle \leq \alpha$ whenever $\mu_{C}(e)\leq 1$.  Since $e_0 T$ is single-valued at $x_0$, $e_0 T\left(x_0\right) =\alpha$ and $x^{*}\in T\left(x_0\right)$,
we have $\left\langle x^{*},e_0 \right\rangle = \alpha$. This implies that $x^{*}\in \alpha \partial \mu_{C}\left(e_0\right)$. Hence
$T(x_0)\subset \alpha \partial \mu_{C}\left(e_0\right)$. It follows  that if $\mu_{C}$ is G$\mathrm{\hat{a}}$teaux differentiable at  $e_0\in $ $X$, then $T\left(x_0\right)$ is a singleton, which completes the proof.
\end{proof}

\begin{lemma}
Suppose that

\par (1) $C$ is a bounded closed convex subset of a Banach space $X$ and $0\in \mathrm{int}C$;
\par (2) $\mu_{C}\left(x_0\right)=1$, $\lambda\in (0,1)$, $c\in (0,+\infty)$ and $A= \overline{\mathrm{co}}\left(\left\{x_0\right\}\cup \lambda C\right)$.
\\
Then the point $x_0/\left[\mu_{C}\left(x_0\right)+ c\mu_{A}\left(x_0\right)\right]$ is an extreme point of $D$, where
\[
D=\{x\in X: \mu_{C}\left(x\right)+ c\mu_{A}\left(x\right)\leq 1\}.
\]
\end{lemma}

\begin{proof}

Suppose that the point $x_0/\left[\mu_{C}\left(x_0\right)+ c\mu_{A}\left(x_0\right)\right]$ is not an extreme point of $D$. Then there exist two points $x_1\in D$ and $x_2\in D$ with $x_1\neq x_2$ such that
\[
\frac{1}{\mu_{C}\left(x_0\right)+ c\mu_{A}\left(x_0\right)}x_0=\frac{1}{2}x_1+\frac{1}{2}x_2\in D~~~~\eqno ~~~~~~~~~~~~~~~~~ (2.4)
\]
and $  \left\|x_1- x_2\right\|$ is sufficiently close to $0$.
Since $D=\{x\in X: \mu_{C}\left(x\right)+ c\mu_{A}\left(x\right)\leq 1\}$, we conclude that $\mu_{D}=\mu_{C}+ c\mu_{A}$. It
follows that $\mu_{D}\left(x_1\right) =1$ and $\mu_{D}\left(x_2\right) =1$. Define the two-dimensional space $H$, where
\[
H= \big\{ \alpha  x_1 + \beta x_2\in X:  \alpha \in\mathbb R,~~\beta  \in\mathbb R            \big\}.
\]
Since $A= \overline{\mathrm{co}}\left(\{x_0\}\cup \lambda C\right)$, $\lambda\in (0,1)$, $0\in \mathrm{int}C$ and $\mu_{C}(x_0)=1$, we have $\mu_{A}(x_0)= 1$. Moreover, by Lemma 2.4,
there exists a point $y_1\in \lambda C \cap H$ such that
\[
\big[x_0,y_1\big]= \big \{x\in X: \mu_{A}\left(x\right)=1\big\}\cap \big\{ \left(1-\lambda\right)x_0+\lambda y_1\in X:  \lambda\in\mathbb R    \big   \}.
\]
Since $  \|x_1- x_2\|$ is sufficiently close to $0$, by (2.4), we may assume that $x_1/\mu_{A}(x_1)\in \left [x_0, y_1\right ]$.
Since $\mu_{C}\left(x_0\right)=1$, $y_1\in\mathrm{int}C$ and $x_1/\mu_{A}\left(x_1\right)\in \left[x_0, y_1\right]$, we conclude that $x_1/\mu_{A}\left(x_1\right)\in\mathrm{int}C$.
This implies that
$ \mu_{C}\left[x_1/\mu_{A}\left(x_1\right)\right]<1$. Hence we obtain that
$\mu_{A}\left(x_1\right)>\mu_{C}\left(x_1\right)$. Moreover, since $\mu_{D}\left(x_1\right) =1$, we obtain that
\[
1=\mu_{D}\left(x_1\right) = \mu_{C}\left(x_1\right)+ c\mu_{A}\left(x_1\right)>\mu_{C}\left(x_1\right)+c\mu_{C}\left(x_1\right)=  \left(1+c\right)\mu_{C}\left(x_1\right).
\]
Hence $\mu_{C}\left(x_1\right)<1/\left(1+c\right)$. Similarly, we have $\mu_{C}\left(x_2\right)<1/\left(1+c\right)$. Since $\mu_{C}\left(x_0\right)$ $=\mu_{A}\left(x_0\right)=1$,
$\mu_{C}\left(x_1\right)<1/\left(1+c\right)$ and  $\mu_{C}\left(x_2\right)<1/\left(1+c\right)$, we obtain that
\begin{eqnarray*}
\frac{1}{1+c}=\mu_{C} \left(\frac{1}{\mu_{C}\left(x_0\right)+ c\mu_{C}\left(x_0\right)}x_0\right)&=&\mu_{C} \left(\frac{1}{\mu_{C}\left(x_0\right)+ c\mu_{A}\left(x_0\right)}x_0\right)
\\
&=&\mu_{C} \left(\frac{1}{2}x_1+\frac{1}{2}x_2\right)<\frac{1}{1+c},
\end{eqnarray*}
a contradiction. Hence the point $x_0/\left[\mu_{C}\left(x_0\right)+ c\mu_{A}\left(x_0\right)\right]$ is an extreme point of $D$,
which completes the proof.
\end{proof}

We next prove  Theorem 2.1.

\begin{proof}
\par We define the norm $p_1\left(x,y\right)= \|(x,y)\|=\max\left\{\|x\|, |y|\right\}$ on  $X\times \mathbb R$.
Let $f$ be a continuous convex function on $X\times \mathbb R$. Then we define the set
\[
G= \left\{ (x,y)\in X\times \mathbb R : ~\mathrm{The}~~~~\mathrm{set}~~\partial f\left(x,y\right)~~\mathrm{is}~~~~\mathrm{a}~~~~\mathrm{singleton}    \right\}.
\]
It is well known that  $\partial f: X\times \mathbb R\to 2^{X^{*}\times \mathbb R}$ is   a maximal monotone operator.
We next prove that $G$ contains a dense $G_\delta$ subset using the Banach-Mazur game.
For clarity, we divide the proof into five steps.

\textbf{Step 1.} To apply the Banach-Mazur game, note that $X\times\mathbb R$ is a completely metrizable Hausdorff space and $G\subset X\times\mathbb R$. Let $U_1$ be an open subset of $X\times\mathbb R$. Define the following two sets
\[
S_{1}\left(X\times\mathbb R\right)=\left\{(x, y): p_1(x,y)= 1\right\}~~\quad~~~~~\mathrm{and}~~~~~~\quad~~~B_{1}\left(X\times\mathbb R\right)=\left\{(x, y) : p_1(x,y)\leq 1\right\}.
\]
Then Player $A$ chooses an open subset $U_1$ of $X\times\mathbb R$. Moreover, we may assume without loss of generality that $\partial f\left(U_1\right)\subset B_1\big(X^{*}\times\mathbb R\big)$. Then we may assume that
\[
\sup \left\{ \big \|(x^{*},y^*)\big\| \in \mathbb R:   (x^{*},y^*)\in \partial f\left(U_{1}\right)      \right\}>0.
\]
(Otherwise, we have $\sup\bigl\{\|({x}^{*},y^{*})\| : ({x}^{*},y^{*})\in \partial f\left(U_1\right)\bigr\}=0$. Hence $\partial f\left(x,y\right)$ is a singleton for every $(x,y)\in X\times\mathbb R$.
Then $U_1\subset G$. Hence Player $B$ chooses $V_k=U_k$ for every $k\in\mathbb N$. Then $\cap_{k=1}^{\infty}V_k\subset G$.
It follows that Theorem 2.1 holds.)
Since $\partial f\left(U_{1}\right) \subset B_{1}\left(X^{*}\times \mathbb R\right)$, by Lemma 1.8, we conclude that
\begin{eqnarray*}
s_{1}&=&\sup\left\{ \sigma_{\partial f}\left ((x,y),(e_{X},e)\right): ((x,y),(e_{X},e))\in U_{1}\times S_{1}\left(X\times \mathbb R\right) \right\}
\\
&=&\sup \left\{  \big\|(x^{*},y^*)\big\| \in \mathbb R:   (x^{*},y^*)\in \partial f\left(U_{1}\right)      \right\}>0.
\end{eqnarray*}
Since $s_{1}>0$, we may assume without loss of generality that $s_{1}=1$.
Therefore, by Lemma 1.8, we obtain that for every $(e_{X},e)\neq (0,0)$, $(x,y)\in U_{1}$ and $\eta>0$, there exists  a real number $t\in (0,\eta)$ such that
\[
(x,y)+t\left(e_{X},e\right)\in U_{1}, \quad \sigma_{\partial f} \left((x,y),(e_{X},e)\right)\leq\sigma_{\partial f} \left((x,y)+t\left(e_{X},e\right)\right)
\]
and $(e_{X},e)(\partial f)$ is single-valued at
$(x,y)+t\left(e_{X},e\right)$. Thus we have
\begin{eqnarray*}
s_{1}&=&\sup \left\{ \sigma_{\partial f}\left ((x,y),(e_{X},e)\right): ((x,y),(e_{X},e))\in U_{1}\times S_{1}\left(X\times \mathbb R\right)~~~\mathrm{and}~~\right.
\\
&&\quad \quad \left. (e_{X},e)(\partial f) ~~\mathrm{is}~~~~\mathrm{ a }~~\mathrm{ single}~~\mathrm{valued } ~~\mathrm{ mapping} \right\}.
\end{eqnarray*}
Define the projection $T\colon X\times\mathbb R\to X$ by $T(x,y)=x$. Since $X$ is a weak Asplund space and $p_1(x,y)=\max\{\|x\|,|y|\}$, by Lemma 2.2 and Lemma 2.3,
there exists a decreasing sequence of dense open cones $\{O_n^1\}_{n=1}^{\infty}$ contained in $G_1$
such that $p_1$ is G$\mathrm{\hat{a}}$teaux differentiable on $\cap_{n=1}^{\infty}O_n^1$ and $O_{n+1}^1\subset O_n^1$, where
\[
G_1= \{  \lambda \left(x, \eta \right) :   x\in T \left( B_1\left(X\times \mathbb  R\right) \right),~  \lambda\in (0 ,+\infty) ,~~ \eta=\pm 1  \}.
\]
Pick a real number $\varepsilon_{1}\in \left(0,1/512^{6}\right)$ such that $\prod_{i=0}^{\infty}\big[1-(20\varepsilon_{1}/128^{i})\big]>3/4$. Then,
since $O_{1}^{1}$ is a dense open subset of $G_1$,
there exists a point $\big((x_{1},y_{1}),(e_{X,1},e_{1})\big)\in U_{1} \times S_{1}\left(X\times \mathbb R\right)$ with
\[
(e_{X,1},e_{1})\in \mathop {\cap}\limits_{n=1}^{\infty}O_{n}^{1}~~~\quad~~~~~~~\mathrm{and}~~~~~~~~~\quad~~~T\left(e_{X,1},e_{1}\right)\in \mathrm{int} T\left\{ (x,y)\in X\times \mathbb R: p_1 (x,y) \leq 1    \right\}
\]
such that the mapping $(e_{X,1},e_{1})(\partial f)$ is single-valued at $(x_{1},y_{1})$ and
\[
\sigma_{\partial f}\left((x_{1},y_{1}),(e_{X,1},e_{1})\right) >\left(1-\frac{1}{16^{2}}\varepsilon_{1}^{36}\right)s_{1}>0.
\]
Since the mapping $(e_{X,1},e_{1})(\partial f)$ is single-valued at $(x_{1},y_{1})$, by Lemma 1.8, we obtain that the functional
$(x,y) \to \sigma_{\partial f}\left((x,y),(e_{X,1},e_{1})\right)$ is continuous at $(x_{1},y_{1})$.
Since $U_{1}$ is an open set, there exists a real number $r_1\in (0,1)$ such that
\[
B\big((x_{1},y_{1}),2r_1\big)\subset U_{1}~\quad~~~~~\mathrm{and}~~~~~~~~\quad~\sigma_{\partial f}\big((x,y),(e_{X,1},e_{1})\big)>\left(1-\frac{1}{16^{2}}\varepsilon_{1}^{36}\right)s_{1}>0~\eqno~~~~~~~~~~~~~~~~~~~~~~~~~~(2.5)
\]
for all $(x,y)\in B((x_{1},y_{1}),2r_1)$.
Define the set $V_1=\mathrm{int}B((x_{1},y_{1}),r_1)$.
Then Player $A$ may choose
any nonempty open subset $U_2 \subset V_1$. From the preceding argument,
we may assume without loss of generality that
\[
\sup \big\{  \|(x^{*},y^*)\| \in \mathbb R:   (x^{*},y^*)\in \partial f\left(U_{2}\right)     \big\}>0.
\]
Moreover, since $(e_{X,1},e_{1})\in O_{1}^{1}$, there exists a real number $\eta_1\in \big(0,\varepsilon_1^{4}/512^{6}\big)$ such that
$B\big((e_{X,1},e_{1}),256\eta_1\big)\subset O_{1}^{1}$.
Pick
a real number $h_{1}\in \left(32/\varepsilon_{1}^{72},+\infty\right)$.  Then we define the set $C_1\subset X\times\mathbb R$,   where
\[
C_1=\big\{  \left(\alpha e_{X,1},\alpha e_{1}\right):  0\leq \alpha  \leq   h_{1}            \big\} \cup \big\{ (x,y):  p_1 (x,y)=\max \left\{\|x\|, |y|\big\}\leq 1       \right\}.
\]
Therefore, by the definition of $C_1$, we conclude that $0\in \mathrm{int}C_1$. Moreover, since
$h_{1}\in \big(32/\varepsilon_{1}^{72},+\infty\big)$,  we
define the non-convex functional $\mu_{C_1}$,  where
\[
\mu_{ C_1}(x,y)= \inf \left\{  \lambda\in \mathbb R^{+} :  \frac{1}{\lambda} \left(x,y\right)\in   C_1          \right\}
\]
for every  $(x,y)\in X\times \mathbb R$. Then we obtain that $\mu_{C_1}$ is a non-continuous functional and $\mu_{ C_1}(x,y)\in [0,+\infty)$ for every  $(x,y)\in X\times\mathbb R$. Moreover, we have
$\mu_{ C_1}(x,y)=0$ if and only if $(x,y)=(0,0)$.
Define the functional $p_{2}'$,  where
\[
p_{2}'\left(x,y\right)=p_1\left(x,y\right)+16\varepsilon_{1} \cdot\mu_{ C_1}(x,y)~\quad~~~~~\mathrm{for}~~~~~~~~~\mathrm{every}~~~~~~~\quad~~(x,y)\in X\times\mathbb R.
\]
Therefore, since $ p_1\left(e_{X,1},e_{1}\right)=1$,
there exists a real number $\alpha_1\in (0,1)$ such that
\[
p_1\left(\alpha_1e_{X,1},\alpha_1e_{1}\right)+  16\varepsilon_{1} \cdot \mu_{C_1 }\left(\alpha_1e_{X,1},\alpha_1e_{1}\right)= p_{2}'\left(\alpha_1e_{X,1},\alpha_1e_{1}\right)=1.
\]
Since $h_{1}\in \left(32/\varepsilon_{1}^{72},+\infty\right)$ and $\alpha_1\in (0,1)$, by the definition of $\mu_{C_1 }$ and the preceding formula, we obtain the following inequalities
\[
\alpha_1=p_1\left(\alpha_1e_{X,1},\alpha_1e_{1}\right)= 1- 16 \varepsilon_{1}\cdot \mu_{C_1 }\left(\alpha_1e_{X,1},\alpha_1e_{1}\right)\geq 1-  \frac{1}{32}\varepsilon_{1}^{36} .~~~\eqno~~~~~~~~~~~~~~~~~~~~~~~~~~(2.6)
\]
Moreover, we define the set $S'_{2}\left(X\times \mathbb R\right)$ and  the real number $s_2$, where
\[
S'_{2}\big(X\times \mathbb R\big)=\left\{(x, y)\in X\times\mathbb  R: p_{2}'\left(x,y\right)= p_1(x,y)+16\varepsilon_1 \cdot \mu_{ C_1 }(x,y)=1\right\}
\]
and
\[
s_{2}=\sup\left\{ \sigma_{\partial f} \left((x,y),(e_{X},e)\right): ((x,y),(e_{X},e))\in U_{2}\times S'_{2}\left(X\times \mathbb R\right) \right\}>0.
\]
Furthermore,  by  formulas (2.5) and (2.6), we have the following inequalities
\[
\sigma_{\partial f}\big((x,y),(\alpha_1e_{X,1},\alpha_1e_{1})\big)= \alpha_1\sigma_{\partial f}\big((x,y),(e_{X,1},e_{1})\big)
>\left(1-\frac{1}{16^{2}}\varepsilon_{1}^{4}\right)s_{1}>0
\]
for every $(x,y)\in  U_2$.
We claim that if $(x,y)\in  U_2\subset V_1$, $p_2'\left(h_{X,2},h_{2}\right)=1$ and
\[
\sigma_{\partial f}\big((x,y),(h_{X,2},h_{2})\big)\geq \sigma_{\partial f}\big((x,y),(\alpha_1e_{X,1},\alpha_1e_{1})\big)
>\left(1-\frac{1}{16^{2}}\varepsilon_{1}^{4}\right)s_{1}>0,
\]
then $(h_{X,2},h_{2})=\left(\alpha_1e_{X,1},\alpha_1e_{1}\right)$.
In fact, suppose that it is not true. Then we get that $(h_{X,2},h_{2})\neq\left(\alpha_1e_{X,1},\alpha_1e_{1}\right)$. Moreover, by the definition of $C_1$,  it is easy to see that
$p_1\left(h_{X,2},h_{2}\right)=\mu_{C_1}\left(h_{X,2},h_{2}\right)$. Therefore, by the definition of $p_2'$, we have
\[
1=p_2'\left(h_{X,2},h_{2}\right)=p_1\left(h_{X,2},h_{2}\right)+16\varepsilon_1 \cdot p_1\left(h_{X,2},h_{2}\right).
\]
Hence we obtain that $p_1\left(h_{X,2},h_{2}\right)=1/ \left(1+16\varepsilon_1\right)$. Therefore, by the definitions of $s_1$ and $\sigma_{\partial f}$, we have the following inequalities
\[
\frac{s_{1}}{1+16\varepsilon_1} = \frac{1}{1+16\varepsilon_1}\sup\left\{ \sigma_{\partial f} ((x,y),(e_{X},e)): ((x,y),(e_{X},e))\in U_{1}\times S_{1}(X\times \mathbb R) \right\} \quad \quad \quad \quad
\]
\[
\quad \quad \quad\quad = \sup\left\{ \sigma_{\partial f} ((x,y),(e_{X},e)): ((x,y),(e_{X},e))\in U_{1}\times \frac{1}{1+16\varepsilon_1}S_{1}(X\times \mathbb R) \right\}\quad \quad \quad\quad \quad \quad\quad \quad \quad\quad \quad \quad
\]
\[
\quad \quad \quad\quad\geq \sigma_{\partial f}\big((x,y),(h_{X,2},h_{2})\big)\quad \quad \quad\quad\quad \quad \quad\quad\quad \quad \quad\quad\quad \quad \quad\quad\quad \quad \quad\quad\quad \quad \quad\quad\quad \quad \quad\quad\quad \quad \quad\quad\quad \quad \quad\quad
\]
for every $(x,y)\in U_2$. Therefore, since $\varepsilon_{1}\in \left(0,1/512^{6}\right)$ and $s_1=1$, we have
\[
\left(1-\frac{1}{16^{2}}\varepsilon_{1}^{4}\right) s_{1}\leq \sigma_{\partial f}\left((x,y),(h_{X,2},h_{2})\right)\leq \frac{1}{1+16\varepsilon_1}s_1< \left(1-\frac{1}{16^{2}}\varepsilon_{1}^{2}\right) s_{1}
\]
for every $(x,y)\in  U_2$, a contradiction. Hence we obtain that $(h_{X,2},h_{2})=(\alpha_1e_{X,1},$ $\alpha_1e_{1})$.
Therefore, from the preceding argument, it is easy to see that
\[
s_2=   \sup \left\{\sigma_{\partial f}\left((x,y),(\alpha_1 e_{X,1},\alpha_1 e_{1})\right):   (x,y)\in U_2    \right\}.
\]
Moreover,
we define the continuous Minkowski functional $\mu_{\overline{\mathrm{co}}\left(S'_{2}\left(X\times \mathbb R\right)\right)}$, where
\[
\mu_{\overline{\mathrm{co}}\left(S'_{2}\left(X\times \mathbb R\right)\right)}(x,y)= \inf \left\{ \lambda\in\mathbb R^{+}: \frac{1}{\lambda} \left(x,y\right)\in  \overline{\mathrm{co}}\left(S'_{2}\left(X\times \mathbb R\right)\right)                \right\}
\]
for every $(x,y)\in \overline{\mathrm{co}}\left(S'_{2}\left(X\times \mathbb R\right)\right)$. Hence we define the set $S_{2}\left(X\times \mathbb R\right)$, where
\[
S_{2}\left(X\times\mathbb R\right)=\left\{(x, y)\in X\times \mathbb R: \mu_{\overline{\mathrm{co}}\left(S'_{2}\left(X\times \mathbb R\right)\right)}(x,y)=1\right\}.
\]
Therefore, by the definitions of $\mu_{\overline{\mathrm{co}}\left(S_2'(X\times\mathbb R)\right)}$ and $p_2'$, it is easy to see that
\[
s_{2}=\sup\left\{ \sigma_{\partial f} ((x,y),(e_{X},e)): ((x,y),(e_{X},e))\in U_{2}\times S_{2}\left(X\times \mathbb R\right) \right\}>0.
\]
Moreover, since $X$ is a weak Asplund space, by Lemma 2.2 and Lemma 2.3, there exists a decreasing sequence of dense open cones $\{O_n^2\}_{n=1}^{\infty}$ contained in $G_2$
such that $\mu_{\overline{\mathrm{co}}\left(S_2'(X\times\mathbb R)\right)}$ is G\^ateaux differentiable on $\cap_{n=1}^{\infty}O_n^2$ and $O_{n+1}^2\subset O_n^2$, where
\begin{eqnarray*}
G_2 &=& \left\{   \lambda \left(x, f_2(x)\right) \in X\times \mathbb R :   x\in T \left( \overline{\mathrm{co}}\left(S_{2}\left(X\times R\right)\right) \right),~  \lambda\in (0,+\infty)     \right\}
\\
&& \cup\left\{  \lambda \left(x, g_2(x)\right)\in X\times \mathbb R :   x\in T \left( \overline{\mathrm{co}}\left(S_{2}\left(X\times R\right)\right) \right),~  \lambda\in (0,+\infty)      \right\},
\end{eqnarray*}
\[
f_{2}\left(x\right)= \inf \left\{ r\in \mathbb R: (x,r) \in   \overline{\mathrm{co}}\left(S_{2}\left(X\times \mathbb R\right)\right)         \right\}
\]
and
\[
g_{2}\left(x\right)= \sup \left\{ r\in \mathbb R: (x,r) \in   \overline{\mathrm{co}}\left(S_{2}\left(X\times \mathbb R\right)\right)         \right\}.
\]
Let $ p_2=   \mu_{\overline{\mathrm{co}}\left(S'_{2}\left(X\times \mathbb R\right)\right)}$ and
$\varepsilon_2 = \varepsilon_1/128 $. Then, since $\prod_{i=0}^{\infty}\left[1-(20\varepsilon_{1}/128^{i})\right]>3/4$, we have $\left(1-16\varepsilon_1\right)\left(1-16\varepsilon_2\right)>3/4$.
Since $(\alpha_1 e_{X,1},\alpha_1 e_{1})$ is an extreme point of $\overline{\mathrm{co}}\left(S'_{2}\left(X\times \mathbb R\right)\right)$,
we obtain  $\alpha_1 e_{1} \notin (f_{2}\left(\alpha_1 e_{X,1}\right),g_{2}\left(\alpha_1 e_{X,1}\right))$.
Since  $O_{n}^{2} $ is a dense open subset of $G_2$ and $\alpha_1 e_{1} \notin \big(f_{2}\left(\alpha_1 e_{X,1}\right),g_{2}\left(\alpha_1 e_{X,1}\right)\big)$, by Lemma 1.8 and the fact that
\[
s_2=   \sup \left\{\sigma_{\partial f}\left((x,y),(\alpha_1 e_{X,1},\alpha_1 e_{1})\right):   (x,y)\in U_2   \right \},
\]
there exist $r_2\in (0,r_1/4)$ and a point $( (x_{2},y_{2}), (e_{X,2},e_{2}))\in U_2 \times \cap_{n=1}^{\infty}O_{n}^{2}$
with
\[
p_2\left(e_{X,2},e_{2}\right)=1~~\quad~~~~~~~~~\mathrm{and}~~~~~~~~~~~\quad~~T\left(e_{X,2},e_{2}\right)\in \mathrm{int} T\left\{ (x,y)\in X\times \mathbb R: p_2 \left(x,y\right) \leq 1  \right  \}
\]
such that (1) the set-valued mapping $\left(e_{X,2},e_{2}\right)(\partial f)$ is single-valued at $(x_{2},y_{2})\in X\times \mathbb R$;
(2) $p_2 \left((\alpha_1e_{X,1}, \alpha_1e_{1})- (e_{X,2},e_{2})  \right )    <\eta_1/50 $ and
\[
\sigma_{\partial f}\big ((x,y),(e_{X,2},e_{2})\big)\geq \left(1-\frac{1}{16 }\varepsilon_{2}^{32}\right) s_{2}>0~~~~~\eqno~~~~~~~~~~~~~~~~~~~~~~~~~~(2.7)
\]
for every $(x,y)\in B\left((x_{2},y_{2}),r_2\right)\subset U_2$. Moreover, by (2.6), we conclude that
$\alpha_1 \in [3/4,1]$.
Noting that  $O_{1}^{1}$ is an open cone,
$p_2 \left ((\alpha_1e_{X,1},\alpha_1 e_{1})-(e_{X,2},e_{2}) \right)    <\eta_1/50  $ and
$B\left((e_{X,1},e_{1}),256\eta_1\right)\subset O_{1}^{1}$, we obtain that
\[
(e_{X,2},e_{2})\in B\left((\alpha_1 e_{X,1}, \alpha_1 e_{1}),\frac{1}{50}\eta_1\right)\subset B\left((\alpha_1 e_{X,1}, \alpha_1 e_{1}),25 \eta_1\right)            \subset O_{1}^{1}.
\]
Moreover, from the preceding formula and $(e_{X,2},e_2)\in\cap_{n=1}^{\infty}O_n^2\subset X\times\mathbb R$, there exists a real number $\eta_2\in\big(0,\min\big\{\eta_1/128,\varepsilon_2^{64}/128\big\}\big)$ such that
\[
 B\left ((e_{X,2},e_{2}),256\eta_2 \right)\subset O_{2}^{1}~~\quad~~~~~~~~~\mathrm{and}~~~~~~~~~~~\quad~~~~ B\left ((e_{X,2},e_{2}),256\eta_2 \right)\subset O_{2}^{2}.
\]
Define the set $V_2=\mathrm{int}B\left((x_{2},y_{2}),r_2\right)$. Then we obtain that $ V_2\subset U_2$.  Therefore, by (2.7) and the fact that  $ s_{2}\leq 1$, we obtain that
\[
\sigma_{\partial f}\left((x_2,y_2),(e_{X,2},e_{2})\right)~\geq~ \left(1-\frac{1}{16 }\varepsilon_{2}^{32}\right) s_{2}
\]
\[
\quad \quad \quad \quad \quad \quad \quad \geq~ \sup \left\{  \sigma_{\partial f}\left((x,y),( e_{X,2}, e_{2})\right):  (x,y)\in U_2                  \right\}~-~ \frac{1}{16 }\varepsilon_{2}^{32} \cdot s_{2}\quad \quad \quad \quad \quad \quad \quad \quad \quad \quad \quad \quad\quad \quad \quad \quad \quad \quad \quad \quad \quad \quad \quad \quad
\]
\[
\quad \quad \quad \quad \quad \quad \quad \geq~ \sup \left\{  \sigma_{\partial f}\left((x,y),( e_{X,2}, e_{2})\right):  (x,y)\in V_2                  \right\}~-~ \frac{1}{16}\varepsilon_{2}^{32} \quad \quad \quad \quad \quad \quad \quad \quad \quad \quad \quad \quad \quad \quad \quad \quad \quad \quad \quad\quad \quad \quad\quad \quad
\]
Since  $\partial f$ is norm-to-weak$^{*}$ upper-semicontinuous, $\left( e_{X,2}, e_{2}\right)\partial f$ is single-valued at $(x_2,y_2)$
and the functional $(x,y) \to \sigma_{\partial f}\big((x,y),(e_{X,2},e_{2})\big)$ is continuous at $(x_{2},y_{2})$, we may assume without loss of generality that
\[
\left\langle (x^{*},y^{*}), ( e_{X,2}, e_{2}) \right\rangle>\sigma_{\partial f}\big((x_2,y_2),( e_{X,2}, e_{2})\big)-  \frac{1}{8}\varepsilon_2^{64}
\]
for every $(u,v)\in V_2$ and $(x^{*},y^{*})\in \partial f\left(u,v\right)$. Then, from the preceding argument, we have the following inequalities
\begin{eqnarray*}
 \left\langle (x^{*},y^{*}), ( e_{X,2}, e_{2}) \right\rangle &>& \sigma_{\partial f}\left((x_2,y_2),( e_{X,2}, e_{2})\right)-  \frac{1}{8}\varepsilon_2^{64}
\\
&\geq &\sup\left\{  \sigma_{\partial f}\left((x,y),( e_{X,2}, e_{2})\right):  (x,y)\in V_2                 \right\}-\frac{1}{16}\varepsilon_2^{32}-  \frac{1}{8}\varepsilon_2^{64}
\\
&\geq &\sup \left\{  \sigma_{\partial f}\left((x,y),( e_{X,2}, e_{2})\right):  (x,y)\in V_2                 \right\}-\frac{1}{8}\varepsilon_2^{32}~~~~\quad \quad~\quad~\quad ~~~~(2.8)~\quad \quad \quad\quad \quad \quad\quad
\end{eqnarray*}
for every $(u,v)\in V_2$ and $(x^{*},y^{*})\in \partial f\left(u,v\right)$. Moreover,
Player $A$ may choose
any nonempty open subset $U_3 \subset V_2$. Hence
we may assume that
\[
\sup \left\{  \|(x^{*},y^*)\|\in \mathbb R :   (x^{*},y^*)\in \partial f\left(U_{3}\right)      \right\}>0.
\]

\textbf{Step 2.} Since
$ p_2=   \mu_{\overline{\mathrm{co}}\left(S'_{2}\left(X\times \mathbb R \right)\right)}$ is  continuous,
we
define the set $C_2$,   where
\[
C_2=\left\{ \left( \alpha e_{X,2}, \alpha e_{2}\right):  0\leq \alpha \leq  1+\varepsilon_2^{16}     \right\} \cup \left\{ (x,y)\in X\times \mathbb R :    p_2\left(x,y\right)\leq \frac{1}{512^{3}}      \right\} .
\]
Therefore, by the definition of $C_2$, we define the   functional $\mu_{C_2}$, where
\[
\mu_{C_2 }\left(x,y\right)= \inf \left\{  \lambda\in \mathbb R^{+} :  \frac{1}{\lambda}\left(x,y\right)\in   C_2           \right\}
\]
for every $(x,y)\in X\times \mathbb R$. Moreover, we define the functional $p_3'$, where
\[
p_{3}'\left(x,y\right)=p_2\left(x,y\right)+16\varepsilon_{2} \cdot\mu_{ C_2}(x,y)~\quad~~~~~\mathrm{for}~~~~~~~~~\mathrm{every}~~~~~~~\quad~~(x,y)\in X\times \mathbb R.
\]
Therefore, since $p_2\left(e_{X,2}, e_{2}\right)=1$, there exists a real number $\alpha_2\in (0,1) $ such that
\[
p_2\left(\alpha_2 e_{X,2},\alpha_2 e_{2}\right)+ 16 \varepsilon_{2}\cdot \mu_{C_2}\left(\alpha_2 e_{X,2},\alpha_2 e_{2}\right)= p_3'\left(\alpha_2 e_{X,2},\alpha_2 e_{2}\right)=1.
\]
Therefore, by the definition of $C_2$ and the preceding inequalities, we have
\[
p_2\left(\alpha_2 e_{X,2},\alpha_2 e_{2}\right)+ 16 \varepsilon_{2}\cdot p_2\left(\alpha_2 e_{X,2},\alpha_2 e_{2}\right)\geq p_3'\left(\alpha_2 e_{X,2},\alpha_2 e_{2}\right)=1.
\]
Noting that $\alpha_2\in (0,1) $ and $p_2\left(e_{X,2}, e_{2}\right)=1$, from the above inequalities,
we have
\[
1\geq\alpha_2=\alpha_2\cdot p_2\left( e_{X,2}, e_{2}\right)    =p_2\left(\alpha_2 e_{X,2},\alpha_2 e_{2}\right)\geq\frac{1}{1+16\varepsilon_{2}}.
\]
Moreover, we define the set $S'_{3}\left(X\times \mathbb R\right)$ and define a real number $s_{3}'$, where
\[
S'_{3}\left(X\times \mathbb R\right)=\left\{(x, y)\in X\times \mathbb R: p_{3}'\left(x,y\right)= p_2\left(x,y\right)+16\varepsilon_2 \cdot \mu_{ C_2 }\left(x,y\right)=1\right\}
\]
and
\[
s'_{3}=\sup\left\{ \sigma_{\partial f} \left((x,y),(e_{X},e)\right): ((x,y),(e_{X},e))\in U_{3}\times S'_{3}\left(X\times \mathbb R\right) \right\}>0.
\]
We claim that $s_3'=\sup \left\{  \sigma_{\partial f}\left((x,y),( \alpha_2e_{X,2}, \alpha_2e_{2})\right):  (x,y)\in U_3            \right  \}$.
In fact,
noting that $1\geq \alpha_{2}\geq (1+16\varepsilon_2)^{-1}$,  $p_2 \left( ( \alpha_1 e_{X,1}, \alpha_1 e_{1})-  ( e_{X,2}, e_{2})   \right  )<\eta_1/50$ and $\eta_1\in  (0,$ $\varepsilon_1^{4}/512^{6})$, from
$\partial f\left(U_{1}\right) \subset B_{1}\left(X^{*}\times \mathbb R\right)$ and $\varepsilon_2= \varepsilon_1/ 128$, we obtain that
\[
\quad \quad ~\sup \left\{  \sigma_{\partial f}\left((x,y),( \alpha_2e_{X,2}, \alpha_2e_{2})\right):  (x,y)\in U_3              \right\}\quad \quad \quad \quad \quad\quad \quad \quad \quad \quad \quad \quad \quad \quad \quad \quad\quad \quad \quad \quad \quad \quad
\]
\[
=~\alpha_2\sup \left\{  \sigma_{\partial f}\left((x,y),( e_{X,2}, e_{2})\right):  (x,y)\in U_3              \right\}\quad \quad \quad \quad \quad\quad \quad \quad \quad \quad \quad \quad \quad \quad \quad \quad\quad \quad \quad \quad \quad \quad
\]
\[
\geq~ \alpha_2\sup \left\{  \sigma_{\partial f}\left((x,y),( \alpha_1 e_{X,1}, \alpha_1 e_{1})\right):  (x,y)\in U_3              \right\}~-~ \big\| ( \alpha_1 e_{X,1}, \alpha_1 e_{1})-  ( e_{X,2}, e_{2})     \big\|\quad \quad \quad \quad \quad \quad\quad \quad \quad \quad \quad \quad
\]
\[
\geq~ \frac{1}{1+16\varepsilon_2}s_2~-~p_2 \big( ( \alpha_1 e_{X,1}, \alpha_1 e_{1})-  ( e_{X,2}, e_{2})   \big ) \quad \quad \quad \quad \quad \quad \quad \quad \quad  \quad \quad  \quad \quad \quad \quad \quad \quad \quad \quad \quad \quad \quad \quad \quad \quad \quad \quad  \quad \quad \quad \quad \quad \quad \quad \quad \quad
\]
\[
\geq~     \frac{1}{1+16\varepsilon_2}s_2~-~\frac{1}{50}\left( \frac{\varepsilon_1^{4}}{512^{6}} \right)  \quad \quad \quad \quad \quad \quad \quad \quad \quad  \quad \quad \quad \quad \quad \quad \quad \quad \quad  \quad \quad \quad \quad \quad \quad \quad \quad \quad \quad \quad \quad \quad \quad \quad \quad \quad \quad
\]
\[
>~      \frac{1}{1+16 (512)^{3}\varepsilon_2}s_2. \quad \quad \quad \quad \quad \quad \quad \quad \quad \quad \quad \quad\quad \quad \quad \quad \quad \quad \quad \quad \quad \quad \quad \quad \quad \quad \quad \quad \quad \quad \quad \quad \quad \quad \quad \quad \quad \quad \quad \quad \quad \quad
\]
Pick a point $(e_X,e)\in X\times\mathbb R$ with $(e_X,e)\neq (\alpha_2 e_{X,2},\alpha_2 e_2)$ such that $p_3'(e_X,e)=1$.
Then $\mu_{C_2}(e_X,e)=512^3\, p_2(e_X,e)$. Therefore, by the definition of $p_3'$, we have
\[
1= p_3'\left(e_X,e\right)=p_2\left(e_X,e\right) + 16\varepsilon_2  \mu_{ C_2 }(e_X,e)=p_2\left(e_X,e\right) +16 \varepsilon_2 \left(512\right)^{3} p_2\left(e_X,e\right).
\]
Hence we obtain that $ p_2\left(e_X,e\right)= [1+16 \left(512\right)^{3}\varepsilon_2]^{-1}$. Therefore, by the definition of $s_2$ and the equality $ p_2\left(e_X,e\right)= [1+16 \left(512\right)^{3}\varepsilon_2]^{-1}$, we obtain that
\begin{eqnarray*}
s_3'&\geq&\sup \left\{  \sigma_{\partial f}\left((x,y),( \alpha_2e_{X,2}, \alpha_2e_{2})\right):  (x,y)\in U_3              \right\}
\\
&\geq & \frac{1}{1+16 (512)^{3}\varepsilon_2}s_2
\\
&\geq &\sup \left\{  \sigma_{\partial f}\left((x,y),( e_{X}, e)\right):  (x,y)\in U_3              \right\}.
\end{eqnarray*}
Therefore, by the arbitrariness of $(e_X,e)$ and the above inequalities,
we get that
\[
s_3'=\sup \left\{  \sigma_{\partial f}\left((x,y),( \alpha_2e_{X,2}, \alpha_2e_{2})\right):  (x,y)\in U_3              \right\}
\]
holds.
Moreover, we define the bounded closed set $D_2'$, where
\[
D_2'=\left\{ \left( \alpha e_{X,2}, \alpha e_{2}\right):  0\leq \alpha \leq  1+\varepsilon_2^{16}     \right\} \cup \left\{ (x,y)\in X\times\mathbb R:    p_2\left(x,y\right)\leq 1    \right\} .
\]
Therefore, associated with $D_2'$, we define the noncontinuous functional $\mu_{D_2'}$, where
\[
\mu_{D_2'}\left(x,y\right)= \inf \left\{  \lambda\in\mathbb R^{+} :  \frac{1}{\lambda}\left(x,y\right)\in  D_2'           \right\}
\]
for every $(x,y)\in X\times\mathbb R$.
Moreover, we define the functional $p_3''$, where
\[
p_{3}''\left(x,y\right)=p_2\left(x,y\right)+16\varepsilon_{2} \cdot\mu_{D_2'}(x,y)~\quad~~~~~\mathrm{for}~~~~~~~~~\mathrm{every}~~~~~~~\quad~~(x,y)\in X\times\mathbb R.
\]
Define the closed convex set $D_2= \overline{\mathrm{co}}\left\{ (e_{X},e): p_3'' \left(e_{X},e\right)=1        \right\}$ and the Minkowski functional $\mu_{D_2}$.
Let $A_2=\overline{\mathrm{co}}(C_2)$. Then we get that $A_2$ is a closed convex subset of $X\times\mathbb R$. Hence we define the Minkowski functional $\mu_{A_2}$. Define the  two sets
\[
S_{3}\left(X\times \mathbb R\right)=\left\{(x, y)\in X\times \mathbb R: p_2\left(x,y\right)+16\varepsilon_2 \cdot \mu_{A_2 }\left(x,y\right)=1\right\}
\]
and
\[
B_{3}\left(X\times \mathbb R\right)=\left\{(x, y)\in X\times \mathbb R: p_2\left(x,y\right)+16\varepsilon_2 \cdot \mu_{A_2 }\left(x,y\right)\leq 1\right\}.
\]
Then, by the definition of $S_{3}\left(X\times \mathbb R\right)$, we conclude that
$s_{3}>0$, where
\[
s_{3}=\sup\left\{ \sigma_{\partial f}\left ((x,y),(e_{X},e)\right): \left((x,y),(e_{X},e)\right)\in U_{3}\times S_{3}\left(X\times \mathbb R\right) \right\}.
\]
Define $p_3(x, y)= p_2(x, y)+16\varepsilon_2 \cdot \mu_{A_2 }(x, y)$ for every $(x, y)\in X\times \mathbb R$. Then, by the definitions of $p_3$ and $p_3'$, we have
$p_3\left(\alpha_{2}e_{X,2},\alpha_{2}e_{2}\right)= 1$ and $p_3'\left(\alpha_{2}e_{X,2},\alpha_{2}e_{2}\right) =1$.
We  next  prove that
\begin{eqnarray*}
 \quad \quad \quad \quad\quad \quad s_{3}&=&\sup\left\{ \sigma_{\partial f} \left((x,y),(e_{X},e)\right): ((x,y),(e_{X},e))\right.
\\
&\in& U_{3}\times B \left( \left(\alpha_{2}e_{X,2},\alpha_{2}e_{2}\right),25\varepsilon_{2} \right) \cap B_{3}\left(X\times \mathbb R\right) \}  \quad \quad \quad\quad \quad ~(2.9)
\end{eqnarray*}
holds. We proceed to prove formula (2.9), with two-dimensional spatial geometry playing a crucial role throughout the proof.

\par
In fact, there exists a real number $c_2\in(0,+\infty)$ such that
$c_2\cdot\mu_{A_2}(\alpha_2 e_{X,2},\alpha_2 e_2)$ $=p_3''(\alpha_2 e_{X,2},\alpha_2 e_2)$.
Therefore, by the definition of $p_3''$, we have $c_2\in(1,1+1/256)$.
Moreover, we pick a point $(e_{X,0},e_0)\notin\big\{\lambda\cdot(e_{X,2},e_2):\lambda\in\mathbb R\big\}$. Then we define the two-dimensional subspace $M_0$ of $X\times\mathbb R$, where
\[
M_0= \left\{  \lambda \left( e_{X,0},e_{0}\right)+ \xi    \left( e_{X,2},e_{2}\right):   \lambda\in \mathbb R,~\xi  \in \mathbb R                    \right\}\subset X\times \mathbb R.
\]
Since $M_0$ is a two-dimensional  subspace of $X\times \mathbb R$,  by Lemma 2.4, there exists a point $( u_{X,0},u_{0})$ with $p_3''\left( u_{X,0},u_{0}\right)=1$ such that
\begin{eqnarray*}
&&\left[( \alpha_{2}e_{X,2},\alpha_{2}e_{2}), ( u_{X,0},u_{0})\right]
\\
&=&\overline{S''_{3}}\left(X\times \mathbb R\right)\cap \left\{ (1- \xi) \left ( \alpha_{2}e_{X,2},\alpha_{2}e_{2}\right)+ \xi\left( u_{X,0},u_{0}\right):  \xi \in \mathbb R    \right\},
\end{eqnarray*}
where
$
\overline{S_3''}\left(X\times\mathbb R\right)=\{(x,y)\in X\times\mathbb R:\mu_{D_2}(x,y)=1\}
$
and
\[
\big[(\alpha_2 e_{X,2},\alpha_2 e_2),(u_{X,0},u_0)\big]= \left\{ \lambda \left(\alpha_2 e_{X,2},\alpha_2 e_2\right)+ (1-\lambda)  \left(u_{X,0},u_0\right) :  \lambda\in [0,1]                     \right\}.
\]
Since $M_0$ is a two-dimensional subspace of $X\times \mathbb R$, by Lemma 2.4, there exists a point $(v_{X,0},v_0)\in X\times\mathbb R$ with $c_2\mu_{A_2}(v_{X,0},v_0)=1$ such that
\begin{eqnarray*}
&&\left[( \alpha_{2}e_{X,2},\alpha_{2}e_{2}), ( v_{X,0},v_{0})\right]
\\
&=& \left\{ (e_{X},e):  c_2 \mu_{A_2 }(e_{X},e)=1      \right\}\cap \left\{  (1- \xi)\left( \lambda_{2}e_{X,2},\lambda_{2}e_{2}\right)+ \xi \left( v_{X,0},v_{0}\right):  \xi \in \mathbb R     \right\}.
\end{eqnarray*}
Since $M_0$ is a two-dimensional subspace of $X\times\mathbb R$, by Lemma 2.4, there exists a point $(v_{X,0}',v_0')\in\{(\alpha u_{X,0},\alpha u_0):\alpha\in\mathbb R^+\}$ such that
\[
\left( v'_{X,0},v'_{0}\right)=  k \left( u_{X,0},u_{0}\right)\in \left\{(1- \xi) \left ( \alpha_{2}e_{X,2},\alpha_{2}e_{2}\right)+\xi  \left( v_{X,0},v_{0}\right):  \xi \in \mathbb R     \right\}.            \eqno     (2.10)
\]
We shall split the proof of  formula (2.9) into two cases. The argument for Case I
\\
is further broken down into three parts: (a), (b) and (c).

\par Case I. (a) Let  $\|( \alpha_{2}e_{X,2},\alpha_{2}e_{2})- ( v'_{X,0},v'_{0})\|\geq \varepsilon_{2}$ and $\|( \alpha_{2}e_{X,2},\alpha_{2}e_{2})- ( u_{X,0},u_{0})\|$ $\geq \varepsilon_{2}$.
Since $M_0$ is a two-dimensional subspace of $X\times\mathbb R$, we obtain that for a sufficiently small $\xi\in (0,1)$, there exists a function $h\left(\xi\right)\in (0,1)$ and a real number
$\alpha_{\xi}\in (1,1+ $ $1/256)$  such that
\[
\alpha_{\xi} \left[(1-h(\xi))( \alpha_{2}e_{X,2},\alpha_{2}e_{2})+h(\xi)( v'_{X,0},v'_{0})\right]=(1-\xi)( \alpha_{2}e_{X,2},\alpha_{2}e_{2})+ \xi( u_{X,0},u_{0}).
\]
Since $\mu_{D_2 } \left[(1-\xi)\left( \alpha_{2}e_{X,2},\alpha_{2}e_{2}\right)+ \xi\left( u_{X,0},u_{0}\right)\right]\equiv 1$ for all $\xi\in [0,1]$, we obtain the following equalities
\begin{eqnarray*}
&&c_2\cdot \mu_{A_2 } \left[(1-\xi)\left( \alpha_{2}e_{X,2},\alpha_{2}e_{2}\right)+ \xi\left( u_{X,0},u_{0}\right)\right]
\\
&=&\frac{\mu_{ D_2 } \left[(1-\xi)\left( \alpha_{2}e_{X,2},\alpha_{2}e_{2}\right)+ \xi\left( u_{X,0},u_{0}\right)\right]}{\mu_{D_2} \left[\left(1-h(\xi)\right)\cdot( \alpha_{2}e_{X,2},\alpha_{2}e_{2})+ h(\xi)\left( v'_{X,0},v'_{0}\right)\right]}
\\
&=& \frac{1}{\mu_{D_2} \left[\left(1-h(\xi)\right)\cdot( \alpha_{2}e_{X,2},\alpha_{2}e_{2})+ h(\xi)\left( v'_{X,0},v'_{0}\right)\right]}.
\end{eqnarray*}
Moreover, since $\alpha_{2}\left( e_{X,2},e_{2}\right)\neq( u_{X,0},u_{0})$, $( v'_{X,0},v'_{0})\in \{   \alpha\left( u_{X,0},u_{0}\right): \alpha\in \mathbb R \}$ and
\[
(\alpha_{\xi}\cdot (1-h(\xi))-(1-\xi))\cdot( \alpha_{2}e_{X,2},\alpha_{2}e_{2})=\xi\left( u_{X,0},u_{0}\right)-\alpha_{\xi}h(\xi)\cdot\left( v'_{X,0},v'_{0}\right),
\]
it is easy to see that $\alpha_{\xi}\big(1-h\left(\xi\right)\big)=(1-\xi)$. Therefore, since $h\left(0\right)=0$ and $\alpha_{\xi}\in $ $\left(1,1+ 1/256\right)$, we obtain the following inequalities
\[
\mathop {\lim\sup }\limits_{\xi \to 0^{+} }\left(\frac{h(\xi)-h(0)}{\xi-0}\right)=\mathop {\lim\sup }\limits_{\xi \to 0^{+} }\frac{1}{\alpha_{\xi}}\left( \frac{\alpha_{\xi}-1+\xi}{\xi}\right)\geq\frac{128}{129}.  \eqno     ~~~~~~~~~~~~~~~~~~~~~~~~~~(2.11)
\]
Since $\mu_{A_2 }$ is a continuous convex function, we get that $\mu_{A_2 }$ possesses a right-hand derivative.
Therefore, by the chain rule, we obtain that
\[
\quad \quad\quad \quad \mathrm{\frac{d^{+}}{d\xi}}\mu_{A_2 } \big[(1-\xi)\left( \alpha_{2}e_{X,2},\alpha_{2}e_{2}\right)+\xi \left( u_{X,0},u_{0}\right)\big]\quad \quad \quad\quad \quad \quad \quad \quad \quad\quad \quad \quad \quad \quad \quad\quad \quad \quad
\]
\[
\quad \quad=~  \left({\frac{\mathrm{d^{+}}}{\mathrm{d} h(\xi)}}\frac{1}{c_2\mu_{D_2} \left[(1-h(\xi))\left( \alpha_{2}e_{X,2},\alpha_{2}e_{2}\right)+ h(\xi)\left( v'_{X,0},v'_{0}\right)\right]}\right)  \left(\mathrm{\frac{d^{+}}{dt}}h(\xi) \right) \quad \quad \quad\quad \quad \quad\quad \quad \quad\quad \quad \quad \quad \quad \quad\quad \quad \quad \quad \quad \quad\quad \quad \quad
\]
\[
\quad \quad=~\left(\frac{1}{c_2\mu_{D_2} \left[(1-h(\xi))\left( \alpha_{2}e_{X,2},\alpha_{2}e_{2}\right)+ h(\xi)\left( v'_{X,0},v'_{0}\right)\right]}\right)^{2}\quad \quad \quad\quad \quad \quad\quad \quad \quad\quad \quad \quad\quad \quad \quad\quad \quad \quad
\]
\[
\quad \quad\quad \quad\cdot \left\langle     x^{*}|_{h(\xi)} ,  \left ( \alpha_{2}e_{X,2},\alpha_{2}e_{2}\right)-  \left( v'_{X,0},v'_{0}\right)           \right\rangle\cdot \left(\mathrm{\frac{d^{+}}{dt}}h(\xi) \right) \quad \quad \quad \quad \quad \quad \quad \quad \quad\quad \quad \quad \quad \quad \quad\quad \quad \quad \quad \quad \quad\quad \quad \quad
\]
for a sufficiently small $\xi\in  [0,1)$, where $\mathrm{{d^{+}}/{d\xi}}$ denotes right-hand derivative and
\[
 x^{*}|_{h(\xi)}=d_Gc_2\mu_{D_2} \left[(1-h(\xi))( \alpha_{2}e_{X,2},\alpha_{2}e_{2})+ h(\xi)\left( v'_{X,0},v'_{0}\right)\right].
\]
Therefore, by the definitions of $A_2$ and $x^{*}|_{h(\xi)}$, we have the following inequality
\[
\left\langle     x^{*}|_{h(0)} ,   ( \alpha_{2}e_{X,2},\alpha_{2}e_{2})-  ( v_{X,0},v_{0})        \right \rangle \geq \frac{7}{8}.
\]
Moreover, since $\left\|( \alpha_{2}e_{X,2},\alpha_{2}e_{2})- \left( v'_{X,0},v'_{0}\right)\right\|\geq \varepsilon_{2}$, by  (2.10), we conclude that
\[
\left\langle     x^{*}|_{h(0)} ,   ( \alpha_{2}e_{X,2},\alpha_{2}e_{2})-  ( v'_{X,0},v'_{0})        \right \rangle \geq \frac{1}{2}\varepsilon_0.
\]
Therefore, since $c_2\in (1, 1+ 1/256)$, we obtain the following inequalities
\[
\quad\quad \quad \left.\mathrm{\frac{d^{+}}{d\xi}}\mu_{  A_2 } \left[(1-\xi)\left( \alpha_{2}e_{X,2},\alpha_{2}e_{2}\right)+ \xi\left( u_{X,0},u_{0}\right)\right]\right|_{\xi=0}\quad \quad \quad\quad \quad \quad \quad \quad \quad\quad \quad \quad \quad \quad \quad\quad \quad \quad
\]
\[
\quad=~\left(\frac{1}{c_2\mu_{D_2} \left[( \alpha_{2}e_{X,2},\alpha_{2}e_{2})\right]}\right)^{2}\left\langle     x^{*}|_{h(0)} ,   ( \alpha_{2}e_{X,2},\alpha_{2}e_{2})-  \left( v'_{X,0},v'_{0}\right)           \right\rangle \left(\left.\mathrm{\frac{d^{+}}{dt}}h(\xi) \right|_{\xi=0}\right)\quad \quad\quad \quad\quad
\]
\[
\quad\geq~ \left( \frac{128}{129} \cdot \frac{1}{c_2\mu_{D_2} \left[( \alpha_{2}e_{X,2},\alpha_{2}e_{2})\right]}\right)^{2} \left\langle     x^{*}|_{h(0)} ,   ( \alpha_{2}e_{X,2},\alpha_{2}e_{2})-  \left( v'_{X,0},v'_{0}\right)           \right\rangle\quad \quad\quad \quad\quad
\]
\[
\quad\geq~ \frac{1}{2}\varepsilon_{2}\left( \frac{128}{129} \cdot \frac{1}{c_2\mu_{D_2} \left[( \alpha_{2}e_{X,2},\alpha_{2}e_{2})\right]}\right)^{2}  \geq \frac{1}{4}\varepsilon_{2}.\quad \quad\quad \quad\quad\quad \quad\quad \quad\quad\quad \quad\quad \quad\quad\quad \quad\quad \quad\quad\quad \quad\quad \quad\quad\quad \quad\quad \quad\quad\quad \quad\quad \quad\quad\quad \quad\quad \quad\quad
\]
Since $A_2 $ is a bounded closed convex subset of $X\times\mathbb  R$, we define the continuous  convex functional $f$, where
\[
f\left(\xi\right)=\mu_{ A_2 } \left[(1-\xi)\left( \alpha_{2}e_{X,2},\alpha_{2}e_{2}\right)+ \xi\left( u_{X,0},u_{0}\right)\right]~\quad~~~~~~~~~\mathrm{for}~~~~~~~~~\mathrm{every}~~~~~~~~~\quad~\xi\in [0,+\infty).
\]
Since $f$ is a continuous convex function, by the definitions of $f$ and $A_2$, it is easy to see that $f$ is increasing on the interval $[0,1]$. Hence we define the continuous convex function $g$
such that $g(\xi)= g_1(\xi)+ g_2(\xi)$, where
\[
g_1\left(\xi\right)=p_2\left[(1-\xi ) \left( \alpha_{2}e_{X,2},\alpha_{2}e_{2}\right)+ \xi\left( u_{X,0},u_{0}\right)\right]
\]
and
\[
g_2\left(\xi\right)=16\varepsilon_2 \cdot \mu_{ A_2} \left[(1-\xi )\left( \alpha_{2}e_{X,2},\alpha_{2}e_{2}\right)+ \xi\left( u_{X,0},u_{0}\right)\right]
\]
for all $\xi\in [0,+\infty)$. Since the convex function $f$ is increasing on $[0,1]$, we obtain that $g_2$ is increasing on the interval $[0,1]$.
We next prove that
\[
g\left(\xi\right)\geq \mu_{ D_2} \big[(1-\xi )\left( \alpha_{2}e_{X,2},\alpha_{2}e_{2}\right)+ \xi\left( u_{X,0},u_{0}\right)\big] + 3\varepsilon_{2}^{4}~~\quad~~~~~~~~~\mathrm{for}~~~~~~~~~\mathrm{all}~~~~~~~~~\quad~~\xi\in \left[ \varepsilon_2^{4}+ 2\varepsilon_2^{2},1\right]
\]
and $g$ is increasing on the interval $\left[ 2\varepsilon_2^{4},+\infty\right)$.
In fact, since $g_1$ is a convex function and $p_2(\alpha_2 e_{X,2},\alpha_2 e_2)>p_2(u_{X,0},u_0)$, it is easy to see that $g_1$
is first decreasing and then increasing, or is decreasing on the interval $[0,1]$.

\par Let $g_1'$ denote the right-hand derivative of $g_1$.
We next divide the proof of (a) into two cases.

\par Case 1. Let $g_1$ be decreasing in the interval $[0,\varepsilon_2^{4}]$. Then
we claim that $g_1'\left(\varepsilon_2^{4}\right)\geq -\varepsilon_2^{4}/4$. Suppose that $g_1'(\varepsilon_2^{4})<-\varepsilon_2^{4}/4$.
Then, by the definition of $g_1$, we have
\[
\left|g_1(\xi)-p_2\left( u_{X,0},u_{0}\right)\right |\leq4\varepsilon_2^{12}~\quad~~~~~~~~~\mathrm{and}~~~~~~~~~\quad~\left|g_1(\xi)-p_2 \left( \alpha_{2}e_{X,2},\alpha_{2}e_{2}\right)\right |\leq4\varepsilon_2^{12}
\]
for every $\xi\in [0,1]$.
Therefore, from the preceding inequalities, we obtain that
\begin{eqnarray*}
&&4\varepsilon_2^{12} \geq  \left|g_1(\varepsilon_2^{4})-g_1(0)\right|
\\
&=&\left|p_2\left( \alpha_{2}e_{X,2},\alpha_{2}e_{2}\right)-p_2\left((1-\varepsilon_2^{4} ) \left( \alpha_{2}e_{X,2},\alpha_{2}e_{2}\right)+ \varepsilon_2^{4}\left( u_{X,0},u_{0}\right)\right)\right|.
\end{eqnarray*}
However, since the convex function $g_1$ is decreasing on $[0,\varepsilon_2^{4}]$, from the inequality $g_1'(\varepsilon_2^{4})< -\varepsilon_2^{4}/4$, we have the following inequalities
\[
4\varepsilon_2^{12} \geq\left|g_1(\varepsilon_2^{4})-g_1(0)\right|\geq \left|g_1'(\varepsilon_2^{4})\right|\cdot\left|\varepsilon_2^{4}-0\right|\geq\frac{1}{4}\varepsilon_2^{4} \left|\varepsilon_2^{4}-0\right|= \frac{1}{4}\varepsilon_2^{8},
\]
a contradiction. Let $g_2'$ denote the right-hand  derivative of $g_2$. Therefore, from the preceding argument and the definition of $g_2$, we conclude that
\[
g_2' \left(  0\right)= 16 \varepsilon_2 \left( \left.\mathrm{\frac{d^{+}}{d\xi}}\mu_{A_2 }\left[(1-\xi)\left( \alpha_{2}e_{X,2},\alpha_{2}e_{2}\right)+ \xi\left( u_{X,0},u_{0}\right)\right]\right|_{\xi=0}\right )\geq 4\varepsilon_2^{2} .        \eqno                            (2.12)
\]
Since the convex function $g_2$ is increasing on $[0,1]$, from the preceding inequalities, we obtain that $g_2' \left(\xi\right)\geq g_2' \left(0\right)  >4\varepsilon_2^{2} $ whenever $\xi\geq\varepsilon_2^{4}$.
Moreover,  since  $g_1'\left(\varepsilon_2^{4}\right)\geq -\varepsilon_2^{4}/4$ and $g_2'\left(\varepsilon_2^{4}\right)\geq 4\varepsilon_2^{2} $, we conclude that
\[
g\left(\xi\right)-g\left(\varepsilon_2^{4}\right)\geq [g_1'(\varepsilon_2^{4})+g_2'(\varepsilon_2^{4})] \cdot(\xi-\varepsilon_2^{4})   \geq  \left [4\varepsilon_2^{2}-\frac{1}{4} \varepsilon_2^{4} \right]   \cdot(\xi-\varepsilon_2^{4})        >0
\]
whenever $\xi\geq\varepsilon_2^{4}$. It follows that $g$ is increasing on $\left[ \varepsilon_2^{4},+\infty\right)$.
Pick $\xi=\varepsilon_2^{2}+\varepsilon_2^{4} $. Then, from the preceding inequalities, we conclude that
\[
g\left(\varepsilon_2^{2}+\varepsilon_2^{4})-g(\varepsilon_2^{4}\right)\geq \big[g_1'(\varepsilon_2^{4})+g_2'(\varepsilon_2^{4})\big] \cdot\left(\varepsilon_2^{2}+\varepsilon_2^{4}-\varepsilon_2^{4}\right)   \geq  \left [4\varepsilon_2^{2}-\frac{1}{4} \varepsilon_2^{4} \right]  \varepsilon_2^{2} \geq 3 \varepsilon_2^{4}.
\]
Moreover, we define the continuous convex function $f_1$, where
\[
f_1(\xi)=(1+16\varepsilon_2)\cdot p_2\left[(1-\xi ) \left( \alpha_{2}e_{X,2},\alpha_{2}e_{2}\right)+ \xi\left( u_{X,0},u_{0}\right)\right]~\quad~~~~~~~~~\mathrm{for}~~~~~~~~~\mathrm{every}~~~~~~~~~\quad~\xi\in [0,1].
\]
Therefore, by the definition of $g$, we conclude that $g\left(\varepsilon_2^{4}\right)\geq f_1\left(\varepsilon_2^{4}\right) - \varepsilon_2^{10}$.
Moreover, since $g_1'\left(\varepsilon_2^{4}\right)\geq -\varepsilon_2^{4}/4$  and $g_2'\left(\varepsilon_2^{4}\right)\geq 4\varepsilon_2^{2} $, by the definition of $f_1$, we conclude that $g\left(\xi\right)\geq f_1\left(\xi\right)$ for all $\xi\in \left[ \varepsilon_2^{4}+ \varepsilon_2^{2},1\right]$.
Then
\[
g\left(\xi\right)\geq f_1\left(\xi\right)\geq \mu_{ D_2 } \left[(1-\xi )\left( \alpha_{2}e_{X,2},\alpha_{2}e_{2}\right)+ \xi\left( u_{X,0},u_{0}\right)\right]~~~\mathrm{for}~~~~~~~~~\mathrm{all}~~~~~~~~~\quad~~\xi\in \left[ \varepsilon_2^{4}+ \varepsilon_2^{2},1\right].
\]
Noting that $g\left(\varepsilon_2^{4}\right)\geq f_1(\varepsilon_2^{4}) - \varepsilon_2^{10}$ and $ g\left(\varepsilon_2^{2}+\varepsilon_2^{4}\right)-g\left(\varepsilon_2^{4}\right)\geq 3 \varepsilon_2^{4}$, by the definition of $f_1$, we obtain  the following inequalities
\begin{eqnarray*}
g\left(\varepsilon_2^{2}+\varepsilon_2^{4}\right)&\geq& g\left(\varepsilon_2^{4}\right)+3 \varepsilon_2^{4}   \geq f_1\left(\varepsilon_2^{4}\right)+3 \varepsilon_2^{4}-\varepsilon_2^{10}\geq f_1\left(\varepsilon_2^{4}\right)
\\
&\geq& \mu_{ D_2 } \left[(1-\varepsilon_2^{4} )\left( \alpha_{2}e_{X,2},\alpha_{2}e_{2}\right)+ \varepsilon_2^{4}\left( u_{X,0},u_{0}\right)\right].
\end{eqnarray*}
Since $g_2'\left(\varepsilon_2^{4}\right)\geq 4\varepsilon_2^{2} $ and  $g_1'\left(\varepsilon_2^{4}\right)\geq -\varepsilon_2^{4}/4$,
by (2.12), we conclude that
\[
g\left(\varepsilon_2^{4}+ 2\varepsilon_2^{2}\right)-g\left(\varepsilon_2^{4}+\varepsilon_2^{2}\right)\geq \left(g_1'(\varepsilon_2^{4})+ g_2'(\varepsilon_2^{4})\right)\cdot\left(\varepsilon_2^{4}+ 2\varepsilon_2^{2}- \varepsilon_2^{4} -\varepsilon_2^{2}\right)\geq 3\varepsilon_{2}^{4}.
\]
Since $g$ is increasing on $\left[ \varepsilon_2^{4}+ \varepsilon_2^{2},+\infty\right)$, we obtain that
$g\left(\xi\right)-g\left(\varepsilon_2^{4}+\varepsilon_2^{2} \right)\geq 3\varepsilon_{2}^{4} $ for every $\xi\in [\varepsilon_2^{4}+ 2\varepsilon_2^{2},1]$. Moreover,
by the identity
\[
\mu_{ D_2 } \big[(1-\xi )\left( \alpha_{2}e_{X,2},\alpha_{2}e_{2}\right)+ \xi\left( u_{X,0},u_{0}\right)\big]\equiv 1 ~\quad~~~~~~~~~\mathrm{for}~~~~~~~~~\mathrm{every}~~~~~~~~~\quad~~~\xi\in [0,1],
\]
we obtain that $g\left(\varepsilon_2^{2}+\varepsilon_2^{4}\right)\geq \mu_{ D_2 } \left[(1-\xi )( \alpha_{2}e_{X,2},\alpha_{2}e_{2})+ \xi( u_{X,0},u_{0})\right]$ for every $\xi\in \left[\varepsilon_2^{2}+\varepsilon_2^{4},1\right]$.
Therefore, from the inequality $g\left(\xi\right)-g\left(\varepsilon_2^{2}+\varepsilon_2^{4} \right)\geq 3\varepsilon_{2}^{4}$, we obtain  the following inequalities
\begin{eqnarray*}
g\left(\xi\right)&\geq& g\left(\varepsilon_2^{2}+\varepsilon_2^{4} \right)+ 3\varepsilon_{2}^{4}
\\
&\geq&  \mu_{ D_2 } \big[(1-\xi )\left( \alpha_{2}e_{X,2},\alpha_{2}e_{2}\right)+ \xi\left( u_{X,0},u_{0}\right)\big]+3\varepsilon_{2}^{4}
\\
&=&1+3\varepsilon_{2}^{4}
\end{eqnarray*}
for every $\xi\in \left[ \varepsilon_2^{4}+ 2\varepsilon_2^{2},1\right]$. Moreover, we know that $g$ is increasing on $\left[ 2\varepsilon_2^{4},+\infty\right)$.

\par Case 2. Let $g_1$ be first decreasing and then increasing on $[0,\varepsilon_2^{4}]$. Since $g_1$ is convex, similar to the proof of Case 1, we obtain that
\[
g\left(\xi\right)\geq  \mu_{ D_2 } \big[(1-\xi )\left( \alpha_{2}e_{X,2},\alpha_{2}e_{2}\right)+ \xi\left( u_{X,0},u_{0}\right)\big]+3\varepsilon_{2}^{4}=1+3\varepsilon_{2}^{4}.
\]
for every $\xi\in \left[ \varepsilon_2^{4}+ 2\varepsilon_2^{2},1\right]$ and $g$ is increasing on $\left[ 2\varepsilon_2^{4},+\infty\right)$.

\par (b)
Noting that  $p_{3}''\left( u_{X,0},u_{0}\right)=1$  and $( u_{X,0},u_{0}) \notin \{\xi \cdot \left( e_{X,2}, e_{2}\right): \xi\in\mathbb  R \}$, by the definition of $D_2' $, we obtain that
$\mu_{ D_2' }\left( u_{X,0},u_{0}\right) = p_2\left( u_{X,0},u_{0}\right)$. Therefore, by the definition of $p_{3}''$ and the equality $\mu_{ D_2' }\left( u_{X,0},u_{0}\right) = p_2\left( u_{X,0},u_{0}\right)$, we conclude
that
\begin{eqnarray*}
1 &=& p_{3}''\left( u_{X,0},u_{0}\right)
\\
&=&p_2\left( u_{X,0},u_{0}\right)+16\varepsilon_{2} \cdot\mu_{ D_2' }\left( u_{X,0},u_{0}\right)
\\
&=& p_2\left( u_{X,0},u_{0}\right)+16\varepsilon_{2} \cdot p_2\left( u_{X,0},u_{0}\right).
\end{eqnarray*}
Hence we obtain that $ p_2( u_{X,0},u_{0})=1/\left(1+16\varepsilon_{2} \right) $. Therefore, by the definition of $s_2$ and the equality $ p_2\left( u_{X,0},u_{0}\right)=1/(1+16\varepsilon_{2} ) $, we obtain
that
\[
\sigma_{\partial f}\left((x,y),( u_{X,0},u_{0})\right)=p_2\left( u_{X,0},u_{0}\right)\sigma_{\partial f}\left((x,y),   \frac{(u_{X,0},u_{0})}{p_2( u_{X,0},u_{0})} \right)\leq\frac{1}{1+16\varepsilon_{2}}s_2
\]
for every $(x,y)\in U_3 \subset U_2$. Moreover, by the definitions of $D_2' $ and $\mu_{ D_2' }$, we have
$ \left(1+\varepsilon_{2}^{16}\right)  \mu_{ D_2' }\left( e_{X,2},e_{2}\right) = p_2\left( e_{X,2},e_{2}\right)$. Since
$ \left(1+\varepsilon_{2}^{16}\right)  \mu_{ D_2' }\left( e_{X,2},e_{2}\right) = p_2\left( e_{X,2},e_{2}\right)$,
$p_{3}''\left( \alpha_{2}e_{X,2},\alpha_{2}e_{2}\right)=1$ and $p_{2}\left( e_{X,2},e_{2}\right)=1$, we conclude that
\[
~1~=~ p_{3}''\left( \alpha_{2}e_{X,2},\alpha_{2}e_{2}\right)\quad\quad \quad \quad \quad \quad \quad \quad\quad \quad \quad \quad
\]
\[
\quad\quad \quad \quad \quad \quad \quad \quad=~ p_2\left( \alpha_{2}e_{X,2},\alpha_{2}e_{2}\right)~+~16\varepsilon_{2} ~\cdot~\mu_{ D_2' }\left( \alpha_{2}e_{X,2},\alpha_{2}e_{2}\right)\quad\quad \quad\quad\quad \quad \quad \quad \quad \quad \quad\quad\quad \quad \quad \quad \quad \quad \quad\quad\quad \quad \quad \quad \quad \quad \quad\quad \quad \quad \quad
\]
\[
\quad\quad \quad \quad \quad \quad \quad \quad  =~p_2\left( \alpha_{2}e_{X,2},\alpha_{2}e_{2}\right)~+~16\varepsilon_{2}      \left[ \frac{1}{1+\varepsilon_{2}^{16}}    \cdot      p_2\left( \alpha_{2}e_{X,2},\alpha_{2}e_{2}\right)\right] \quad \quad \quad   \quad \quad \quad \quad \quad \quad\quad \quad \quad   \quad \quad \quad \quad \quad \quad \quad \quad   \quad \quad \quad \quad \quad
\]
\[
\quad\quad \quad \quad \quad \quad \quad \quad  =~\alpha_{2}\left[p_2\left( e_{X,2},e_{2}\right)+16\varepsilon_{2}       \cdot \frac{1}{1+\varepsilon_{2}^{16}}    \cdot        p_2\left( e_{X,2},e_{2}\right)\right]\quad\quad \quad \quad   \quad \quad \quad \quad \quad \quad \quad \quad   \quad \quad \quad \quad \quad\quad\quad \quad \quad   \quad \quad
\]
\[
\quad\quad \quad \quad \quad \quad \quad \quad =~ \alpha_{2}\left[1+        \frac{16\varepsilon_{2}}{1+\varepsilon_{2}^{16}}           \right].\quad\quad \quad \quad \quad\quad \quad \quad   \quad \quad   \quad \quad \quad \quad \quad \quad \quad \quad   \quad \quad \quad \quad \quad\quad\quad \quad \quad \quad \quad \quad \quad \quad   \quad \quad
\]
Since $s_1=1$ and $\varepsilon_1\in (0, 1/512^{6})$,  together with $\varepsilon_2= \varepsilon_1/128$, we conclude that $s_2\geq 7/8$. Moreover,
since  $ p_2\left( u_{X,0},u_{0}\right)=1/\left(1+16\varepsilon_{2} \right) $ and $U_3\subset U_2$, using (2.7) together with
$1= \alpha_{2}\big[1+        {16\varepsilon_{2}}/\left({1+\varepsilon_{2}^{16}} \right)    \big]      $, we obtain that
\[
\sigma_{\partial f}\big((x,y),(\alpha_{2}e_{X,2},\alpha_{2}e_{2})\big)- \sigma_{\partial f}\big((x,y),( u_{X,0},u_{0})\big)\quad \quad\quad \quad\quad
\]
\[
 \quad\quad \quad=~ \alpha_{2}\cdot\sigma_{\partial f}\big((x,y),(e_{X,2},e_{2})\big)~-~ \sigma_{\partial f}\big((x,y),( u_{X,0},u_{0})\big)\quad \quad\quad \quad\quad \quad\quad \quad\quad \quad\quad \quad\quad \quad\quad \quad\quad \quad\quad \quad\quad \quad\quad \quad
\]
\[
 \quad\quad \quad\geq~ \left[1+        \frac{16\varepsilon_{2}}{1+\varepsilon_{2}^{16}}           \right]^{-1}\left(1-\frac{1}{16 }\varepsilon_{2}^{32}\right) s_2~-~ \frac{1}{1+16\varepsilon_{2}}s_2\quad \quad\quad \quad \quad\quad \quad\quad\quad \quad\quad \quad\quad \quad\quad \quad
\]
\[
 \quad\quad \quad\geq~ \left( \left[1+        \frac{16\varepsilon_{2}}{1+\varepsilon_{2}^{16}}           \right]^{-1}s_2- \frac{1}{1+16\varepsilon_{2}}s_2 \right)   ~-~\frac{1}{16 }\varepsilon_{2}^{32}\left[1+        \frac{16\varepsilon_{2}}{1+\varepsilon_{2}^{16}}           \right]^{-1}s_2\quad \quad\quad \quad\quad \quad\quad \quad\quad \quad\quad \quad\quad \quad\quad \quad\quad \quad\quad \quad\quad \quad
\]
\[
 \quad\quad \quad\geq ~\varepsilon_{2}^{18}s_2   ~-~\frac{1}{16 }\varepsilon_{2}^{32}\left[1+        \frac{16\varepsilon_{2}}{1+\varepsilon_{2}^{16}}           \right]^{-1}s_2 \geq \varepsilon_{2}^{20} ~~\quad~~~\mathrm{for}~~~~\mathrm{every}~~~\quad~~~ (x,y)\in U_3\subset V_2.\quad \quad\quad \quad \quad\quad \quad \quad\quad \quad\quad \quad\quad \quad\quad\quad\quad \quad\quad\quad
\]

\par (c) Let $L_{0}=\left\{(1-\xi)\left(\alpha_2 e_{X,2},\alpha_2 e_{2}\right)+\xi\left( u_{X,0},u_{0}\right): \xi\geq 0 \right\}$. Then, from the preceding argument, we obtain that
$g$ is increasing on $\left[ \varepsilon_2^{4}+ 2\varepsilon_2^{2},+\infty\right)$. Pick a point
\[
(e_X,e)\in S_{3}\left(X\times R\right)\cap M_0~~~~~\quad~~~~~ \mathrm{with}~~~~~\quad~~~~~~~ (e_X,e)\not\in  B\left((\alpha_2 e_{X,2},\alpha_2 e_{2}),25\varepsilon_2\right).
\]
In part (c), we complete the proof of the following inequality
\[
\sigma_{\partial f}\big((x,y),(\alpha_{2}e_{X,2},\alpha_{2}e_{2})\big)\geq \sigma_{\partial f}\big((x,y),(e_X,e) \big)~~\quad~~~\mathrm{for}~~~~\mathrm{every}~~~\quad~~~ (x,y)\in U_3.
\]
Since $M_0$ is a two-dimensional subspace of $X\times \mathbb  R$  and  $(0,0)  \in \{  \xi\cdot(\alpha_2 e_{X,2},\alpha_2 e_{2}): \xi\leq 1   \}$, we conclude that
\[
(0,0) \in \overline{\mathrm{co} }\left( L_{0}\cup \left\{  \xi\cdot(\alpha_2 e_{X,2},\alpha_2 e_{2}): \xi\leq 1  \right\}\right).~~~
\]
Moreover, since  $M_0$ is a two-dimensional subspace of $X\times \mathbb  R$ and $ (e_{X,2},e_{2})\neq (0,0)$,
there exists a functional $x^{*}|_{M_{0}}\in M_0^{*}$ such that
\[
N\left(x^{*}|_{M_{0}}\right)=\{x\in M_{0}: x^{*}|_{M_{0}}(x)=0\}=\{\xi\cdot(e_{X,2},e_{2})\in M_{0}:\xi\in \mathbb R \}.~~
\]
Moreover, we know that $(e_X,e)\notin N\left(x^{*}|_{M_{0}}\right)$ or $(e_X,e)\in N\left(x^{*}|_{M_{0}}\right)$.

\par Suppose that $(e_X,e)\notin N\left(x^{*}|_{M_{0}}\right)$. Then, since $N\left(x^{*}|_{M_{0}}\right)=\{\xi\cdot(e_{X,2},e_{2})\in M_{0}:\xi\in \mathbb R \}$, we conclude that  $( u_{X,0},u_{0})$ $\notin N\left(x^{*}|_{M_{0}}\right)$.
Since $M_0$ is a two-dimensional subspace of $X\times\mathbb R$, by the fact that $( u_{X,0},u_{0})\notin N\left(x^{*}|_{M_{0}}\right)$ and $(e_X,e)\notin N\left(x^{*}|_{M_{0}}\right)$, we may assume without loss of generality that
\[
\left\langle x^{*}|_{M_{0}}, ( u_{X,0},u_{0})\right\rangle>0 ~~~\quad~~~~~ \mathrm{and}~~~~~\quad~~\left\langle x^{*}|_{M_{0}}, (e_X,e)\right\rangle>0 .~~~~~~~~~~~~~~~~~\eqno~~~~~~~~~~~~~~~~~~~~~~~~~~~~~~~~~~~~  (2.13)
\]
We proceed to verify the formula given below
\[
(e_X,e) \in \overline{\mathrm{co}}\left( L_{0}\cup \left\{  \xi\left(\alpha_2 e_{X,2},\alpha_2 e_{2}\right): \xi\leq 1  \right\}\right).~~~~~~~~~~~~~~~~~~~~~~~~~~\eqno~~~~~~~~~~~~~~~~~~~~~~~~~~~~~~~~~~~~  (2.14)
\]
Otherwise, we obtain that $(e_X,e) \notin \overline{\mathrm{co}}\left( L_{0}\cup \left\{  \xi\left(\alpha_2 e_{X,2},\alpha_2 e_{2}\right): \xi\leq 1  \right\}\right)$.
Since  $M_0$ is a two-dimensional space, by Lemma 2.5,
there exists $\xi_0\in [0,+\infty)$ such that
\[
(1-\xi_0)\left( \alpha_2e_{X,2}, \alpha_2e_{2}\right) +\xi_0\left( u_{X,0},u_{0}\right)\in \left\{   (1-\lambda)\left (0,0\right)+ \lambda\left(e_X,e\right)  : \lambda\in [0,1]    \right\}.~~~~~~~~~~~~~~~\eqno~~~~~~~~~~~~~~~~~~~~~~~~~~~~~~~~~~~~  (2.15)
\]
We claim that $\xi_0\in \left[0, \varepsilon_2^{4}+ 2\varepsilon_2^{2}\right]$. In fact, suppose that $\xi_0\notin \left[0, \varepsilon_2^{4}+ 2\varepsilon_2^{2}\right]$. Then, since
$\xi_0\in [0,+\infty)$, we obtain that
$\xi_0 \in\left (\varepsilon_2^{4}+ 2\varepsilon_2^{2},+\infty\right)$.
Pick a point
\[
( w_{X,0},w_{0})=(1-\xi_0)\left( \alpha_2e_{X,2}, \alpha_2e_{2}\right) +\xi_0\left( u_{X,0},u_{0}\right)\in M_0.
\]
Moreover, from the definition of $g$, we have $g(\xi)\geq 1-\varepsilon_2^{10}$ for every $\xi\in [0,+\infty)$.
Then, from the conclusion of (a) and $\xi_0 \in\left (\varepsilon_2^{4}+ 2\varepsilon_2^{2},+\infty\right)$, it is easy to see that $p_3\left( w_{X,0},w_{0}\right)>1$.
Since $(e_X,e)\in S_{3}\left(X\times \mathbb R\right)\cap M_0$, by (2.15), there exists a real number $\lambda_0\in [0,1]$ such that
\[
( w_{X,0},w_{0})=(1-\xi_0)\left( \alpha_2e_{X,2}, \alpha_2e_{2}\right) +\xi_0\left( u_{X,0},u_{0}\right)=\lambda_0\left( e_X, e\right).
\]
Therefore, from the inequalities $p_3\left( w_{X,0},w_{0}\right)>1$ and $ p_3\left(\lambda_0 e_X,\lambda_0 e\right)\leq1$, we have
\begin{eqnarray*}
1 &<& p_3( w_{X,0},w_{0})
\\
&=& p_3\left( (1-\xi_0)\left( \alpha_2e_{X,2}, \alpha_2e_{2}\right) +\xi_0\left( u_{X,0},u_{0}\right)\right)
\\
&=&  p_3\left(\lambda_0 e_X,\lambda_0 e\right)\leq1,
\end{eqnarray*}
a contradiction. Hence we obtain that $\xi_0\in \left[0, \varepsilon_2^{4}+ 2\varepsilon_2^{2}\right]$. It follows that
\[
p_3  \big(  (1-\xi_0)\left( \alpha_2e_{X,2}, \alpha_2e_{2}\right) +\xi_0\left( u_{X,0},u_{0}\right)\big)\geq 1-4\varepsilon_2^{4}-8\varepsilon_2^{2}.
\]
Then, by (2.15) and the fact that $(e_X,e)\in S_{3}\left(X\times \mathbb R\right)\cap M_0$, we conclude that
\[
\left\|(1-\xi_0)\left( \alpha_2e_{X,2}, \alpha_2e_{2}\right) +\xi_0\left( u_{X,0},u_{0}\right)-(e_X,e)\right\|<4\varepsilon_2
\]
Since $\xi_0\in \big[0, \varepsilon_2^{4}+ 2\varepsilon_2^{2}\big]$, we conclude that $\left\| \left( \alpha_2e_{X,2}, \alpha_2e_{2}\right) -(e_X,e)\right\|< 6\varepsilon_2$. Then
$(e_X,e)\in  B\left((\alpha_2 e_{X,2},\alpha_2 e_{2}),25\varepsilon_2\right)$,  a contradiction.
Hence the formula (2.14) holds.
Define the closed convex subset
$
\mathrm{co}\left\{ (\alpha_2 e_{X,2},\alpha_2 e_{2}),  ( u_{X,0},u_{0}),(0,0)              \right\}
$
of $M_0$. Then
\[
(e_X,e)\in \mathrm{co}\left\{ (\alpha_2 e_{X,2},\alpha_2 e_{2}),  ( u_{X,0},u_{0}),(0,0)              \right\}\subset  M_0
\]
or
\[
(e_X,e)\notin \mathrm{co}\left\{ (\alpha_2 e_{X,2},\alpha_2 e_{2}),  ( u_{X,0},u_{0}),(0,0)              \right\}\subset  M_0.
\]
Suppose that $(e_X,e)\in \mathrm{co}\left\{ (\alpha_2 e_{X,2},\alpha_2 e_{2}),  ( u_{X,0},u_{0}),(0,0)              \right\}$. Then, from
the proof of (b), it follows that
\begin{eqnarray*}
\sigma_{\partial f}\left((x,y),(\alpha_{2}e_{X,2},\alpha_{2}e_{2})\right)>
\sigma_{\partial f}\left((x,y), ( u_{X,0},u_{0})\right)+ \varepsilon_{2}^{20}~\quad~~~\mathrm{whenever}~~~~~\quad~~~ (x,y)\in U_3.
\end{eqnarray*}
Moreover, since $(e_X,e)\in\operatorname{co}\bigl\{(\alpha_2 e_{X,2},\alpha_2 e_2),(u_{X,0},u_0),(0,0)\bigr\}$, there exist $\lambda_1,\lambda_2,\lambda_3$ $\in[0,1]$
with $\lambda_1+\lambda_2+\lambda_3=1$ such that
\[
(e_X,e)=\lambda_1\left (\alpha_2 e_{X,2},\alpha_2 e_{2}\right) + \lambda_2\left ( u_{X,0},u_{0}\right)+ \lambda_3 \left(0,0\right).~~~~~~~~~~~~~~\eqno~~~~~~~~~~~~~~~~~~~~~~~~~~~~~~~~~~~~  (2.16)
\]
We pick a point $(x_0,y_0)\in U_3$ and
a functional  $(x_0^*, y_0^*)\in\partial f(x_0,y_0)$. Then, by the identity $(e_X,e)=\lambda_1 \left(\alpha_2 e_{X,2},\alpha_2 e_{2}\right) + \lambda_2 \left( u_{X,0},u_{0}\right)+ \lambda_3 \left(0,0\right) $, we conclude that
\[
\left\langle    (x_0^*, y_0^*),  (e_X,e)     \right\rangle = \lambda_1  \left\langle    (x_0^*, y_0^*),  (\alpha_2 e_{X,2},\alpha_2 e_{2})    \right\rangle +  \lambda_2 \left\langle    (x_0^*, y_0^*),  ( u_{X,0},u_{0})   \right\rangle + \lambda_3 0.
\]
Therefore, since $\{\lambda_1,\lambda_2,\lambda_3 \}\subset [0,1]$
and $\lambda_1+\lambda_2+\lambda_3 =1$,
we obtain that
\[
\left\langle    (x_0^*, y_0^*),  (e_X,e)     \right\rangle \leq \max \left\{  \left\langle    (x_0^*, y_0^*),  (\alpha_2 e_{X,2},\alpha_2 e_{2})    \right\rangle,   \left\langle    (x_0^*, y_0^*),  ( u_{X,0},u_{0})   \right\rangle ,0          \right\}.
\]
Noting that $(x_0^*, y_0^*)\in\partial f(x_0,y_0)$ and $\sigma_{\partial f}\left((x_0,y_0),(\alpha_{2}e_{X,2},\alpha_{2}e_{2})\right)\geq \sigma_{\partial f}((x_0,y_0), $ $( u_{X,0},u_{0}))+ \varepsilon_{2}^{20}$, by the above inequalities together with (2.7), we obtain that
\begin{eqnarray*}
&&\big\langle    (x_0^*, y_0^*),  (e_X,e)     \big\rangle
\\
&\leq& \max \left\{  \left\langle    (x_0^*, y_0^*),  (\alpha_2 e_{X,2},\alpha_2 e_{2})    \right\rangle,   \left\langle    (x_0^*, y_0^*),  ( u_{X,0},u_{0})   \right\rangle  ,0         \right\}
\\
&\leq& \max \left\{  \sigma_{\partial f}\left((x_0,y_0),(\alpha_{2}e_{X,2},\alpha_{2}e_{2})\right),   \sigma_{\partial f}\left((x_0,y_0), ( u_{X,0},u_{0})\right) ,0         \right\}
\\
&=& \sigma_{\partial f}\big((x_0,y_0),(\alpha_{2}e_{X,2},\alpha_{2}e_{2})\big).
\end{eqnarray*}
Since $(x_0^*, y_0^*)\in\partial f\left(x_0,y_0\right)$ is arbitrary, by the above inequalities, we obtain that
\begin{eqnarray*}
\sigma_{\partial f}\big((x_0,y_0), (e_X,e)  \big) &=& \sup_{\left(x^*, y^*\right)\in \partial f\left(x_0,y_0\right)}\big\langle    (x^*, y^*),  (e_X,e)     \big\rangle
\\
&\leq& \sigma_{\partial f}\big((x_0,y_0),(\alpha_{2}e_{X,2},\alpha_{2}e_{2})\big).
\end{eqnarray*}
Moreover, since $(x_0,y_0)\in U_3$ is arbitrary, by the above inequalities, we have
\[
\sigma_{\partial f}\big((x,y),(\alpha_{2}e_{X,2},\alpha_{2}e_{2})\big)\geq \sigma_{\partial f}\big((x,y), (e_X,e)\big)~\quad~~~\mathrm{for}~~~~\mathrm{every}~~~\quad~~~ (x,y)\in U_3.
\]

\par Suppose that $(e_X,e)\not\in \mathrm{co}\left\{ (\alpha_2 e_{X,2},\alpha_2 e_{2}),  ( u_{X,0},u_{0}),(0,0)              \right\}$.
Since $M_0$ is a two-dimensional subspace of $X\times \mathbb R$,
\[
(e_X,e) \in \overline{\mathrm{co}}\left( L_{0}\cup \left\{  \xi\cdot(\alpha_2 e_{X,2},\alpha_2 e_{2}): \xi\leq 1  \right\}\right),
\]
and $(e_X,e)\not\in \mathrm{co}\left\{ (\alpha_2 e_{X,2},\alpha_2 e_{2}),  ( u_{X,0},u_{0}),(0,0)              \right\}$,
there exists a point $(w_x,w)\in \left\{  \xi\left( u_{X,0},u_{0}\right): \xi\in [0, 1]  \right\}$ such  that
\[
 (w_x,w)\in\left\{(1-\xi)\cdot(\alpha_2 e_{X,2},\alpha_2 e_{2})+\xi\left(e_X,e\right): \xi\in [0,1] \right\}.
\]
Noting that $(w_x,w)\in \left\{ \xi \left(  u_{X,0}, u_{0}\right): \xi\in [0, 1]  \right\}$, there exists $\xi_0^{1} \in [0,1]$ such that $(w_x,w)=\xi_0^{1}\left( u_{X,0}, u_{0}\right)$.
Therefore, since
\[
(e_X,e)\in S_{3}\left(X\times \mathbb  R\right)\cap M_0~~~\quad~~~~~ \mathrm{and}~~~~~\quad~~( \alpha_2e_{X,2}, \alpha_2e_{2})\in S_{3}\left(X\times \mathbb R\right)\cap M_0,
\]
we obtain that $ (w_x,w)\in B_{3}\left(X\times \mathbb R\right)\cap M_0 $. Moreover, since $(w_x,w)=\xi_0^{1} \left(u_{X,0}, u_{0}\right)$ and
$\sigma_{\partial f}\left((x,y),(\alpha_{2}e_{X,2},\alpha_{2}e_{2})\right)>
\sigma_{\partial f}\left((x,y),(u_{X,0},u_{0})\right) +\varepsilon_{2}^{20}$,
we obtain that
\begin{eqnarray*}
\sigma_{\partial f}\big((x,y),(\alpha_{2}e_{X,2},\alpha_{2}e_{2})\big)&>&
\sigma_{\partial f}\big((x,y),(u_{X,0},u_{0})\big) +\varepsilon_{2}^{20}
\\
 &\geq&   \sigma_{\partial f}\left((x,y), (w_x,w) \right)+\varepsilon_{2}^{20}
\end{eqnarray*}
for every $(x,y)\in U_3 \subset V_2$. Noting that $U_3\subset V_2$, by (2.8),
 we conclude that
\[
\big\langle (x^{*},y^{*}), ( e_{X,2}, e_{2}) \big\rangle > \sup \left\{  \sigma_{\partial f}\big((x,y),( e_{X,2}, e_{2})\big):  (x,y)\in U_3              \right\}-\frac{1}{8}\varepsilon_2^{32}
\]
for every $(x^{*},y^{*})\in \partial f\left(U_3\right)$. Therefore, by the above inequality, we conclude that
\begin{eqnarray*}
&&\big\langle (x^{*},y^{*}), ( \alpha_2e_{X,2}, \alpha_2e_{2}) \big\rangle
\\
 & >& \sup \left\{  \sigma_{\partial f}((x,y),( \alpha_2e_{X,2}, \alpha_2e_{2})):  (x,y)\in U_3              \right\}-\frac{1}{8}\varepsilon_2^{32}
\\
&\geq& \sup \left\{  \sigma_{\partial f}((x,y), (w_x,w)):  (x,y)\in U_3              \right\}
- \frac{1}{8}\varepsilon_2^{32}  +\frac{1}{2}\varepsilon_{2}^{20}
\\
&\geq& \sup\left\{ \sigma_{\partial f} ((x,y),(w_x,w)): (x,y)\in U_{3} \right\}+\frac{1}{2}\varepsilon_{2}^{32}
\end{eqnarray*}
for every $(x^*, y^*)\in\partial f\left(U_3\right)$. Moreover,
since  $(w_x,w)\in \{ \xi\cdot \left(  u_{X,0}, u_{0}\right)$ $\in X\times \mathbb R: \xi\in [0, 1]  \}$ and
\[
(w_x,w)\in\left\{\xi\left(\alpha_2 e_{X,2},\alpha_2 e_{2}\right)+(1-\xi)\cdot(e_X,e): \xi\in [0,1] \right\},
\]
 there exists a real number $\xi_2\in [0,1]$ such that
\[
\big\langle (x^*, y^*),(w_x,w) \big\rangle= \xi_2\big\langle (x^*, y^*), \left(\alpha_2 e_{X,2},\alpha_2 e_{2}\right) \big\rangle+ (1-\xi_2)\big\langle (x^*, y^*),(e_X,e) \big\rangle~~~~~~~~~~~~~~\eqno~~~~~~~~~~~~~~~~~~~~~~~~~~~~~~~~~~~~  (2.17)
\]
for every $(x^*, y^*)\in\partial f\left(U_3\right)$. From the preceding argument, we conclude that
\[
\big\langle (x^*, y^*), \left(\alpha_2 e_{X,2},\alpha_2 e_{2}\right) \big\rangle   > \sigma_{\partial f} \big((x,y),(w_x,w)\big)
\]
for every $(x^*, y^*)\in\partial f\left(U_3\right)$ and  $(x,y)\in U_3$. It follows that
\[
\big \langle (x^*, y^*), \left(\alpha_2 e_{X,2},\alpha_2 e_{2}\right) \big\rangle >\sigma_{\partial f} \big ((x,y),(w_x,w) \big)\geq \big\langle (x^*, y^*), (w_x,w) \big\rangle
\]
for every $(x^*, y^*)\in\partial f\left(U_3\right)$ and  $(x,y)\in U_3$. Moreover, since $\xi_2\in [0,1]$, by (2.17) and the fact that $\langle (x^*, y^*), \left(\alpha_2 e_{X,2},\alpha_2 e_{2}\right) \rangle >\langle (x^*, y^*), (w_x,w) \rangle $,
we conclude that
\[
\big\langle (x^*, y^*), \left(\alpha_2 e_{X,2},\alpha_2 e_{2}\right) \big\rangle \geq \big\langle (x^*, y^*),(e_X,e) \big\rangle
\]
for every $(x^*, y^*)\in\partial f(U_3)$.
Pick a point $(x_0,y_0)\in U_3$ and a functional $(x_0^*, y_0^*)\in\partial f(x_0,y_0)$.
Then, by the above inequalities, we obtain that
\[
\sigma_{\partial f}\left((x_0,y_0),(\alpha_{2}e_{X,2},\alpha_{2}e_{2})\right)\geq \left\langle (x_0^*, y_0^*), \left(\alpha_2 e_{X,2},\alpha_2 e_{2}\right) \right\rangle \geq  \left\langle (x_0^*, y_0^*),(e_X,e) \right\rangle.
\]
Since $(x_0^*, y_0^*)$ is an arbitrary point in  $\partial f(x_0,y_0)$, from the preceding inequalities, we conclude that
\[
\sigma_{\partial f}\left((x_0,y_0),(\alpha_{2}e_{X,2},\alpha_{2}e_{2})\right)\geq    \sigma_{\partial f}\left((x_0,y_0),(e_X,e) \right).
\]
Moreover, since $(x_0,y_0)$ is an arbitrary point in $U_3$, from the preceding inequalities, we conclude that
\[
\sigma_{\partial f}\left((x,y),(\alpha_{2}e_{X,2},\alpha_{2}e_{2})\right)\geq \sigma_{\partial f}\left((x,y),(e_X,e) \right)~~\quad~~~\mathrm{for}~~~~\mathrm{every}~~~\quad~~~ (x,y)\in U_3.
\]

Suppose that
$(e_X,e)\in N\left(x^{*}|_{M_{0}}\right)$. Since
$M_0$ is a two-dimensional space, $s_3'=\sup \left\{  \sigma_{\partial f}\left((x,y),( \alpha_2e_{X,2}, \alpha_2e_{2})\right):  (x,y)\in U_3              \right\}$ and $(e_X,e)\not\in  B\left((\alpha_2 e_{X,2},\alpha_2 e_{2}),25\varepsilon_2\right)$, we obtain that $p_{3}(e_X,e)=p_{3}'(e_X,e)$. It follows that
\[
s_3\geq s_3'\geq \sup\left\{ \sigma_{\partial f} ((x,y),(e_X,e)): (x,y)\in U_{3} \right\}~~~\quad~~~~~ \mathrm{whenever}~~~~~\quad~~~(e_X,e)\in N\left(x^{*}|_{M_{0}}\right).
\]

\par Case II. Let $\|( \alpha_{2}e_{X,2},\alpha_{2}e_{2})- ( v'_{X,0},v'_{0})\|< \varepsilon_{2}$ or $\|( \alpha_{2}e_{X,2},\alpha_{2}e_{2})- ( u_{X,0},u_{0})\|< \varepsilon_{2}$.
We next prove that $\|( \alpha_{2}e_{X,2},\alpha_{2}e_{2})- ( u_{X,0},u_{0})\|< 20 \varepsilon_{2}$.
In fact, we may assume without loss of generality that $\|( \alpha_{2}e_{X,2},\alpha_{2}e_{2})- ( v'_{X,0},v'_{0})\|< \varepsilon_{2}$.
Otherwise, we have $\|( \alpha_{2}e_{X,2},\alpha_{2}e_{2})- ( u_{X,0},u_{0})\|< \varepsilon_{2}$.
Moreover, since
$\alpha_{2} \geq 1/(1+ 16\varepsilon_{2})$, we obtain that
\[
p_2 \left( \alpha_{2}e_{X,2},\alpha_{2}e_{2} \right)=\alpha_{2}p_2 \left( e_{X,2},e_{2} \right)=\alpha_{2}\geq \frac{1}{1+ 16\varepsilon_{2}}>1-17\varepsilon_{2}.
\]
Therefore, using the fact that
$\left\|( \alpha_{2}e_{X,2},\alpha_{2}e_{2})- ( v'_{X,0},v'_{0})\right\|< \varepsilon_{2}$, we conclude that
$
p_2(( \alpha_{2}e_{X,2},\alpha_{2}e_{2})- ( v'_{X,0},v'_{0}))<2 \varepsilon_{2}.
$
Therefore, using the fact that $p_2 \left( \alpha_{2}e_{X,2},\alpha_{2}e_{2} \right)$ $>1-17\varepsilon_{2}$, we obtain that
\begin{eqnarray*}
p_2\left( v'_{X,0},v'_{0}\right) &\geq& p_2\big(\alpha_{2}e_{X,2},\alpha_{2}e_{2}\big)-p_2\left(( \alpha_{2}e_{X,2},\alpha_{2}e_{2})- ( v'_{X,0},v'_{0})\right)
\\
&\geq & (1-17\varepsilon_{2})-2\varepsilon_{2} =   1-19\varepsilon_{2}.
\end{eqnarray*}
Hence we obtain that $p_2 \left( v'_{X,0},v'_{0}\right) \geq 1-19\varepsilon_{2}$.
Moreover, since $p_2\left( u_{X,0},u_{0}\right)\leq1$ and $( v'_{X,0},v'_{0})\in \{   \alpha\cdot( u_{X,0},u_{0}): \alpha\in \mathbb R^{+} \}$, we have the following inequalities
\[
p_2\left((u_{X,0},u_{0})-( v'_{X,0},v'_{0})\right) =  p_2\left(u_{X,0},u_{0}\right)- p_2\left( v'_{X,0},v'_{0}\right)\leq 1-(1-19\varepsilon_{2})=19\varepsilon_{2}.
\]
It follows that $\left\|(u_{X,0},u_{0})-( v'_{X,0},v'_{0})\right\|\leq 19\varepsilon_{2}$. Since $\left\|( \alpha_{2}e_{X,2},\alpha_{2}e_{2})- ( v'_{X,0},v'_{0})\right\|$ $< \varepsilon_{2}$, by the triangle inequality, we obtain that
\[
 \left\|\left( \alpha_{2}e_{X,2},\alpha_{2}e_{2}\right)- \left( u_{X,0},u_{0}\right)\right\|\quad \quad  \quad    \quad \quad\quad \quad  \quad  \quad   \quad
\]
\[
\quad \quad  \quad    \quad \quad\leq~ \left\|(\alpha_{2}e_{X,2},\alpha_{2}e_{2})- ( v'_{X,0},v'_{0}) \right\| ~+~ \left\|(u_{X,0},u_{0})-( v'_{X,0},v'_{0})\right\|\quad \quad  \quad \quad \quad \quad  \quad \quad   \quad \quad \quad \quad  \quad \quad   \quad \quad  \quad \quad \quad \quad  \quad \quad   \quad \quad\quad \quad  \quad \quad   \quad \quad \quad \quad  \quad \quad   \quad \quad
\]
\[
\quad \quad  \quad    \quad \quad\leq~ 19\varepsilon_{2}~+~ \varepsilon_{2} = 20 \varepsilon_{2}. \quad \quad  \quad \quad   \quad \quad \quad \quad  \quad \quad   \quad \quad  \quad \quad  \quad \quad  \quad \quad  \quad \quad   \quad \quad \quad \quad  \quad \quad  \quad \quad  \quad  \quad  \quad    \quad \quad \quad  \quad \quad \quad \quad \quad \quad \quad \quad  \quad \quad   \quad \quad\quad \quad  \quad \quad   \quad \quad \quad \quad  \quad \quad   \quad \quad
\]
Hence we obtain that $\left\|( \alpha_{2}e_{X,2},\alpha_{2}e_{2})- ( u_{X,0},u_{0})\right\|< 20 \varepsilon_{2}$.
We pick a point
\[
(e_X,e)\in S_{3}\left(X\times \mathbb R\right)\cap M_0~~~~~\quad~~~~~ \mathrm{with}~~~~~\quad~~~~~~~ (e_X,e)\not\in  B\left((\alpha_2 e_{X,2},\alpha_2 e_{2}),25\varepsilon_2\right).
\]
Then we obtain that $(e_X,e)\notin N\left(x^{*}|_{M_{0}}\right)$ and $(e_X,e)\in N\left(x^{*}|_{M_{0}}\right)$, where
\[
N\left(x^{*}|_{M_{0}}\right)=\{x\in M_{0}: x^{*}|_{M_{0}}(x)=0\}=\{\xi\left( e_{X,2}, e_{2}\right)\in M_{0}:\xi\in\mathbb R \}.~~
\]
Suppose that $(e_X,e)\notin N\left(x^{*}|_{M_{0}}\right)$.
Since $M_0$ is a two-dimensional subspace of $X\times \mathbb R$  and
$(e_X,e)\notin N\left(x^{*}|_{M_{0}}\right)$, we may assume   that
$\left\langle x^{*}|_{M_{0}}, \left( u_{X,0},u_{0}\right)\right\rangle>0 $ and $ \langle x^{*}|_{M_{0}}, $ $(e_X,e)\rangle>0$.
We know  that
$
L_{0}=\left\{(1-\xi)\cdot (\alpha_2 e_{X,2},\alpha_2 e_{2})+\xi\cdot( u_{X,0},u_{0}): \xi\geq 0 \right\}
$.
Then we pick a real number $\xi_1\in (1,+\infty) $ such that
\[
\left\|(1-\xi_1)\left(\alpha_2 e_{X,2},\alpha_2 e_{2}\right)+\xi_1\left( u_{X,0},u_{0}\right)-\left( u_{X,0},u_{0}\right)\right\|=40\varepsilon_{2}.
\]
Define $\left( w_{X,0},w_{0}\right)=\left(1-\xi_1\right)\left(\alpha_2 e_{X,2},\alpha_2 e_{2}\right)+\xi_1\left( u_{X,0},u_{0}\right)$.
Then, from the proof of Lemma 2.4, there exists a point $( u_{X,0}',u_{0}')$ with
\[
( u_{X,0}',u_{0}')\in \{  (1- \xi)\left( \lambda_{2}e_{X,2},\lambda_{2}e_{2}\right)+ \xi \left( v_{X,0},v_{0}\right):  \xi \in [0,+\infty)    \}
\]
such that $( u_{X,0}',u_{0}')\in \{ \lambda  \cdot  \left( w_{X,0},w_{0}\right):  \lambda  \in [0,1]         \}$.
We claim that $\|(\alpha_2 e_{X,2},\alpha_2 e_{2})-( u_{X,0}',u_{0}')\|\geq\varepsilon_2$. In fact, suppose that
$\|(\alpha_2 e_{X,2},\alpha_2 e_{2})-( u_{X,0}',u_{0}')\|<\varepsilon_2$. Then, from the previous argument for Case II, we have
$ \|(\alpha_2 e_{X,2},\alpha_2 e_{2})-\left( w_{X,0},w_{0}\right)\|< 20\varepsilon_2$,  a contradiction. This implies that
$\|(\alpha_2 e_{X,2},\alpha_2 e_{2})$ $-( u_{X,0}',u_{0}')\|\geq\varepsilon_2$.
Then, from
the proof of (a) in Case I, it is easy to see that
\[
\left(\left.\mathrm{\frac{d^{+}}{d\xi}}\mu_{  A_2 } \left[(1-\xi)\left( \alpha_{2}e_{X,2},\alpha_{2}e_{2}\right)+\xi \left( w_{X,0},w_{0}\right)\right]\right)\right|_{\xi=0}>\frac{1}{4}\varepsilon_{2}.~~~~~~~~~~~~~~\eqno~~~~~~~~~~~~~~~~~~~~~~~~~~~~~~~~~~~~  (2.18)
\]
Define the continuous convex function $h\left(\xi\right)$ such that $h\left(\xi\right)=h_{1}(\xi)+h_{2}(\xi)$, where
\[
h_{1}\left(\xi\right)=p_2\left[ (1-\xi)\left( \alpha_{2}e_{X,2},\alpha_{2}e_{2}\right)+\xi \left( w_{X,0},w_{0}\right)\right]
\]
and
\[
h_{2}\left(\xi\right)= 16\varepsilon_{2}\cdot\mu_{  A_2 }\left[(1-\xi)\left( \alpha_{2}e_{X,2},\alpha_{2}e_{2}\right)+\xi \left( w_{X,0},w_{0}\right)\right].
\]
Then there exists $\xi_2\in [0,1]$ such that $( u_{X,0},u_{0})=(1-\xi_2)( \alpha_{2}e_{X,2},\alpha_{2}e_{2})+ \xi_2( w_{X,0},$ $w_{0})$.
Hence
$h_1\left(\xi\right)$ and  $h_2\left(\xi\right)$ are increasing on $\left[\xi_2,+\infty\right)$. Since the convex functions $h_1(\xi)$ and  $h_2(\xi)$ are increasing on $[\xi_2,+\infty)$, by (2.18) and arguing as in part (a) of Case I, we conclude that
\[
p_{3}\left[ (1-\xi)\left( \alpha_{2}e_{X,2},\alpha_{2}e_{2}\right)+ \xi\left( w_{X,0},w_{0}\right)\right]> 1+3\varepsilon_2^{4}
\]
for all $\left\|(1-\xi)(\alpha_{2}e_{X,2},\alpha_{2}e_{2})+ \xi( w_{X,0},w_{0})-( u_{X,0},u_{0})\right\|\geq 2\varepsilon_2$ and $\xi\geq \xi_2$. Moreover,
the function $f_3\left(\xi\right)$ is increasing on the interval $[\xi_2,+\infty)$, where
\[
f_3\left(\xi\right)=p_{3}\left[    (1-\xi)\left(\alpha_2 e_{X,2},\alpha_2 e_{2}\right)+\xi\left( w_{X,0},w_{0}\right)       \right].
\]
Therefore, from the argument in part (b) of Case I, we conclude that
\[
\sigma_{\partial f}\left((x,y),(\alpha_{2}e_{X,2},\alpha_{2}e_{2})\right)- \sigma_{\partial f}\left((x,y),( u_{X,0},u_{0})\right)\geq \varepsilon_{2}^{20}~~\quad~~~\mathrm{for}~~~~\mathrm{every}~~~\quad~~~(x,y)\in U_3.
\]
Noting that $(e_X,e)\not\in  B\left(\left(\alpha_2 e_{X,2},\alpha_2 e_{2}\right),25\varepsilon_2\right)$ and $\left\|\left( \alpha_{2}e_{X,2},\alpha_{2}e_{2}\right)- ( u_{X,0},u_{0})\right\|< 20 \varepsilon_{2}$,
and arguing as in part (c) of Case I, we conclude that
\[
(e_X,e) \in \overline{\mathrm{co}}\left( L_{0}\cup \left\{  \xi\cdot(\alpha_2 e_{X,2},\alpha_2 e_{2}): \xi\leq 1  \right\}\right).
\]
Arguing as in part (c) of Case I, we obtain the following inequality
\[
\sigma_{\partial f}\big((x,y),(\alpha_{2}e_{X,2},\alpha_{2}e_{2})\big)\geq \sigma_{\partial f}\big((x,y),(e_X,e) \big)~~\quad~~~\mathrm{for}~~~~\mathrm{every}~~~\quad~~~ (x,y)\in U_3.
\]
Suppose that
$(e_X,e)\in N\left(x^{*}|_{M_{0}}\right)$. Arguing as in Case I, we conclude that
\[
s_3\geq s_3'\geq \sup\left\{ \sigma_{\partial f} \left((x,y),(e_X,e)\right): (x,y)\in U_{3} \right\}~~~\quad~~~~~ \mathrm{whenever}~~~~~\quad~~~(e_X,e)\in N\left(x^{*}|_{M_{0}}\right).
\]
Therefore, by Case I and Case II, we have the following formula
\begin{eqnarray*}
s_{3}&=&\sup\left\{ \sigma_{\partial f}\left ((x,y),(e_{X},e)): ((x,y),(e_{X},e)\right)\right.
\\
&\in& U_{3}\times B \left( \left(\alpha_{2}e_{X,2},\alpha_{2}e_{2}\right),25\varepsilon_{2} \right) \cap B_{3}\left(X\times\mathbb R\right) \}.
\end{eqnarray*}
Hence (2.9) holds. We next continue the proof of the theorem.

\par Since $X$ is a weak Asplund space, by Lemma 2.2 and Lemma 2.3, there exists a decreasing sequence of dense open cones $\{O_n^3\}_{n=1}^{\infty}$ contained in $G_3$
such that $p_3$ is G\^ateaux differentiable on $\cap_{n=1}^{\infty}O_n^3$ and $O_{n+1}^3\subset O_n^3$, where
\begin{eqnarray*}
G_3 &=& \left\{  \lambda  \left(x, f_3\left(x\right)\right) \in X\times \mathbb R :   x\in T \left( \overline{\mathrm{co}}\left(S_{3}\left(X\times\mathbb  R\right)\right) \right),~  \lambda\in (0,+\infty)     \right\}
\\
&& \cup\left\{  \lambda  \left(x, g_3\left(x\right)\right)\in X\times\mathbb  R :   x\in T \left( \overline{\mathrm{co}}\left(S_{3}\left(X\times\mathbb  R\right)\right) \right),~  \lambda\in (0,+\infty)      \right\},
\end{eqnarray*}
\[
f_{3}\left(x\right)= \inf \left\{ r\in\mathbb  R: (x,r) \in    \overline{\mathrm{co}}\left(S_{3}\left(X\times\mathbb  R\right)\right)         \right\}
\]
and
\[
g_{3}\left(x\right)= \sup \left\{ r\in\mathbb  R: (x,r) \in    \overline{\mathrm{co}}\left(S_{3}\left(X\times\mathbb  R\right)\right)         \right\}.
\]
We pick a real number $\varepsilon_3=\varepsilon_2/128 >0$. Then, since $\prod_{i=0}^{\infty}\big[1-(20\varepsilon_{1}/128^{i})\big]>3/4$,  we obtain that $(1-16\varepsilon_1)\left(1-16\varepsilon_2\right)$ $(1-16\varepsilon_3)>3/4$.
Moreover, by Lemma 2.7, we get that the point $( \alpha_{2}e_{X,2},\alpha_{2} e_{2})$ is an extreme point of $B_{3}\left(X\times\mathbb R\right)$. Hence we have
$\alpha_{2} e_{2}\notin\big( f_3(\alpha_{2}e_{X,2}), g_3(\alpha_{2}e_{X,2})\big)$.
Since $p_3\left(\alpha_{2}e_{X,2},\alpha_{2}e_{2}\right)=1$ and
\[
s_3'=\sup \left\{  \sigma_{\partial f}\big(\left(x,y\right),\left( \alpha_2e_{X,2}, \alpha_2e_{2}\right)\big):  \left(x,y\right)\in U_3              \right\},
\]
there exists a real number $r_3\in (0,r_2/4)$ and
two points $(x_{3},y_{3})\in U_3$,
$(e_{X,3},e_{3})\in  \cap_{n=1}^{\infty}O_{n}^{3}\subset X\times\mathbb R
$ with
\[
p_3\left(e_{X,3},e_{3}\right)=1~~\quad~~~~~\mathrm{and}~~~~~~~~\quad~~\sigma_{\partial f}\big((x,y),(e_{X,3},e_{3})\big)\geq\left(1-\frac{1}{16}\varepsilon_{3}^{36}\right) s_{3}'>0~~~~~~~~~~~~~~\eqno~~~~~~~~~~~~~~~~~~~~~~~~~~~~~~~~~~~~  (2.19)
\]
for every $(x,y)\in B\left((x_{3},y_{3}),r_3\right)\subset U_3$
such that
\par (1) the mapping $(e_{X,3},e_{3})(\partial f)$ is single-valued at $(x_{3},y_{3})\in X\times\mathbb R$;
\par (2) $p_3 \left(( \alpha_{2}e_{X,2},\alpha_{2} e_{2})- (e_{X,3},e_{3})   \right)   \leq \eta_2 / 50 $;
\par (3) $T\left(e_{X,3},e_{3}\right)\in \mathrm{int} T\left \{ (x,y)\in X\times\mathbb R: p_3 \left(x,y\right) \leq 1    \right\}$.
\\
Let $V_3=\mathrm{int}B\left((x_{3},y_{3}\right),r_3)$. Then $V_3$ is an open set.
Since $ B\left((e_{X,2},e_{2}),256\eta_2 \right)\subset  \left( O_{2}^{1} \cap O_{2}^{2}\right)$ and $O_{i}^{j}$
is an
open cone,  by the fact that $p_3 (( \alpha_{2}e_{X,2},\alpha_{2} e_{2})- (e_{X,3},e_{3})  )   \leq \eta_2  / 50$
and  $\alpha_{2}\in [3/4,1]$,
we obtain the following formula
\[
(e_{X,3},e_{3})\in B\left(( \alpha_{2}e_{X,2}, \alpha_{2}e_{2}),\frac{1}{2}\eta_2\right)\subset B\big(\left( \alpha_{2}e_{X,2}, \alpha_{2}e_{2}\right),25\eta_2\big)\subset \left( O_{2}^{2} \cap O_{2}^{1} \right).
\]
Therefore, from $(e_{X,3},e_3)\in\cap_{n=1}^{\infty}O_n^3\subset X\times\mathbb R$ and the above relation, there exists a real number $\eta_3\in\bigl(0,\min\{\eta_2/128,\varepsilon_3^{64}/128\}\bigr)$ such that
\[
 B\big ((e_{X,3},e_{3}),256\eta_3\big)\subset \left(  O_{3}^{1} \cap  O_{3}^{2} \cap  O_{3}^{3} \right)\subset X\times\mathbb R.
\]
Moreover, since $\partial f(U_1)\subset B_1(X^*\times\mathbb R)$, by (2.9), we conclude that $s_3'\ge s_3-25\varepsilon_2$. Therefore, by (2.19) and the fact that $s_3'\ge s_3-25\varepsilon_2$, we obtain that
\[
\sigma_{\partial f}\big((x,y),(e_{X,3},e_{3})\big)~\geq~ \left(1-\frac{1}{16 }\varepsilon_{3}^{32}\right) s_{3}' \geq \left(1-\frac{1}{16 }\varepsilon_{3}^{32}\right) \big( s_{3}-25\varepsilon_{2} \big)~
\]
\[
\quad\quad\quad\quad \quad \quad \quad \quad \quad \quad \quad \quad=~  s_{3}~-~25\varepsilon_{2}~-~\frac{1}{16 }\varepsilon_{3}^{32}s_{3}~+~\frac{1}{16 }\varepsilon_{3}^{32}\cdot25\varepsilon_{2} \quad \quad \quad\quad \quad \quad \quad \quad \quad \quad \quad \quad \quad \quad \quad\quad \quad \quad \quad \quad \quad \quad
\]
\[
\quad\quad\quad \quad \quad \quad\quad \quad \quad \quad \quad \quad \geq~ s_3~-~25\varepsilon_{2}~-~ \frac{1}{8 }\varepsilon_{2} \quad \quad \quad \quad \quad \quad \quad \quad \quad \quad \quad \quad \quad \quad \quad \quad \quad \quad \quad\quad \quad \quad\quad \quad\quad \quad \quad \quad \quad \quad\quad \quad \quad \quad \quad \quad
\]
for every $(x,y)\in \mathrm{int}B\left((x_{3},y_{3}),r_3\right)$.
Since $( e_{X,3}, e_{3})\partial f$ is a single-valued mapping at $(x_3,y_3)\in X\times \mathbb R$, we obtain that
the functional $(x,y) \to \sigma_{\partial f}\left(\left(x,y\right),\left(e_{X,3},e_{3}\right)\right)$ is continuous at $\left(x_{3},y_{3}\right)$. Moreover, since
$\partial f$ is norm-to-weak$^{*}$ upper-semicontinuous, $\left( e_{X,3}, e_{3}\right)\partial f$ is single-valued at $(x_3,y_3)$ and
$(x,y) \to \sigma_{\partial f}(\left(x,y\right),\left(e_{X,3},e_{3}\right))$ is continuous at $\left(x_{3},y_{3}\right)\in X\times \mathbb R$, we may assume without loss of generality that
\[
\big\langle (x^{*},y^{*}), ( e_{X,3}, e_{3}) \big\rangle > \sigma_{\partial f}\big((x,y),(e_{X,3},e_{3})\big)- \frac{1}{128}\varepsilon_{2}^{64}\geq s_3-25\varepsilon_{2}-\frac{1}{2 }\varepsilon_{2}
\]
for every $(u,v)\in V_3$ and $(x^{*},y^{*})\in \partial f\left(u,v\right)$. Since  $\eta_2\in (0,\min \{\eta_1/128, \varepsilon_2^{64}/128 \} )$, $\varepsilon_3=\varepsilon_2/128 $ and $p_3 \left(( \alpha_{2}e_{X,2},\alpha_{2} e_{2})- (e_{X,3},e_{3})  \right )   \leq \eta_2 / 50 $, by the fact that $V_{3}\subset U_{3}$ and $s_3'=\sup \left\{  \sigma_{\partial f}\left(\left(x,y\right),\left( \alpha_2e_{X,2}, \alpha_2e_{2}\right)\right):  \left(x,y\right)\in U_3              \right\}$, we obtain that
\begin{eqnarray*}
&&\big\langle (x^{*},y^{*}), ( e_{X,3}, e_{3}) \big\rangle
\\
&>&   \sigma_{\partial f}\big((x,y),(e_{X,3},e_{3})\big)- \frac{1}{128}\varepsilon_{2}^{64}
\\
&\geq& \left(1-\frac{1}{16 }\varepsilon_{3}^{32}\right) s_{3}'- \frac{1}{128}\varepsilon_{2}^{64}
\\
&\geq & \sup \left\{  \sigma_{\partial f}\left(\left(x,y\right),\left( \alpha_2e_{X,2}, \alpha_2e_{2}\right)\right):  \left(x,y\right)\in V_3              \right\}- \frac{1}{16 }\varepsilon_{3}^{32}s_{3}'- \frac{1}{128}\varepsilon_{2}^{64}
\\
&\geq & \sup \left\{  \sigma_{\partial f}\left(\left(x,y\right),\left( e_{X,3}, e_{3}\right)\right):  \left(x,y\right)\in V_3              \right\}- \frac{1}{128}\varepsilon_{2}^{64}-\frac{1}{16 }\varepsilon_{3}^{32}s_{3}'- \frac{1}{128}\varepsilon_{2}^{64}
\\
&\geq & \sup \left\{  \sigma_{\partial f}\left(\left(x,y\right),\left( e_{X,3}, e_{3}\right)\right):  \left(x,y\right)\in V_3              \right\}- \frac{1}{128}\varepsilon_{2}^{64}-\frac{1}{16 }\varepsilon_{3}^{32}- \frac{1}{128}\varepsilon_{2}^{64}
\\
&\geq& \sup \left\{  \sigma_{\partial f}\left(\left(x,y\right),\left( e_{X,3}, e_{3}\right)\right):  \left(x,y\right)\in V_3              \right\}-\frac{1}{8}\varepsilon_{3}^{32}
\end{eqnarray*}
for every $(u,v)\in V_3$ and $(x^{*},y^{*})\in \partial f\left(u,v\right)$.
Moreover,
Player $A$ may choose
any nonempty open subset $U_4 \subset V_3$. Hence
we may assume  that
\[
\sup \left\{ \left \|(x^{*},y^*)\right\|\in \mathbb R :   (x^{*},y^*)\in \partial f\left(U_{4}\right)      \right\}>0.
\]

\textbf{Step 3.} In this step, we prove that if the conclusion of Step 2 holds for natural number $k$, then the conclusion of Step 2 holds for natural number $k+1$.

\par By hypothesis, we define $p_k= p_{k-1}+16\varepsilon_{k-1}\cdot \mu_{ A_{k-1}}$ and
$\varepsilon_{k}=\varepsilon_{k-1}/128$. Then, since $\prod_{i=0}^{\infty}\big[1-(20\varepsilon_{1}/128^{i})\big]>3/4$,  we have $\prod_{i=1}^{k}\big(1+16\varepsilon_i\big)^{-1}>3/4$.
Moreover, we know that
\[
s'_{k}=\sup\left\{ \sigma_{\partial f} \left((x,y),(e_{X},e)\right): ((x,y),(e_{X},e))\in U_{k}\times S'_{k}\left(X\times\mathbb R\right) \right\}.
\]
Moreover, by hypothesis,
there exist a real number $r_k\in (0,r_{k-1}/4)$ and
two points $(x_{k},y_{k}) \in U_k$, $(e_{X,k},e_{k})\in  \cap_{n=1}^{\infty}O_{n}^{k} \subset X\times\mathbb R
$  with
\[
p_k\left(e_{X,k},e_{k}\right)=1~~\quad~~~~~\mathrm{and}~~~~~~~~\quad~~\sigma_{\partial f}\big((x,y),(e_{X,k},e_{k})\big)>\left(1-\frac{1}{16}\varepsilon_{k}^{36}\right) s_{k}'>0
\]
for every $(x,y)\in B\big((x_{k},y_{k}),r_k\big)\subset U_{k}$ such that
\par (1) the mapping $\left(e_{X,k},e_{k}\right)(\partial f)$ is single-valued at $(x_{k},y_{k})\in X\times\mathbb R$;
\par (2) $p_{k} \left(\alpha_{k-1}\cdot(e_{X,k-1}, e_{k-1})- (e_{X,k},e_{k})   \right)   \leq \eta_{k-1}/50$;
\par (3) $T\left(e_{X,k},e_{k}\right)\in \mathrm{int} T\left\{ (x,y)\in X\times\mathbb R: p_k \left(x,y\right) \leq 1    \right\}$.
\\
Define the set $V_k=\mathrm{int}B\big((x_k,y_k),r_k\big)$. Then we obtain that $V_k$ is open  and $\overline{V_k}\subset U_k$.
Recall that $\partial f$ is norm-to-weak$^*$ upper semicontinuous, and $(e_{X,k},e_k)(\partial f)$ is single-valued at the point $(x_k,y_k)$.
Under the induction hypothesis, we obtain that
$\eta_{k}\in (0,\min \{\eta_{k-1}/128,\varepsilon_{k}^{64}/128\} )$ and
\begin{eqnarray*}
\quad \quad s_{k}&=&\sup\left\{ \sigma_{\partial f} \big((x,y),(e_{X},e)\big): ((x,y),(e_{X},e))\right.
\\
&\in& \left.U_{k+1}\times B \big( \left(\alpha_{k-1}e_{X,k-1},\alpha_{k-1}e_{k-1}\big),25\varepsilon_{k-1} \right) \cap B_{k}\left(X\times\mathbb R\right) \right\}. \quad \quad ~                (2.20)
\end{eqnarray*}
We next prove that the conclusion of Step 2 holds for natural number $k+1$.
First, we define the closed set $C_k$, where
\[
C_k=\left\{  \left(\alpha e_{X,k},\alpha e_{k}\right):  0\leq \alpha \leq  1+\varepsilon_k^{16}      \right\} \cup \left\{ (x,y):    p_k\left(x,y\right)\leq \frac{1}{512^{3}}     \right\} . \quad
\]
Using the definition of $C_k$, we define the non-convex functional $\mu_{C_k}$, where
\[
\mu_{C_k }(x,y)= \inf \left\{  \lambda\in\mathbb R^{+} :  \frac{1}{\lambda}\left(x,y\right)\in   C_k           \right\}
\]
for every $(x,y)\in X\times\mathbb R$. Moreover, we define the functional $p_{k+1}'$, where
\[
p_{k+1}'\left(x,y\right)=p_k\left(x,y\right)+16\varepsilon_{k} \cdot\mu_{ C_k}(x,y)~\quad~~~~~\mathrm{for}~~~~~~~~~\mathrm{every}~~~~~~~\quad~~(x,y)\in X\times\mathbb R.
\]
Therefore, since $ p_k\left( e_{X,k}, e_{k}\right)=1$, there exists a real number $\alpha_k\in (0,1) $ such that
\[
p_k\left(\alpha_k e_{X,k},\alpha_k e_{k}\right)+ 16 \varepsilon_{k}\cdot \mu_{C_k}\left(\alpha_k e_{X,k},\alpha_k e_{k}\right)= p_{k+1}'\left(\alpha_k e_{X,k},\alpha_k e_{k}\right)=1.
\]
Since $ p_k\left( e_{X,k}, e_{k}\right)=1$ and $\mu_{C_k}\left(\alpha_k e_{X,k},\alpha_k e_{k}\right)\leq p_k\left(\alpha_k e_{X,k}, \alpha_k e_{k}\right)$, we obtain that
\[
p_k\left(\alpha_k e_{X,k},\alpha_k e_{k}\right)+ 16 \varepsilon_{k}\cdot p_k\left(\alpha_k e_{X,k},\alpha_k e_{k}\right)> p_{k+1}'\left(\alpha_k e_{X,k},\alpha_k e_{k}\right)=1.
\]
Then we obtain that $1\geq\alpha_k>(1+16 \varepsilon_{k})^{-1}$.
Hence we define the set $S'_{k+1}\left(X\times\mathbb R\right)$ and define a real number $s_{k+1}'>0$, where
\[
S'_{k+1}\left(X\times\mathbb R\right)=\left\{(x, y)\in X\times\mathbb R: p_{k+1}'\left(x,y\right)= p_k\left(x,y\right)+16\varepsilon_k \cdot \mu_{ C_k }(x,y)=1\right\}
\]
and
\[
s'_{k+1}=\sup\left\{ \sigma_{\partial f} \left((x,y),(e_{X},e)\right): ((x,y),(e_{X},e))\in U_{k+1}\times S'_{k+1}\left(X\times\mathbb R\right) \right\}.
\]
We next prove that
\[
s_{k+1}'=\sup\big\{\sigma_{\partial f}\big((x,y),(\alpha_k e_{X,k},\alpha_k e_k)\big):(x,y)\in U_{k+1}\big\}
\]
holds. In fact, noting that $\partial f\left(U_{1}\right) $ is a subset of $ B_1\left(X^{*}\times\mathbb R\right)$ and
\[
s_{k}'= \sup \left\{  \sigma_{\partial f}\big((x,y),( \alpha_{k-1} e_{X,k-1},  \alpha_{k-1}e_{k-1})\big):  (x,y)\in U_{k}             \right\},
\]
by (2.20), we conclude that $ s_{k}'\geq s_{k}-25\varepsilon_{k-1}$. Since $ U_{k+1}  \subset V_{k}  $, we conclude that
\[
\sup \left\{  \sigma_{\partial f}\left((x,y),(  e_{X,k},  e_{k})\right):  (x,y)\in U_{k+1}             \right\}\geq \left(  1-\frac{1}{16} \varepsilon_k^{32}   \right)s_k'.
\]
Noting that $ s_{k}'\geq s_{k}-25\varepsilon_{k-1}$ and $\varepsilon_{k}=\varepsilon_{k-1}/128$, by the above inequality together with $\varepsilon_{1}\in \left(0,1/512^{6}\right)$, we have the following inequalities
\[
\sup \left\{  \sigma_{\partial f}\left((x,y),(  e_{X,k},  e_{k})\right):  (x,y)\in U_{k+1}             \right\} \quad
\]
\[
 \quad  \quad\quad\quad \quad \quad\geq~\left(  1-\frac{1}{16} \varepsilon_k^{32}   \right)s_k' \geq \left(  1-\frac{1}{16} \varepsilon_k^{32} \right) \big(s_{k}-25\varepsilon_{k-1}\big)\quad \quad \quad \quad \quad \quad \quad \quad \quad  \quad \quad \quad \quad \quad \quad \quad \quad \quad \quad \quad \quad \quad \quad \quad
\]
\[
 \quad\quad \quad\quad \quad \quad \geq~ s_{k}~-~ \left(25\varepsilon_{k-1}+\frac{s_k}{16} \varepsilon_k^{32} +\frac{25}{16} \varepsilon_k^{32}\cdot\varepsilon_{k-1}\right)   \quad\quad \quad\quad \quad \quad \quad\quad \quad\quad \quad \quad \quad\quad \quad\quad \quad \quad \quad\quad \quad\quad \quad \quad \quad\quad \quad\quad \quad \quad \quad\quad \quad\quad \quad \quad
 \]
\[
 \quad\quad \quad\quad \quad \quad  \geq~
  s_{k}~-~ 30\varepsilon_{k-1}. \quad \quad \quad \quad \quad \quad \quad \quad \quad  \quad \quad \quad \quad \quad \quad \quad \quad \quad \quad \quad \quad \quad \quad \quad \quad \quad \quad  \quad \quad \quad \quad \quad \quad \quad \quad \quad \quad \quad \quad \quad \quad  \quad \quad \quad \quad \quad\quad \quad \quad \quad \quad \quad \quad \quad \quad \quad
\]
Moreover, from the previous argument, we obtain that $\alpha_k>\left(1+16 \varepsilon_{k}\right)^{-1}$. Therefore, using the above inequalities together with $\alpha_k>(1+16 \varepsilon_{k})^{-1}$, we obtain that
\[
 \sup \left\{  \sigma_{\partial f}\left((x,y),(  \alpha_k e_{X,k}, \alpha_k e_{k})\right):  (x,y)\in U_{k+1}             \right\} \geq \frac{1}{1+16\varepsilon_{k}} \left(s_{k}- 30\varepsilon_{k-1}\right).
\]
Moreover, we define the closed convex set $B_{k,0} \left(X\times \mathbb R\right)$, where
\[
B_{k,0} \left (X\times R\right)= \left\{   (x, y)\in X\times\mathbb  R:      p_k\left(x,y\right) +\left(16   \varepsilon_{k} \cdot   512^{3} \right) p_k\left(x,y\right)\leq 1           \right\}.
\]
Noting that $s_{k}=\sup\{ \sigma_{\partial f} ((x,y),(e_{X},e)): ((x,y),(e_{X},e))\in U_{k}\times S_{k}(X\times\mathbb  R) \}$,  by the definitions of $B_{k,0} \left (X\times\mathbb  R\right)$ and $s_k$, we conclude that
\[
 \frac{s_{k}}{1+16(512)^{3}\varepsilon_{k}} =\sup\left\{ \sigma_{\partial f} ((x,y),(e_{X},e)): ((x,y),(e_{X},e))\in U_{k}\times B_{k,0}  (X\times\mathbb  R) \right\}.
\]
Therefore, using $U_{k+1}\subset U_k$ together with the above equality, we conclude that
\[
 \frac{s_{k}}{1+16(512)^{3}\varepsilon_{k}} \geq \sup\{ \sigma_{\partial f} ((x,y),(e_{X},e)): ((x,y),(e_{X},e))\in U_{k+1}\times B_{k,0}  (X\times\mathbb  R) \}.
\]
Noting that $s_1=1$,  $\varepsilon_{k+1}=\varepsilon_k/128$ and  $\varepsilon_{1}\in \left(0,1/512^{6}\right)$, we conclude that $s_{k}>7/8$.
Moreover, since $\varepsilon_{k+1}=\varepsilon_k/128$ and  $\varepsilon_{1}\in \left(0,1/512^{6}\right)$, we obtain that
\[
\frac{1}{1+16\varepsilon_{k}} \left(s_{k}- 30\varepsilon_{k-1}\right)\geq\frac{1}{1+16\varepsilon_{k}} \left(s_{k}- 30\cdot128\varepsilon_{k}\right)\geq     \frac{1}{1+16(512)^{3}\cdot\varepsilon_{k}} s_{k}.
\]
Therefore, from the preceding inequalities and the definition of $B_{k,0} \left (X\times\mathbb R\right) $, we obtain that
\begin{eqnarray*}
 &&\sup \left\{  \sigma_{\partial f}\left((x,y),(  \alpha_k e_{X,k}, \alpha_k e_{k})\right):  (x,y)\in U_{k+1}             \right\}
 \\
 &\geq& \frac{1}{1+16\varepsilon_{k}} \left(s_{k}- 30\varepsilon_{k-1}\right)  \geq     \frac{1}{1+16(512)^{3}\cdot\varepsilon_{k}} s_{k}
\\
 &\geq& \sup\left\{ \sigma_{\partial f} \left((x,y),(e_{X},e)\right): ((x,y),(e_{X},e))\in U_{k+1}\times B_{k,0}  \left(X\times \mathbb R\right) \right\}.
\end{eqnarray*}
Since $p_{k+1}'= p_k(x,y)+16\varepsilon_k \cdot \mu_{ C_k }(x,y)$ for every $(x,y)\in X \times \mathbb R$,
by the definition of $C_k$ and  the preceding inequalities, we conclude that
\[
s_{k+1}'= \sup \left\{  \sigma_{\partial f}\big((x,y),( \alpha_k e_{X,k},  \alpha_ke_{k})\big):  (x,y)\in U_{k+1}             \right\}.
\]
Moreover, we define the closed set $D_k'$, where
\[
D_k'=\left\{ \left( \alpha e_{X,k}, \alpha e_{k}\right):  0\leq \alpha \leq  1+\varepsilon_k^{16}     \right\} \cup \left\{ (x,y)\in X\times \mathbb R:    p_k\left(x,y\right)\leq 1    \right\} .
\]
Using the definition of $D_k'$, we define the  functional $\mu_{D_k'}$, where
\[
\mu_{D_k'}\left(x,y\right)= \inf \left\{  \lambda\in\mathbb R^{+} :  \frac{1}{\lambda}\left(x,y\right)\in  D_k'           \right\}
\]
for every $(x,y)\in X\times \mathbb R$.
Hence we define the functional $p_{k+1}''$, where
\[
p_{k+1}''\left(x,y\right)=p_2\left(x,y\right)+16\varepsilon_{k} \cdot\mu_{D_k'}(x,y)~\quad~~~~~\mathrm{for}~~~~~~~~~\mathrm{every}~~~~~~~\quad~~(x,y)\in X\times \mathbb R.
\]
It follows that $D_k= \overline{\mathrm{co}}\left\{ (e_{X},e): p_{k+1}' (e_{X},e)=1        \right\}$ is a closed convex set.
Hence we define the the Minkowski functional $\mu_{D_k}$. Let $A_k=  \overline{\mathrm{co}} \left(C_k\right)$. Then the set $A_k$ is a closed convex set.
Hence we define the Minkowski functional $\mu_{  A_k }$. Moreover, we define the two sets $S_{k+1}\left(X\times \mathbb R\right)$ and $B_{k+1}\left(X\times \mathbb R\right)$, where
\[
S_{k+1}\left(X\times \mathbb R\right)=\left\{(x, y)\in X\times\mathbb R: p_k\left(x,y\right)+16\varepsilon_k \cdot \mu_{A_k }(x,y)=1\right\}
\]
and
\[
B_{k+1}\left(X\times \mathbb R\right)=\left\{(x, y)\in X\times\mathbb R: p_k\left(x,y\right)+16\varepsilon_k \cdot \mu_{A_k }(x,y)\leq 1\right\}.
\]
Let $p_{k+1}\left(x, y\right)= p_k\left(x, y\right)+16\varepsilon_k \cdot \mu_{ A_k }\left(x, y\right)$ for every $(x, y)\in X\times \mathbb R$. Then, by the definitions of $p_{k+1}$ and $p_{k+1}'$, it is easy to see that $p_{k+1}\left(\alpha_{k}e_{X,k},\alpha_{k}e_{k}\right)=1$. Hence we define
the real number $s_{k+1}\in (0,+\infty)$, where
\[
s_{k+1}=\sup\left\{ \sigma_{\partial f} \left((x,y),(e_{X},e)\right): ((x,y),(e_{X},e))\in U_{k+1}\times S_{k+1}\left(X\times \mathbb R\right) \right\}.
\]
Arguing as in the argument for (2.9) in Step 2, we obtain the following relation
\begin{eqnarray*}
~\quad~\quad\quad\quad~\quad s_{k+1}&=&\sup\left\{ \sigma_{\partial f} \left((x,y),(e_{X},e)\right): ((x,y),(e_{X},e))\right.
\\
&\in& \left.U_{k+1}\times B \left( \left(\alpha_{k}e_{X,k},\alpha_{k}e_{k}\right),25\varepsilon_{k} \right) \cap B_{k+1}\left(X\times \mathbb R\right) \right\}.~\quad~\quad~\quad\quad~(2.21)
\end{eqnarray*}
Since $X$ is a weak Asplund space, by Lemma 2.2 and Lemma 2.3, there exists a decreasing sequence of dense open cones $\{O_n^{k+1}\}_{n=1}^{\infty}$ contained in $G_{k+1}$
such that $p_{k+1}$ is G\^ateaux differentiable on $\cap_{n=1}^{\infty}O_n^{k+1}$ and $O_{n+1}^{k+1}\subset O_n^{k+1}$, where
\[
G_{k+1} = \left\{  \lambda \left (x, f_{k+1}\left(x\right)\right) \in X\times\mathbb R :   x\in T \left( \overline{\mathrm{co}}\left(S_{k+1}\left(X\times\mathbb R\right)\right) \right),~  \lambda\in (0,+\infty)     \right\}\quad
\]
\[
\quad \quad\quad\quad  \cup\left\{  \lambda  \left(x, g_{k+1}\left(x\right)\right)\in X\times\mathbb R :   x\in T \left( \overline{\mathrm{co}}\left(S_{k+1}(X\times\mathbb R )\right) \right),~  \lambda\in (0,+\infty)      \right\},
\]
\[
f_{k+1}\left(x\right)= \inf \left\{ r\in\mathbb R: (x,r) \in   \overline{\mathrm{co}}\left(S_{k+1}\left(X\times\mathbb R\right)\right)         \right\}
\]
and
\[
g_{k+1}\left(x\right)= \sup \left\{ r\in\mathbb R: (x,r) \in    \overline{\mathrm{co}}\left(S_{k+1}\left(X\times\mathbb R\right)\right)         \right\}.
\]
Pick a real number $\varepsilon_{k+1}= \varepsilon_k/128  $. Noting that  $\prod_{i=0}^{\infty}\left[1-(20\varepsilon_{1}/128^{i})\right]>3/4$,
we obtain that
$\prod_{i=1}^{k+1}(1+16\varepsilon_i)^{-1}>3/4$.
We have shown that
\[
s_{k+1}'= \sup \left\{  \sigma_{\partial f}\left((x,y),( \alpha_k e_{X,k},  \alpha_ke_{k})\right):  (x,y)\in U_{k+1}             \right\}.
\]
Moreover,
by Lemma 2.7 and the fact that $p_{k+1}\left(\alpha_k e_{X,k}, \alpha_k e_{k}\right)=1$, we conclude that $\left(\alpha_k e_{X,k}, \alpha_k e_{k}\right)$ is an extreme point of $B_{k+1}\left(X\times\mathbb R\right)$.
This implies that $\alpha_{k} e_{k}\notin( f_{k+1}(\alpha_{k}e_{X,k}), g_{k+1}(\alpha_{k}e_{X,k}))$.
Hence there exist a real number $r_{k+1}\in (0,$ $r_k/4)$ and
two points $(x_{k+1},y_{k+1})\in U_{k+1}$,
$(e_{X,k+1},e_{k+1})\in\cap_{n=1}^{\infty}O_{n}^{k+1}
$ with
\[
p_{k+1}\left(e_{X,k+1},e_{k+1}\right)=1~~\quad~~~~~\mathrm{and}~~~~~~~~\sigma_{\partial f}\left((x,y),(e_{X,k+1},e_{k+1})\right)\geq\left(1-\frac{1}{16}\varepsilon_{k+1}^{36}\right) s_{k+1}'
\]
for every $(x,y)\in B\left((x_{k+1},y_{k+1}),r_{k+1}\right)\subset U_{k+1}$
such that
\par (1) the mapping $(e_{X,k+1},e_{k+1})(\partial f)$ is single-valued at $(x_{k+1},y_{k+1})$;
\par (2) $p_{k+1} \big((\alpha_k e_{X,k}, \alpha_k e_{k})- (e_{X,k+1},e_{k+1})   \big)  \leq \eta_k/ 50$;
\par (3) $T\left(e_{X,k+1},e_{k+1}\right)\in \mathrm{int} T \left\{ (x,y)\in X\times\mathbb R: p_{k+1} \left(x,y\right) \leq 1    \right\}$.
\\
Define the set $V_{k+1}=\mathrm{int}B\left((x_{k+1},y_{k+1}),r_{k+1}\right)$. Then $V_{k+1}$ is open   and $\overline{V_{k+1}}\subset U_{k+1}$.
Since $O_{i}^{j}$
is an open cone and
$B\left(( e_{X,k},  e_{k}),256\eta_{k}\right)\subset( O_{k}^{1} \cap  O_{k}^{2} \cap \cdot\cdot\cdot  \cap O_{k}^{k}  )$,
by the fact that $\alpha_{k}\in [3/4,1]$ and the inequalities
\[
 \big\|(\alpha_k e_{X,k}, \alpha_k e_{k})- (e_{X,k+1},e_{k+1}) \big\|  \leq   p_{k+1}\big ((\alpha_k e_{X,k}, \alpha_k e_{k})- (e_{X,k+1},e_{k+1})  \big )  \leq \frac{1}{50}\eta_k,
\]
it is easy to see that $(e_{X,k+1},e_{k+1})\in B\left(\left(\alpha_k  e_{X,k}, \alpha_k e_{k}\right),25\eta_{k}\right)$. It follows that
\[
(e_{X,k+1},e_{k+1})\in B\big(\left(\alpha_k  e_{X,k}, \alpha_k e_{k}\right),25\eta_{k}\big)\subset  \left( O_{k}^{1} \cap  O_{k}^{2} \cap \cdot\cdot\cdot  \cap O_{k}^{k}  \right).~~~\eqno~~~~~~~~~~~~~~~~~~~~~~~~~~(2.22)
\]
Therefore, since $(e_{X,k+1},e_{k+1})\in \cap_{n=1}^{\infty}O_{n}^{k+1}
\subset X\times\mathbb R$, there exists a real number $\eta_{k+1}\in \left(0,\min \left\{\eta_k/128,\varepsilon_{k+1}^{64}/128\right\} \right)$ such that
\[
 B\big(\left(e_{X,k+1},e_{k+1}\right),256\eta_{k+1}\big)\subset \left(  O_{k+1}^{1} \cap  O_{k+1}^{2} \cap \cdot\cdot\cdot  \cap O_{k+1}^{k+1} \right).~~~\eqno~~~~~~~~~~~~~~~~~~~~~~~~~~(2.23)
\]
Noting that $\partial f\left(U_{1}\right) \subset B_{1}\left(X^{*}\times\mathbb R\right)$, by (2.21),
we obtain that $s_{k+1}'\geq s_{k+1} -25\varepsilon_k$. Moreover, since $ s_{k+1}' \leq 1$,
we obtain the following inequalities
\begin{eqnarray*}
\sigma_{\partial f}\big((x,y),(e_{X,k+1},e_{k+1})\big) &\geq & \left(1-\frac{1}{16}\varepsilon_{k+1}^{32}\right) s_{k+1}'
\\
&\geq& \left(1-\frac{1}{16}\varepsilon_{k+1}^{32}\right)\left( s_{k+1}-25\varepsilon_{k} \right)
\\
&= & \left [s_{k+1}-25\varepsilon_{k}-\frac{1}{16 }\varepsilon_{3}^{32}s_{k+1}+\frac{1}{16 }\varepsilon_{k+1}^{32}\cdot25\varepsilon_{k} \right]
\\
&\geq& s_{k+1}-25\varepsilon_{k}- \frac{1}{8 }\varepsilon_{k}
\end{eqnarray*}
for all $(x,y)\in \mathrm{int}B\left((x_{k+1},y_{k+1}),r_{k+1}\right)$. Since
$( e_{X,k+1}, e_{k+1})\partial f$ is single-valued  at the point $(x_{k+1},y_{k+1})\in X\times\mathbb R$, we obtain that the mapping
 $(x,y) \to \sigma_{\partial f}(\left(x,y\right),$ $\left(e_{X,k+1},e_{k+1}\right))$ is continuous at $\left(x_{k+1},y_{k+1}\right)\in X\times\mathbb R$.
Moreover, since $\partial f$ is norm-to-weak$^{*}$ upper-semicontinuous, $\left( e_{X,k+1}, e_{k+1}\right)\partial f$ is single-valued at $(x_{k+1},y_{k+1})$ and
the functional $(x,y) \to \sigma_{\partial f}((x,y),(e_{X,k+1},e_{k+1}))$ is continuous at $(x_{k+1},y_{k+1})$,
we may assume without loss of generality that
\begin{eqnarray*}
\big\langle (x^{*},y^{*}), ( e_{X,k+1}, e_{k+1}) \big\rangle &>& \left[\sigma_{\partial f}\left((x,y),(e_{X,k+1},e_{k+1})\right)- \frac{1}{128}\varepsilon_{k}^{64} \right]
\\
&\geq& s_{k+1}-25\varepsilon_{k}- \frac{1}{8}\varepsilon_{k}- \frac{1}{128}\varepsilon_{k}^{64}
\\
&\geq& s_{k+1}-25\varepsilon_{k}- \frac{1}{2}\varepsilon_{k}
\end{eqnarray*}
for every $(u,v)\in V_{k+1}$ and $(x^{*},y^{*})\in \partial f\left(u,v\right)$.
Noting that $\varepsilon_{k+1}=\varepsilon_{k}/128$, $\eta_{k}\in (0,\min \left\{\eta_{k-1}/128, \varepsilon_k^{64}/128\right\})$ and
$p_{k+1} \big((\alpha_k e_{X,k}, \alpha_k e_{k})- (e_{X,k+1},e_{k+1})  \big )  \leq \eta_k/ 50$, we obtain the following inequalities
\[
\big\langle (x^{*},y^{*}), ( e_{X,k+1}, e_{k+1}) \big\rangle\quad\quad\quad\quad\quad\quad\quad\quad\quad\quad\quad\quad\quad\quad\quad\quad\quad\quad
\]
\[
\quad\quad> ~  \left[  \sigma_{\partial f}\left((x,y),(e_{X,k+1},e_{k+1})\right)- \frac{1}{128}\varepsilon_{k}^{64} \right]\quad\quad\quad\quad\quad\quad\quad\quad\quad\quad\quad\quad\quad\quad\quad\quad\quad\quad\quad\quad\quad\quad\quad\quad\quad\quad\quad\quad\quad\quad
\]
\[
\quad\quad\geq ~  \left[\left(1-\frac{1}{16 }\varepsilon_{k+1}^{32}\right) s_{k+1}'- \frac{1}{128}\varepsilon_{k}^{64}\right]\quad\quad\quad\quad\quad\quad\quad\quad\quad\quad\quad\quad\quad\quad\quad\quad\quad\quad\quad\quad\quad\quad\quad\quad\quad\quad\quad\quad\quad\quad
\]
\[
\quad\quad\geq  ~  \sup \left\{  \sigma_{\partial f}\left(\left(x,y\right),\left( \alpha_ke_{X,k}, \alpha_ke_{k}\right)\right):  \left(x,y\right)\in V_{k+1}            \right\}~-~\frac{1}{16 }\varepsilon_{k+1}^{32}s_{k+1}'~-~ \frac{1}{128}\varepsilon_{k}^{64}\quad\quad\quad\quad\quad\quad\quad\quad\quad\quad\quad\quad\quad\quad\quad\quad\quad\quad\quad\quad\quad\quad\quad\quad\quad\quad\quad\quad\quad\quad
\]
\[
\quad\quad\geq  ~ \sup \left\{  \sigma_{\partial f}\left(\left(x,y\right),\left( \alpha_ke_{X,k}, \alpha_ke_{k}\right)\right):  \left(x,y\right)\in V_{k+1}            \right\}~-~\frac{1}{10 }\varepsilon_{k+1}^{32}s_{k+1}'\quad\quad\quad\quad\quad\quad\quad\quad\quad\quad\quad\quad\quad\quad\quad\quad\quad\quad\quad\quad\quad\quad\quad\quad\quad\quad\quad\quad\quad\quad
\]
\[
\quad\quad\geq  ~ \sup \left\{  \sigma_{\partial f}\left(\left(x,y\right),\left( e_{X,k+1}, e_{k+1}\right)\right):  \left(x,y\right)\in V_{k+1}            \right\}~-~ \frac{1}{128}\varepsilon_{k}^{64}~-~\frac{1}{10 }\varepsilon_{k+1}^{32}s_{k+1}'\quad\quad\quad\quad\quad\quad\quad\quad\quad\quad\quad\quad\quad\quad\quad\quad\quad\quad\quad\quad\quad\quad\quad\quad\quad\quad\quad\quad\quad\quad
\]
\[
\quad\quad\geq ~ \sup \left\{  \sigma_{\partial f}\left(\left(x,y\right),\left( e_{X,k+1}, e_{k+1}\right)\right):  \left(x,y\right)\in V_{k+1}            \right\}~-~ \frac{1}{8}\varepsilon_{k+1}^{32} \quad\quad\quad\quad\quad\quad\quad\quad\quad\quad\quad\quad\quad\quad\quad\quad\quad\quad\quad\quad\quad\quad\quad\quad\quad\quad\quad\quad\quad\quad
\]
for every $(u,v)\in V_3$ and $(x^{*},y^{*})\in \partial f\left(u,v\right)$.
Moreover,
Player $A$ may choose
any nonempty open subset $U_{k+2} \subset V_{k+1}$. Hence
we may assume that
\[
\sup \left\{  \|(x^{*},y^*)\|\in\mathbb R :   (x^{*},y^*)\in \partial f\left(U_{k+2}\right)      \right\}>0.
\]
This implies that the conclusion of Step 2 holds for every natural
number $k\in \mathbb N$.
Moreover, by the fact that $r_k\in (0,r_{k-1}/4)$, we obtain the following inequalities
\[
0<r_k\leq \frac{1}{4}r_{k-1}\leq \frac{1}{4^{2}}r_{k-2}\leq\cdot\cdot\cdot\leq\frac{1}{4^{k-1}}r_1\leq\frac{1}{4^{k-1}}~\quad~~~~~\mathrm{for}~~~~~~~\mathrm{every}~~~~~~\quad~k\in \mathbb N.
\]
Since $(x_{k+1},y_{k+1})\in \mathrm{int}B\left((x_{k},y_{k}),r_{k}\right)$, we obtain that $\|(x_{k+1},y_{k+1})-(x_{k},y_{k})\|\leq r_k$ for every $k\in \mathbb N$.
Moreover, since $0<r_k\leq {1}/{4^{k-1}}$, we obtain that
$\{(x_k,y_k)\}_{k=1}^{\infty}$ is a Cauchy sequence. Hence there exists a point $\left(x_0,y_0\right)$ in $X\times  \mathbb R$ such that
\[
\mathop {\lim }\limits_{k \to \infty}\big\|(x_k,y_k)-(x_0,y_0)\big\|=0.
\]
Moreover, since $\overline{V_k}\subset U_k\subset V_{k-1}$ and $0<r_k\leq {1}/{4^{k-1}}$, by the nested closed ball theorem,
we conclude that $\cap_{k=1}^{\infty}V_k = \{(x_0,y_0)\}$.

\textbf{Step 4.}
We  first prove that the sequence $ \{(e_{X,k},e_{k})\}_{k=1}^\infty $ is a Cauchy sequence. In fact,
from the proof of Step 3, we have
$p_{k+1} \left((\alpha_k e_{X,k}, \alpha_k e_{k})- (e_{X,k+1},e_{k+1})   \right)  \leq \eta_{k}/50 $. Therefore, by the definition of $p_k$, we obtain the following inequalities
\[
\big\|(\alpha_ke_{X,k}, \alpha_k e_{k})- (e_{X,k+1},e_{k+1})  \big \| \leq p_{k+1} \left((\alpha_k e_{X,k}, \alpha_k e_{k})- (e_{X,k+1},e_{k+1})   \right)   \leq \frac{1}{50}\eta_{k}
\]
for every $k\in N$. Noting that $\eta_{k+1}\in\big(0,\min\{\eta_k/128,\varepsilon_k^{64}/128\}\big)$, from the preceding inequalities and (2.21)-(2.22), we obtain the following inequalities
\[
\left\| (e_{X,k+p},e_{k+p})  -  \left(\prod_{i=0}^{p-1}\alpha_{k+i}\right)\left( e_{X,k}, e_{k}\right)      \right\|~ \quad\quad\quad\quad\quad\quad\quad\quad\quad\quad\quad\quad\quad
\]
\[
\quad\leq~ \big\|   (e_{X,k+p},e_{k+p}) -\alpha_{k+p-1} (e_{X,k+p-1},e_{k+p-1})       \big\|\quad\quad\quad~\quad\quad\quad\quad\quad\quad\quad\quad\quad\quad\quad\quad\quad\quad
\]
\[
\quad\quad+ \big\| \alpha_{k+p-1}  (e_{X,k+p-1},e_{k+p-1}) - \alpha_{k+p-2}\alpha_{k+p-1}  (e_{X,k+p-2},e_{k+p-2})       \big\|\quad\quad\quad~\quad\quad\quad\quad\quad\quad\quad\quad\quad\quad\quad\quad\quad\quad\quad\quad\quad\quad\quad\quad\quad
\]
\[
\quad\quad\quad+~\cdot\cdot\cdot~+\left\| \left(\prod_{i=1}^{p-1}\alpha_{k+i}\right)(e_{X,k+1},e_{k+1})  -\left(\prod_{i=1}^{p-1}\alpha_{k+i}\right)\alpha_{k}\left( e_{X,k}, e_{k}\right)        \right\| \quad\quad\quad~\quad\quad\quad\quad\quad\quad\quad\quad\quad\quad\quad\quad\quad\quad\quad\quad\quad\quad\quad\quad\quad
\]
\[
\quad\leq~ \big\|   (e_{X,k+p},e_{k+p}) -  \alpha_{k+p-1} (e_{X,k+p-1},e_{k+p-1})      \big\|\quad\quad\quad~\quad\quad\quad\quad\quad\quad\quad\quad\quad\quad\quad\quad\quad\quad\quad\quad\quad\quad\quad\quad\quad\quad\quad
\]
\[
\quad\quad+ \big\|   (e_{X,k+p-1},e_{k+p-1}) -\alpha_{k+p-2} (e_{X,k+p-2},e_{k+p-2})       \big\|\quad\quad\quad\quad\quad\quad\quad\quad\quad\quad\quad\quad\quad\quad\quad\quad\quad\quad\quad\quad\quad\quad\quad\quad\quad\quad\quad\quad\quad\quad\quad\quad
\]
\[
\quad\quad\quad+~\cdot\cdot\cdot~+\big\| (e_{X,k+1},e_{k+1}) -  \alpha_{k}\left( e_{X,k}, e_{k}\right)       \big\|
\leq \frac{1}{50} \left( \sum_{j=0}^{p-1} \eta_{k+j} \right)~<~ \frac{1}{2}\eta_{k}\quad\quad\quad~(2.25)\quad\quad\quad\quad\quad \quad~\quad\quad\quad\quad\quad\quad\quad\quad\quad\quad\quad\quad\quad\quad\quad\quad\quad \quad\quad\quad\quad\quad\quad\quad\quad\quad\quad\quad\quad\quad\quad~\quad\quad\quad\quad
\]
for every $k\in\mathbb N$ and $p\in\mathbb N$.
Moreover, from the proof of Step 3, we know that $\alpha_k\geq $ $ (1+16 \varepsilon_{k})^{-1}$ for every $k\in\mathbb N$.
Since $p_k(e_{X,k},e_k)=1 $, we have $\|(e_{X,k},e_k)\|\le 1$.
Noting that $\eta_k\in\bigl(0,\min\{\eta_{k-1}/128,\varepsilon_{k-1}^{64}/128\}\bigr)$ and $\|(\alpha_k e_{X,k},\alpha_k e_k)-(e_{X,k+1},e_{k+1})\|<\eta_k/50$, we obtain the following inequalities
\begin{eqnarray*}
0 &\leq & \big\| ( e_{X,k},e_{k})-(e_{X,k+1},e_{k+1})          \big\|
\\
&\leq& \big\| (\alpha_k e_{X,k},\alpha_k e_{k})-(e_{X,k},e_{k})         \big\| + \big\| (\alpha_k e_{X,k},\alpha_ke_{k})-(e_{X,k+1},e_{k+1})          \big\|
\\
&\leq& \left|  \alpha_k-1 \right|\cdot \big\|(e_{X,k},e_{k})\big\|+ \big\| (\alpha_k e_{X,k},\alpha_ke_{k})-(e_{X,k+1},e_{k+1})          \big\|
\\
&\leq & \left|  \alpha_k-1 \right|\cdot \big\|(e_{X,k},e_{k})\big\|+ \frac{1}{50}\eta_{k}
\\
&\leq&   \left|   \frac{1}{1+16\varepsilon_k}-1 \right| +\frac{1}{50}\eta_{k}
\\
& \leq&  25 \varepsilon_{k}.
\end{eqnarray*}
Since $\varepsilon_{k+1}= \varepsilon_k/128 $ and $0 \leq  \big\| ( e_{X,k},e_{k})-(e_{X,k+1},e_{k+1})          \big\| $ $\leq  25 \varepsilon_{k}$, we obtain that
\begin{eqnarray*}
\big\| (e_{X,k},e_{k})-(e_{X,k+p},e_{k+p})          \big\|
 & = & \left\|\sum_{j=1}^{p} \left[\left(e_{X,k+j-1},e_{k+j-1}\right)-\left(e_{X,k+j},e_{k+j}\right) \right] \right\|
\\
&\leq& \left(\sum_{j=1}^{p}\big\| \left(e_{X,k+j-1},e_{k+j-1}\right)-\left(e_{X,k+j},e_{k+j}\right)  \big\| \right)
\\
&\leq& \left(\sum_{j=1}^{\infty}\big\| \left(e_{X,k+j-1},e_{k+j-1}\right)-\left(e_{X,k+j},e_{k+j}\right)  \big\| \right)
\\
& \leq &\sum_{j=1}^{\infty} \left(25   \varepsilon_{k+j-1} \right)= 25    \sum_{j=1}^{\infty}\frac{1}{128^{j-1}}\varepsilon_{k}
\leq 40\varepsilon_{k}
\end{eqnarray*}
whenever $k\in\mathbb N$ and $p\in\mathbb N$. Moreover, since $\varepsilon_{k} \to 0$,
we obtain that
$ \{(e_{X,k},e_{k})\}_{k=1}^\infty $ is a Cauchy sequence. Hence there exists a point $ (e_{X,0},e_{0})\in X\times \mathbb R$ such that
$\|(e_{X,k},e_{k}) - (e_{X,0},e_{0})\|  \to 0$ as $k\to \infty$. Then, by (2.25) and the triangle inequality, we have the following inequalities
\[
\left\| (e_{X,0},e_{0})    - \left(\prod_{i=0}^{\infty}\alpha_{k+i}\right)\left( e_{X,k}, e_{k}\right)      \right\|\quad \quad \quad \quad \quad \quad \quad \quad \quad \quad \quad \quad \quad \quad \quad \quad \quad
\]
\[
\leq\mathop {\lim }\limits_{p \to \infty }\big\| (e_{X,0},e_{0})    -(e_{X,k+p},e_{k+p}) \big\|+  \mathop {\lim }\limits_{p \to \infty }\left\| (e_{X,k+p},e_{k+p})   -\left(\prod_{i=0}^{p-1}\alpha_{k+i}\right)\left( e_{X,k}, e_{k}\right)       \right\|
\]
\[
 \quad +\mathop {\lim }\limits_{p \to \infty }\left\|\left(\prod_{i=0}^{p-1}\alpha_{k+i}\right)\left( e_{X,k}, e_{k}\right)  -\left(\prod_{i=0}^{\infty}\alpha_{k+i}\right)\left( e_{X,k}, e_{k}\right)      \right\|\quad \quad \quad \quad \quad \quad \quad \quad \quad \quad \quad \quad \quad \quad \quad
\]
\[
 \leq 0~+~\frac{1}{2}\eta_{k}~+~0~=~\frac{1}{2}\eta_{k}~~\quad~~~~\mathrm{for}~~~\mathrm{every}~~~~~~~~~~~~~\quad~~k\in N.\quad \quad \quad \quad \quad \quad \quad \quad \quad \quad \quad \quad \quad \quad \quad \quad \quad \quad \quad\quad \quad \quad \quad \quad \quad \quad \quad \quad \quad \quad \quad \quad \quad \quad
\]
Since $\prod_{i=0}^{\infty}\left[1-\left(20\varepsilon_{1}/128^{i}\right)\right]>3/4$ and $1\geq \alpha_k\geq \left(1+16\varepsilon_k\right)^{-1}$, it is easy to see that $\prod_{i=1}^{\infty}\alpha_{k+i}\in [3/4,1]$.
Moreover, since each
$O_{i}^{j}$ is an open cone, by  (2.23), we obtain the following formula
\[
(e_{X,0},e_{0})\in  B\left (\left(\prod_{i=1}^{\infty}\alpha_{k+i}\right)( e_{X,k+1}, e_{k+1}), \frac{1}{2}\eta_{k+1}\right)\subset \left(  O_{k+1}^{1} \cap  O_{k+1}^{2} \cap \cdot\cdot\cdot  \cap O_{k+1}^{k+1} \right) \quad
\]
for every $k\in\mathbb N$. Therefore, by the above relation, we conclude that
\[
(e_{X,0},e_{0})\in  \mathop {\cap }\limits_{k=1}^{\infty} B\left (\left(\prod_{i=1}^{\infty}\alpha_{k+i}\right)( e_{X,k},e_{k}),\frac{1}{2}\eta_{k+1}\right)
\subset \mathop {\cap}\limits_{k=1}^{\infty}\left(\mathop {\cap}\limits_{j=1}^{k+1}O_{k+1}^{j}\right).
\]
Therefore, by the definition of $p_{k}$, we obtain that for every $k\in \mathbb N$, $p_{k}$ is
G$\mathrm{\hat{a}}$teaux differentiable at $(e_{X,0},e_{0})\in X\times  \mathbb R$. Hence we
define the functional
\[
p_0\left(x,y\right)=p_2\left(x,y\right)+  16\left(\sum_{k=2}^{\infty} \varepsilon_k\cdot \mu_{ A_k }\left(x,y\right)\right)~~\quad~~\mathrm{for}~~~~~~~\mathrm{every}~~~~~~~~\quad~~(x,y)\in X\times\mathbb R.
\]
We claim that $p_0$ is a Minkowski functional on $X\times\mathbb R$. In fact, by the definition of $A_k$,
there exists a real number $m_{x,y}\in (0,+\infty)$ such that $\mu_{ A_k }(x,y)<m_{x,y}$ for every $k\in\mathbb N$.
Noting that $\varepsilon_{1}\in \left(0,1/512^{6}\right)$ and
$\varepsilon_{k}=\varepsilon_{k-1}/128$, by the definition of $A_k$, we obtain that  $\sum_{k=2}^{\infty} 16\varepsilon_k\cdot \mu_{ A_k }\left(x,y\right)<+\infty$ for every $(x,y)\in X\times\mathbb R$.
Hence, for every $(x,y)\in X\times\mathbb R$,
we have the following relation
\[
p_0\left(x,y\right)=p_2\left(x,y\right)+  16\left(\sum_{k=2}^{\infty} \varepsilon_k\cdot \mu_{ A_k }\left(x,y\right)\right)<+\infty~~\quad~~\mathrm{for}~~~~~~~\mathrm{all}~~~~~~~~\quad~~(x,y)\in X\times\mathbb R.
\]
Noting that $p_0\left(\lambda x, \lambda y\right)=\lambda \cdot p_0\left(x,y\right)$ for all $\lambda\geq 0$ and $(x,y)\in X\times\mathbb R$,
we obtain that
$p_0$ is a Minkowski functional. Moreover, it is easy to see that
\[
(0,0)\in \mathrm{int} \left\{(x, y)\in X\times\mathbb R   :p_0\left(x,y\right)\leq1\right\}.
\]
Hence we obtain that $p_0$ is a continuous Minkowski functional on $X\times\mathbb R$. Define the set
\[
S\left(X\times\mathbb R\right)=\left\{(x, y): p_0\left(x,y\right)=p_2\left(x,y\right)+  16\left(\sum_{k=2}^{\infty} \varepsilon_k\cdot \mu_{ A_k }\left(x,y\right)\right)=1\right\}.
\]
We next prove that $p_{0}$ is
G$\mathrm{\hat{a}}$teaux differentiable at the point $(e_{X,0},e_{0})\in X\times\mathbb R$. Pick a point $(u,v)\in S(X\times\mathbb R)$. Then, for any $\varepsilon>0$, there exists a natural number $k_0\in\mathbb N$ such that $4\sum_{k=k_0 +1}^{\infty} 512^{6}\varepsilon_k<\varepsilon /8$.
Since  $p_{k_0}$ is
G$\mathrm{\hat{a}}$teaux differentiable at $(e_{X,0},e_{0})\in X\times\mathbb R$, there exists a real number $t_0\in (0,1)$ such that
\[
\frac{1}{t}\big[ p_{k_0} \left( (e_{X,0},e_{0})+t \left(u,v\right) \right) + p_{k_0} \left( (e_{X,0},e_{0})-t \left(u,v\right) \right)- 2p_{k_0} \left( e_{X,0},e_{0} \right) \big]<\frac{1}{8}\varepsilon
\]
whenever $t\in (0,t_0)$. Since $f_{k}\left(t\right)= \mu_{ A_k }( (e_{X,0},e_{0})+t (u,v) )$ is a  continuous convex function on $\mathbb R$, we conclude that
\[
\frac{1}{t} \left[ \mu_{ A_k }( (e_{X,0},e_{0})+t (u,v) )-\mu_{ A_k } ( e_{X,0},e_{0} )\right]\leq \mu_{ A_k }( (e_{X,0},e_{0})+ (u,v) )-\mu_{ A_k } ( e_{X,0},e_{0} )
\]
for each $t\in (0,t_0)$ and $k\in\mathbb N$. Moreover, since $g_{k}(t)= \mu_{ A_k }( (e_{X,0},e_{0})+t (-u,-v) )$ is  a continuous convex function on $\mathbb R$, we conclude that
\[
\frac{1}{t} \left[ \mu_{ A_k }( (e_{X,0},e_{0})-t (u,v) )-\mu_{ A_k } ( e_{X,0},e_{0} )\right]\leq \mu_{ A_k }( (e_{X,0},e_{0})- (u,v) )-\mu_{ A_k }( e_{X,0},e_{0} )
\]
for every $t\in (0,t_0)$ and $k\in\mathbb N$. Therefore, by the above two inequalities, we have the following inequality
\begin{eqnarray*}
&&\frac{1}{t}\left[ \mu_{ A_k } \left( (e_{X,0},e_{0})+t (u,v) \right) +\mu_{ A_k } \left( (e_{X,0},e_{0})-t (u,v) \right)- 2\mu_{ A_k } \left( e_{X,0},e_{0} \right) \right]
\\
&\leq& \mu_{ A_k }\left( (e_{X,0},e_{0})+ (u,v) \right) + \mu_{ A_k } \left( (e_{X,0},e_{0})- (u,v) \right)- 2\mu_{ A_k } \left( e_{X,0},e_{0} \right)
\end{eqnarray*}
whenever $t\in (0,t_0)$ and $k\in\mathbb N$. Moreover, by the definition of $A_k$, it is easy to see that
if $p_1\left(e_X,e\right)\leq 2$ then
\[
\mu_{ A_k } \left(e_X,e\right)\leq (512)^{4}~\quad~~~~\mathrm{for}~~~\mathrm{every}~~~~~~~~~~~~~\quad~k\in\mathbb N.
\]
Since $ p_{0}\left(e_{X,0},e_{0}\right)=1$, by the definition of $p_1$, we have $ p_{1}\left(e_{X,0},e_{0}\right)\leq 1$. Moreover,
since $ p_{1}\left(e_{X,0},e_{0}\right)\leq 1$ and $p_{1}\left( u, v\right)\leq2 $, we obtain that
\[
\mu_{ A_k } ( (e_{X,0},e_{0})+ (u,v) ) \leq \mu_{ A_k } (e_{X,0},e_{0}) + \mu_{ A_k } (u,v) \leq (512)^{4} + (512)^{4}\leq (512)^{5}
\]
for every $k\in\mathbb N$. Noting that $p_{1}\left(- u,- v\right)\leq2 $, by the definition of $A_k$, we obtain that
\[
\mu_{ A_k } ( (e_{X,0},e_{0})-(u,v) ) \leq \mu_{ A_k } (e_{X,0},e_{0}) + \mu_{ A_k } (-u,-v) \leq (512)^{4} + (512)^{4} \leq (512)^{5}
\]
for every $k\in\mathbb N$. Since $p_0$ is a Minkowski functional,
by the above inequalities and the inequality $4\sum_{k=k_0 +1}^{\infty} (512)^{6}\varepsilon_k$ $<\varepsilon /8$,
we obtain that
if $t\in (0,t_0)$, then
\[
\frac{1}{t}\big[ p_{0} \left( (e_{X,0},e_{0})+t (u,v) \right) + p_{0} \left( (e_{X,0},e_{0})-t (u,v) \right)- 2p_{0} \left( e_{X,0},e_{0} \right) \big]\quad \quad \quad\quad
\]
\[
=~\frac{1}{t}\big[ p_{k_0} \left( (e_{X,0},e_{0})+t (u,v) \right) + p_{k_0} \left( (e_{X,0},e_{0})-t (u,v) \right)- 2p_{k_0} \left( e_{X,0},e_{0} \right) \big]\quad \quad \quad \quad \quad \quad \quad \quad \quad \quad \quad \quad \quad \quad \quad \quad \quad \quad \quad \quad \quad \quad \quad \quad
\]
\[
+16\sum_{k=k_0 +1}^{\infty} \frac{\varepsilon_k}{t}\big[  \mu_{ A_k } ( (e_{X,0},e_{0})+t (u,v) ) +  \mu_{ A_k }( (e_{X,0},e_{0})-t (u,v) )- 2 \mu_{ A_k } ( e_{X,0},e_{0} ) \big]
\]
\[
\leq~ \frac{1}{t}\big[ p_{k_0} \left( (e_{X,0},e_{0})+t (u,v) \right) + p_{k_0} \left( (e_{X,0},e_{0})-t \left(u,v\right) \right)- 2p_{k_0} \left( e_{X,0},e_{0} \right) \big]\quad \quad \quad \quad \quad \quad \quad \quad \quad \quad \quad \quad \quad \quad \quad \quad \quad \quad \quad \quad \quad \quad \quad \quad
\]
\[
+16 \left[\sum_{k=k_0 +1}^{\infty} \varepsilon_k \left[ \mu_{ A_k } ( (e_{X,0},e_{0})+ (u,v) ) +  \mu_{ A_k } ( (e_{X,0},e_{0})- (u,v) )- 2 \mu_{ A_k }( e_{X,0},e_{0} ) \right]  \right] \quad \quad \quad \quad \quad \quad \quad \quad \quad \quad \quad \quad \quad \quad \quad \quad \quad \quad \quad \quad \quad \quad \quad \quad
\]
\[
 \leq ~ \frac{1}{8}\varepsilon  ~+~ 16 \left[\sum_{k=k_0 +1}^{\infty} \varepsilon_k \left[ (512)^{5} +  (512)^{5} + 2 (512)^{5} \right]  \right]\quad \quad \quad \quad \quad \quad \quad \quad \quad \quad \quad \quad \quad \quad \quad \quad \quad \quad \quad \quad \quad \quad \quad \quad
\]
\[
 \leq~ \frac{1}{8}\varepsilon  ~+~ 2\left[\sum_{k=k_0 +1}^{\infty} (512)^{6}\varepsilon_k \right]< \varepsilon.\quad \quad \quad \quad \quad \quad \quad \quad \quad \quad \quad \quad \quad \quad \quad \quad \quad \quad \quad \quad \quad \quad \quad \quad\quad \quad \quad \quad \quad \quad \quad
\]
Since $p_0$ is a  Minkowski functional, from the preceding inequalities, we have
\[
\mathop {\lim }\limits_{t\to 0 }\frac{1}{t}\big[ p_{0} \left( (e_{X,0},e_{0})+t \left(u,v\right) \right) + p_{0} \left( (e_{X,0},e_{0})-t \left(u,v\right) \right)- 2p_{0} \left( e_{X,0},e_{0} \right) \big]=0.
\]
Hence we obtain that $p_{0}$ is
G$\mathrm{\hat{a}}$teaux differentiable at $(e_{X,0},e_{0})\in X\times\mathbb R$.

\textbf{Step 5.} Since $\{s_k\}_{k=1}^{\infty}$ is a bounded decreasing sequence, we conclude that $\{s_k\}_{k=1}^{\infty}$ is a Cauchy sequence.
Let $s_k \to s_0$ as $k\to \infty$. Then, since $s_k\geq 0$, it is easy to see that $s_0\geq 0$.

\par We claim that $p_0 \left( e_{X,0}, e_{0}\right)  = 1  $.
In fact,
by the fact that $ p_{k+1}(e_{X,k+1},e_{k+1})=1$ and $ p_{k}(x,y)\leq  p_{0}(x,y)$ for every $(x,y)\in X\times\mathbb R$,
we obtain that $ p_{0}\left(e_{X,k+1},e_{k+1}\right)\geq 1$. Moreover, since
$\left\|\left(e_{X,k+1},e_{k+1}\right)- \left( e_{X,0}, e_{0}\right)\right\| \to 0$, we obtain that $ p_{0}\left(e_{X,0},e_{0}\right)\geq 1$.
On the other hand, by the fact that $ p_{k+1}\left(e_{X,k+1},e_{k+1}\right)=1$ and $  p_{k}\left(x,y\right)\leq  p_{i}\left(x,y\right)$ for every $(x,y)\in X\times\mathbb R$ and $i\geq k+1$, we conclude that
\[
( e_{X,j}, e_{j}) \in \left\{    (x,y)\in X\times\mathbb R : p_{k+1} \left(x,y\right) \leq 1          \right\}~~\quad~~\mathrm{for}~~~~~~~\mathrm{every}~~~~~~~~\quad~~j \geq k+1.
\]
Since $\bigl\{(x,y)\in X\times\mathbb R:p_{k+1}(x,y)\le 1\bigr\}$ is a bounded closed subset of $X\times\mathbb R$
and $\|(e_{X,k+1},e_{k+1})-(e_{X,0},e_0)\|\to 0$, we conclude that
\[
(e_{X,0},e_{0}) \in \left\{    (x,y) :  p_{k+1} \left(x,y \right) \leq 1          \right\}~~\quad~~\mathrm{for}~~~~~~~\mathrm{every}~~~~~~~~\quad~k\in\mathbb N.
\]
It follows that
$p_{k+1} (e_{X,0},e_{0})  \leq 1  $ for every $k\in\mathbb N$. Since $p_{k}\to p_0$ and $p_{k+1} (e_{X,0},e_{0}) $ $ \leq 1  $, we obtain that $p_0 \left(e_{X,0},e_{0}\right)  \leq 1  $.
Moreover, since  $ p_{0}\left(e_{X,0},e_{0}\right)\geq 1$,
we obtain that $p_0 \left(e_{X,0},e_{0}\right)  = 1  $.
Pick a functional $(x_0^{*},y_0^{*}) \in\partial f\left(x_0,y_0\right) $.
We first prove that
\[
s_0=\mathop {\lim }\limits_{k \to \infty} s_k = \langle(x_0^{*},y_0^{*}), ( e_{X,0}, e_{0}) \rangle>0 .
\]
holds. In fact, since $s_k \to s_0$, by the definitions of $p_k$ and $s_k$, we obtain that $s_0>0$. Moreover, from the argument for Step 3, we obtain the following inequality
\[
\left\langle (x^{*},y^{*}), (e_{X,k+1},e_{k+1})   \right\rangle >s_{k+1} -\frac{1}{2}\varepsilon_k-25\varepsilon_k
\]
for every $k\geq 2$, $(u,v)\in U_{k+1}$ and $(x^{*},y^{*})\in \partial f\left(u,v\right)$. Therefore, since $(x_0^{*},y_0^{*})\in \partial f\left(x_0,y_0\right) $ and $(x_0,y_0) \in U_{k+1}$, we obtain the following inequalities
\[
\left\langle (x_{0}^{*},y_{0}^{*}), (e_{X,k+1},e_{k+1}) \right\rangle >s_{k+1} -\frac{1}{2}\varepsilon_k-25\varepsilon_k~~\quad~~~~~\mathrm{for}~~~~~~~~~\mathrm{every}~~~~~~~\quad~~~k\geq 2.      ~\eqno~~~~~~~~~~~~~~~~~~~~~~~~~~(2.26)
\]
Since $\big \|(e_{X,k+1},e_{k+1})- ( e_{X,0},e_{0}) \big\| \to 0$, $s_k \to s_0$ and $\varepsilon_{k}=\varepsilon_{k-1}/128$,  by (2.26),
we obtain the following inequalities
\[
\langle(x_0^{*},y_0^{*}), (e_{X,0}, e_{0}) \rangle= \mathop {\lim}\limits_{k \to \infty}\left\langle (x_0^{*},y_0^{*}), (e_{X,k+1},e_{k+1})\right\rangle \geq \mathop {\lim}\limits_{k \to \infty}s_k = s_0.~~\eqno~~~~~~~~~~~~~~~~~~~~~~~~~~(2.27)
\]
On the other hand, since $(x_{0},y_{0})\in U_{k} $ and $(x_0^{*},y_0^{*})\in\partial f\left(x_0,y_0\right)$, by the definition of $s_k$, we conclude that
$\left\langle (x_0^{*},y_0^{*}), (e_{X,k+1},e_{k+1}) \right\rangle\leq s_{k-1}$. Therefore, since $s_k \to s_0$ and $\langle (x_0^{*},y_0^{*}), (e_{X,k+1},e_{k+1}) \rangle\leq s_{k-1}$, we obtain the following inequalities
\[
\langle(x_0^{*},y_0^{*}), (e_{X,0},e_{0}) \rangle= \mathop {\lim }\limits_{k \to \infty}\big\langle (x_0^{*},y_0^{*}), (e_{X,k+1},e_{k+1})  \big\rangle \leq \mathop {\lim}\limits_{k \to \infty}s_{k-1} = s_0.
\]
Therefore, by (2.27),
we obtain that $s_0=\langle(x_0^{*},y_0^{*}),(e_{X,0},e_0)\rangle$.
Since $(x_0^*,y_0^*)$ is an arbitrary point in the set $\partial f(x_0,y_0)$, we obtain that $(e_{X,0},e_0)\partial f$ is a single-valued mapping at $(x_0,y_0)\in X\times\mathbb R$.

\par Secondly, we prove that $s_0= \sup \left\{\left\langle(x_0^{*},y_0^{*}), (e_{X},e) \right\rangle  : p_0  (e_{X},e)\leq 1                 \right \}$ holds. In
fact,
since $p_0 \left( e_{X,0}, e_{0}\right)  = 1  $ and $s_0=\langle(x_0^{*},y_0^{*}), ( e_{X,0}, e_{0}) \rangle$,
we conclude that
\[
s_0 \leq \sup \left\{\langle(x_0^{*},y_0^{*}), (e_{X},e) \rangle  : p_0 \left (e_{X},e\right)\leq 1                   \right \}.
\]
Suppose that $s_0 \neq \sup \left\{\left\langle(x_0^{*},y_0^{*}), (e_{X},e) \right\rangle  : p_0 \left (e_{X},e\right)\leq 1                 \right   \}$. Then we get that
$s_0 < $ $\sup \{\left\langle(x_0^{*},y_0^{*}), (e_{X},e) \right\rangle  : p_0 \left (e_{X},e\right)\leq 1                   \}$. Hence there exists a real number $r\in (0,1)$ such that
\[
s_0+2r < \sup   \left \{\left\langle(x_0^{*},y_0^{*}), (e_{X},e) \right\rangle  : p_0  \left(e_{X},e\right)\leq 1                 \right   \}.
\]
Moreover, since $s_k \to s_0$, we may assume without loss of generality that
\[
s_k +2r< \sup \left\{\langle(x_0^{*},y_0^{*}), (e_{X},e) \rangle  : p_0 \left (e_{X},e\right)\leq 1                   \right \}
\]
for every $k\in N$. We know that $B_{k}\left(X\times\mathbb R\right)= \{(x,y) \in X\times\mathbb R: p_k\left(x,y\right)\leq 1\}$. Then, since $p_k \leq p_0$, we conclude that
\[
B_{k}\left(X\times R\right)\supset \{ (e_{X},e)\in X\times\mathbb R:   p_0 \left (e_{X},e\right)\leq 1         \}.
\]
Therefore,
by the definition of $s_k$ and the fact that $(x_{0},y_{0})\in U_{k} $, we obtain the following inequalities
\[
s_{k}~=~\sup\left\{ \sigma_{\partial f}\left((x,y),(e_{X},e)): ((x,y),(e_{X},e)\right)\in U_{k}\times S_{k}\left(X\times\mathbb R\right) \right\}\quad
\]
\[
\quad\quad \quad \quad =~\sup\left\{ \sigma_{\partial f} \left((x,y),(e_{X},e)\right): ((x,y),(e_{X},e))\in U_{k}\times B_{k}\left(X\times\mathbb R\right) \right\}\quad \quad \quad \quad \quad \quad\quad\quad \quad\quad \quad \quad \quad \quad \quad\quad\quad \quad\quad \quad \quad \quad \quad \quad\quad\quad \quad
\]
\[
\quad\quad \quad \quad\geq~ \sup \left\{\left\langle(x_0^{*},y_0^{*}), (e_{X},e) \right\rangle  : p_0  \left(e_{X},e\right)\leq 1                   \right \}\quad\quad \quad \quad\quad\quad \quad \quad\quad\quad \quad \quad\quad \quad \quad\quad\quad\quad \quad \quad\quad\quad \quad \quad\quad\quad \quad \quad\quad\quad \quad \quad\quad\quad \quad \quad
 \]
\[
\quad\quad \quad \quad >~ s_k ~+~2r,\quad \quad \quad \quad \quad \quad\quad\quad \quad\quad \quad \quad \quad\quad\quad\quad \quad\quad \quad \quad  \quad \quad\quad\quad \quad\quad \quad \quad \quad \quad \quad\quad\quad \quad
\]
a contradiction. Hence we obtain that $s_0= \sup \left\{\langle(x_0^{*},y_0^{*}), (e_{X},e) \rangle  : p_0  (e_{X},e)\leq 1                   \right \}$.
Since the mapping $\left(e_{X,0},e_{0}\right)\partial f$ is a single-valued mapping at $(x_{0},y_{0})\in X\times\mathbb R$, we obtain that
\[
s_0=       \langle(x_0^{*},y_0^{*}), ( e_{X,0}, e_{0}) \rangle       \geq \sup \left\{    \sigma_{\partial f}\left ((x_0,y_0),(e_{X},e)\right):      p_0  \left(e_{X},e\right)\leq 1        \right \}.
\]
Therefore, by Lemma 2.6,
we obtain that $\partial f\left(x_0,y_0\right)\subset s_0 \partial p_0 \left (e_{X,0},e_{0}\right)$. Moreover,
since $p_{0}$ is
G$\mathrm{\hat{a}}$teaux differentiable at $(e_{X,0},e_{0})\in X\times\mathbb R$, we obtain that $\partial f\left(x_0,y_0\right)$ is a singleton.
Hence we obtain that
\[
\mathop {\cap}\limits_{k=1}^{\infty}V_k = \{(x_0,y_0)\}\subset G=\left\{ (x,y)\in X\times\mathbb R : ~\mathrm{The}~~~~\mathrm{set}~~\partial f\left(x,y\right)~~\mathrm{is}~~~~\mathrm{a}~~~~\mathrm{singleton}    \right\}.
\]
Therefore, by Lemma 1.10, we obtain that $G$ is a dense $G_{\delta}$ subset of $X\times\mathbb R$.
Hence we get that $X\times\mathbb R$ is a weak Asplund space, which completes the proof.
\end{proof}

\begin{theorem}
Suppose that $X$ is a weak Asplund space and $Y$ is a finite dimensional space. Then the space $X\times Y$ is a weak Asplund space.

\end{theorem}

\begin{proof}

By Theorem 2.1,  it follows that Theorem 2.8 holds, which completes the proof.
\end{proof}

\section{Some problems}

This section proposes a set of research questions for subsequent discussion.

\begin{problem}
Let $X$ be a weak Asplund  space. Must $X\times \ell^{2}$ be a weak Asplund  space?

\end{problem}

\begin{problem}
Let $X$ be a weak Asplund  space. Must $X\times \ell^{p}$ be a weak Asplund  space?

\end{problem}

\begin{problem}
Let $X$ be a weak Asplund  space and $Y$ be a separable space. Must $X\times Y$ be a weak Asplund  space?

\end{problem}

\begin{problem}
Let $X$ be a weak Asplund  space and $Y$ be a reflexive space. Must $X\times Y$ be a weak Asplund  space?

\end{problem}

\textbf{Acknowledgement.} This research is supported by "China Natural Science Fund under grant 12271121".

\textbf{Declaration.} This thesis represents the author's original work, and no AI tools were used for text generation throughout the writing process.

\bibliographystyle{amsplain}

\end{document}